\documentclass[a4paper, 10pt, american]{amsart}
\usepackage{geometry}
\usepackage{times,latexsym,amssymb}
\usepackage{amsmath,amsthm,bm}
\usepackage{color}
\usepackage[colorlinks,pdfpagelabels,pdfstartview = FitH,bookmarksopen
= true,bookmarksnumbered = true,linkcolor = blue,plainpages =
false,hypertexnames = false,citecolor = red,pagebackref=false]{hyperref}

\usepackage{soul}
\usepackage{microtype} 
\usepackage{amsbsy}
\usepackage{amstext}
\usepackage{esint}
\usepackage{stmaryrd}
\usepackage[pdftex]{graphicx}
\usepackage{floatflt}
\usepackage{comment}
\usepackage{cancel}
\usepackage{titletoc}
\usepackage{paralist}
\usepackage{enumerate}
\usepackage{schemata}
\usepackage{tikz}
\usetikzlibrary{decorations.pathreplacing,calc,tikzmark}
\usetikzlibrary{matrix}
\usepackage{blindtext}
\usepackage[utf8]{inputenc}
\usepackage{mathtools}
\usepackage{mathrsfs}
\usepackage{yfonts}
\usepackage{braket}
\usepackage{bbm}
\usepackage{pifont}
\usepackage{inputenc}
\usepackage{graphicx}

\allowdisplaybreaks
\newtheorem{mytheorem}{Theorem}
\newtheorem{mylem}[mytheorem]{Lemma}
\newtheorem{mydef}[mytheorem]{Definition}

\newtheorem{mycor}[mytheorem]{Corollary}
\newtheorem{remark}[mytheorem]{Remark}
\numberwithin{mytheorem}{section}
\numberwithin{equation}{section}

\SetSymbolFont{stmry}{bold}{U}{stmry}{m}{n}

\DeclareMathOperator\inn{Int}
\DeclareMathOperator*{\esssup}{ess\,sup}

\DeclareMathOperator*{\essosc}{ess\,osc}

\DeclareMathOperator\divv{div}

\DeclareMathOperator\loc{loc}

\newcommand{\bigchi}{\scalebox{1.3}{$\chi$}}

\makeatletter
\DeclareRobustCommand*{\bfseries}{%
  \not@math@alphabet\bfseries\mathbf
  \fontseries\bfdefault\selectfont
  \boldmath
}
\makeatother

\hypersetup{linkcolor=blue,linktoc=page}

\newcommand{\foo}[1]{\mathbf{#1}}

\allowdisplaybreaks
\newtheorem{myproposition}[mytheorem]{Proposition}
\numberwithin{mytheorem}{section}
\numberwithin{equation}{section}

\renewcommand{\d}{\mathrm{d}}
\newcommand{\dx}{\mathrm{d}x}

\newcommand{\dt}{\mathrm{d}t}
\newcommand{\ds}{\mathrm{d}s}

\newcommand{\tildexi}{\Tilde{\xi}}
\newcommand{\R}{\mathbb{R}}
\newcommand{\N}{\mathbb{N}}
\renewcommand{\epsilon}{\varepsilon}

\renewcommand{\d}{\mathrm{d}}
\newcommand{\F}{\mathcal{F}}
\newcommand{\G}{\mathcal{G}}
\newcommand{\B}{\mathcal{B}}
\newcommand{\h}{\mathcal{H}}
\renewcommand{\epsilon}{\varepsilon}

\newcommand{\uproman}[1]{\uppercase\expandafter{\romannumeral#1}}

\def\YYint#1#2#3{{\setbox0=\hbox{$#1{#2#3}{\iint}$}
    \vcenter{\hbox{$#2#3$}}\kern-.51\wd0}}

\def\Xint#1{\mathchoice
{\XXint\displaystyle\textstyle{#1}}%
{\XXint\textstyle\scriptstyle{#1}}%
{\XXint\scriptstyle\scriptscriptstyle{#1}}%
{\XXint\scriptscriptstyle\scriptscriptstyle{#1}}%
\!\int}
\def\XXint#1#2#3{{\setbox0=\hbox{$#1{#2#3}{\int}$ }
\vcenter{\hbox{$#2#3$ }}\kern-0.555\wd0}}

\def\fint{\Xint-}
\newcommand{\vertiii}[1]{{\left\vert\kern-0.25ex\left\vert\kern-0.25ex\left\vert #1 
    \right\vert\kern-0.25ex\right\vert\kern-0.25ex\right\vert}}


\begin{document}
\title[Gradient regularity for degenerate elliptic equations]{Gradient regularity for widely degenerate elliptic partial differential equations}
\date{\today}


\author[M. Strunk]{Michael Strunk}
\address{Michael Strunk\\
Fachbereich Mathematik, Universit\"at Salzburg\\
Hellbrunner Str. 34, 5020 Salzburg, Austria}
\email{michael.strunk@plus.ac.at}


\begin{abstract}
In this paper, we investigate the regularity of weak solutions~$u\colon\Omega\to\R$ to elliptic equations of the type
\begin{equation*}
    \mathrm{div}\, \nabla \F(x,Du) = f\qquad\text{in~$\Omega$},
\end{equation*}
whose ellipticity degenerates in a fixed bounded and convex set~$E\subset\R^n$ with~$0\in \mathrm{Int}\, E$. Here,~$\Omega\subset\R^n$ denotes a bounded domain, and~$\F \colon \Omega\times\R^n \to\R_{\geq 0}$ is a function with the properties: for any~$x\in\Omega$, the mapping~$\xi\mapsto \F(x,\xi)$ is regular outside~$E$ and vanishes entirely within this set. Additionally, we assume~$f\in L^{n+\sigma}(\Omega)$ for some~$\sigma > 0$, representing an arbitrary datum. Our main result establishes the regularity
 \begin{equation*}
     \mathcal{K}(Du)\in C^0(\Omega)
 \end{equation*}
for any continuous function~$\mathcal{K}\in C^0(\R^n)$ vanishing on~$E$.
\end{abstract}


\maketitle
\vspace{-0.5cm}
\tableofcontents


\section{Introduction and main results}
We study the behavior of weak solutions to 
\begin{equation} \label{pde}
    \divv \nabla \F(x,Du) = f\qquad\text{in~$\Omega$}.
\end{equation}
Here and in the following,~$\Omega$ denotes a bounded domain in~$\R^n$ ($n\geq 2$). The datum~$f$ is assumed to belong to~$L^{n+\sigma}(\Omega)$ for some parameter~$\sigma>0$.  For any~$x\in\Omega$, the function~$\F\colon\Omega\times\R^n \to\R_{\geq 0}$ is required to satisfy certain regularity and structure conditions outside a bounded and convex set~$E\subset\R^n$ with~$0\in\inn E$ but vanishes entirely inside~$E$. Outside of~$E$, i.e. in the set of points~$\{ \xi\in\R^n: \xi\notin \overline{E} \}$, we assume~$\F$ to be nondegenerate with respect to the~$\xi$ variable, where the smallest eigenvalue of its Hessian~$\nabla^2\F$ degenerates either when approaching the set~$E$ from outside or for large values of~$\xi$. We refrain from stating the precise conditions imposed on~$\F$ for the moment and rather refer the reader to Section~\ref{subsec:assumptions}. It is a well-known fact from the theory of calculus of variations that equation~\eqref{pde} can also be interpreted as the Euler-Lagrange equation of the respective variational functional 
\begin{equation} \label{minimizer}
    \mathcal{L}(u) \coloneqq \int_{\Omega} \big(\F(x,Du) + f u\big)\,\dx.
\end{equation}
We recall that any local minimizer~$u$ to the functional~$\mathcal{L}(u)$ in a suitable set of admissible functions is a weak solution to~\eqref{pde}. The particular characteristic of equation~\eqref{pde} is given by the large set of degeneracy of~$\F$. This circumstance is problematic in the sense that less regularity of weak solutions to~\eqref{pde} in general can be expected when comparing it to equations degenerating only at a single point, as it is for example the case with the well-known~$p$-Laplace equation. In particular, no more than Lipschitz regularity of weak solutions can be expected in general. To illustrate this point, consider the prototypical example of widely degenerate equations that we always keep in mind, namely
\begin{equation} \label{prototype}
   \divv \bigg(a(x) \frac{(|Du|-1)^{p-1}_+ Du}{|Du|}\bigg) = f\qquad\text{in~$\Omega$},
\end{equation}
where~$p>1$ denotes an arbitrary growth exponent and~$a\in W^{1,\infty}(\Omega)$ denotes a Lipschitz-continuous function with the property that~$0<C_1 \leq a(x) \leq C_2$ for any~$x\in\Omega$ and positive constant~$C_1\leq C_2$. We note that the model equation~\eqref{prototype} represents the Euler-Lagrange equation of the functional
\begin{equation}\label{prototypefunctional}
  \mathcal{L}_p(u) \coloneqq \int_{\Omega} \bigg( \frac{a(x)}{p} (|Du|-1)^p_+ + f u\bigg)\,\dx,
\end{equation}
where local minimizers of~$\mathcal{L}_p(u)$ are considered in the function space~$W^{1,p}_{\loc}(\Omega)$. Accordingly, any~$1$-Lipschitz function is indeed a weak solution to~\eqref{prototype} in the case where~$f\equiv 0$, resulting in a lack of higher regularity of weak solutions in general. In particular, the desired~$C^{1,\alpha}_{\loc}(\Omega)$-regularity is not to be expected in general. Nevertheless, it is still possible to study the gradient regularity in the subset of points where~$\{ x\in\Omega: Du \in\R^n\setminus E \}$, as the equation ensures regularity in these points. Notably, the gradient is indeed continuous within this subset. More broadly, as outlined in our main regularity result, Theorem~\ref{hauptresultat}, the regularity of the composition~$\mathcal{K}(Du)\in C^0(\Omega)$ can be established for any continuous mapping~$\mathcal{K}\in C^0(\R^n)$ that vanishes on~$E$. The novelty of this article presents itself in two ways. Firstly, we allow non-autonomous functionals, i.e. the integrand $\F$ may additionally depend on the spatial variable~$x\in\Omega$. Secondly, we employ a proof technique differing from the methods utilized in prior research on this topic, such as~\cite{colombo2017regularity,santambrogio2010continuity}. Our approach is inspired by the pioneering works of De Giorgi, DiBenedetto, Friedman, and Uhlenbeck,
which have demonstrated considerable flexibility in previous applications. 


\subsection{Literature overview} In previous works such as~\cite{brasco2011global,brascocongested,clop2020very,fonseca2002existence}, the local Lipschitz regularity of weak solutions in the domain~$\Omega$ has been proven for certain variants of~\eqref{pde}, where the integrand~$\F$ is assumed to satisfy similar growth and ellipticity conditions to the prototype equation~\eqref{prototype}, subject to appropriate assumptions regarding the datum~$f$. In view of higher differentiability or $C^1$-regularity of solutions, rather than investigating these properties for~$Du$ itself, more about regularity of the composition of~$Du$ with a suitable function~$\mathcal{K}\colon \R^n\to\R$ can be said. For example, it has been shown in~\cite{brasco2011global,brasco2014certain,clop2020very} that the composition~$\mathcal{K}(Du)$ is weakly differentiable, where~$\mathcal{K}\colon \R^n\to\R$ denotes a certain nonlinear function vanishing on the closed unit ball. Similarly, any result regarding gradient continuity can never directly concern~$Du$ itself, since the equation does not provide information about regions where~$\{|Du| \leq 1\}$ and thus fails to capture any values in this case. However, for the quantity~$\h(x,Du)\coloneqq \nabla\F(x,Du)$ respectively for~$\mathcal{K} (Du)$, where~$\mathcal{K} \in C^0(\R^n)$ denotes an arbitrary continuous function vanishing in the closed unit ball, some $C^1$-regularity results have been obtained. First shown by Santambrogio and Vespri in dimension~$n=2$ for a variant of equation~\eqref{pde}, where the function~$\F$ does not depend on the variable~$x$, they derived in~\cite{santambrogio2010continuity} that the composition of any continuous function vanishing in the unit ball with~$
D u$ is again continuous. This result was subsequently generalized by Colombo and Figalli in~\cite{colombo2017regularity} for any dimension~$n \in\N$ and a more general bounded and convex set~$E\subset\R^n$ with~$0\in \inn{E}$ as the set of degeneracy, not necessarily being the unit ball. We note that these assumptions on the set of degeneracy~$E$ align with those outlined in this manuscript. Yet, the proofs found in both~\cite{colombo2017regularity,santambrogio2010continuity} are specifically designed for the scalar case and cannot be easily extended to accommodate the vectorial case, i.e. the case where~$u=(u_1,\ldots,u_N)\in\R^N$ for some~$N\geq 2$. In the context of the prototype example presented in~\eqref{prototype} without Lipschitz continuous coefficients~$a(x)$ however, a similar result was demonstrated by Bögelein, Duzaar, Giova, and Passarelli di Napoli in their recent work~\cite{bogelein2023higher}. Their approach, inspired by Campanato-type comparison estimates and a De Giorgi-type level set method, has proven to be highly flexible, allowing also the treatment of vector valued solutions. This result was later generalized by Mons in~\cite{mons2023higher} who recovered the very same regularity result for vector valued solutions to a more general class of widely degenerate equations, with the vector field incorporated in the diffusion term assumed to satisfy standard~$p$-growth conditions in the set of points~$\{ |\xi|>1 \}$. Moreover, Grimaldi~\cite{grimaldi2024higher} extended the analysis to vector-valued solutions~$u \colon \Omega \to \R^N$,~$N \geq 1$, considering a more general structural setting. In this framework, the standard Euclidean norm in~\eqref{prototype} is replaced by a norm that is induced by a bounded, symmetric, and coercive bilinear form defined on $\mathbb{R}^{Nn}$. 
In this paper, we extend the approach established in~\cite{bogelein2023higher,colombo2017regularity,mons2023higher,santambrogio2010continuity} to equations that degenerate within a fixed bounded and convex set~$E \subset \mathbb{R}^n$ with~$0 \in \inn{E}$. Our approach does not rely on any specific growth or monotonicity conditions on~$\mathcal{F}$ outside~$E$, and also includes the case where~$\mathcal{F}$ depends on the spatial variable~$x$. The flexibility of this method suggests potential for further exploration, potentially culminating in a corresponding regularity result in the parabolic setting. 


\subsection{Structure conditions}\label{subsec:assumptions}
Throughout this paper, we consistently assume the following set of structure conditions to hold true. Let~$E\subset \R^n$ denote a bounded and convex set with~$0\in\inn E$. By~$|\cdot|_E$ we denote the~\textit{Minkowski functional} (sometimes also \textit{Gauge functional}) on~$\R^n$, given by
\begin{align} \label{Minkowski}
   |\cdot|_E \colon\R^n \to\R_{\geq 0},\qquad |\xi|_E \coloneqq \inf\{ t>0:\xi\in tE  \},
\end{align}
where~$tE\coloneqq\{tx:x\in E\}$. As already considered by Colombo and Figalli~\cite{colombo2017regularity} in a very similar context, the Minkowski functional will also turn out expedient in our discussion throughout this manuscript. Several properties of~$|\cdot|_E$ shall be discussed later on in Section~\ref{subsec:convex}. We consider a non-negative function
$$\F\colon\Omega\times\R^n\to\R_{\geq 0}$$
that is subject to the following assumptions
\begin{align} \label{fregularity}
        \begin{cases}
    \F(x,\xi) = 0  & \mbox{for any~$x\in\Omega$,~$\xi\in E$}, \\
   \xi \mapsto \F(x,\xi)~\mbox{is convex}  &\mbox{for any~$x\in\Omega$},  \\
    \xi \mapsto\F(x,\xi)\in C^1(\R^n) \cap C^2(\R^n\setminus \overline{E} ), & \mbox{for any~$x\in\Omega$}  \\ 
   |\partial_{x_i} \nabla\F(x,\xi)| \leq C(L) &\mbox{for any~$x\in\Omega$,~$|\xi|_E \leq L$}, 
        \end{cases}
  \end{align}
for any~$i=1,\ldots,n$ and any~$L\geq 1$, where~$C=C(L)$ denotes a positive Lipschitz constant that depends on~$L\geq 0$. Put differently,~$\nabla^2 \F$ is assumed to be continuous outside of~$E$ with respect to the gradient variable~$\xi$ for any~$x\in\Omega$, whereas~$\nabla\F$ is assumed to be continuous on the whole of~$\R^n$ and also locally Lipschitz continuous with respect to~$x$. However, the~$C^2$-regularity of~$\F$ with respect to~$\xi$ may break down at the boundary~$\partial E$ of the set of degeneracy. Moreover,~$\F$ vanishes entirely for any~$x\in\Omega$ whenever there holds~$\xi\in E$, whereas the partial mapping~$\xi\mapsto \F(x,\xi)$ is additionally assumed to be convex. Here, the operators~$\nabla$ and~$\nabla^2$ always denote the derivatives with respect to the gradient variable~$\xi$, whereas~$\partial_x $ represents the derivative with respect to the spatial variable~$x$. In addition, we assume that~$\mathcal{F}$ is elliptic outside the degeneracy set~$E$. That is, for every~$\delta > 0$, there exist positive constants~$\lambda=\lambda(\delta)$ and $\Lambda=\Lambda(\delta)$ with $0<\lambda\le\Lambda$, such that
\begin{equation} \label{voraussetzung}
  \qquad  \lambda(\delta) |\eta|^2 \leq \langle \nabla^2 \F(x,\xi) \eta,\eta \rangle \leq \Lambda(\delta) |\eta|^2\qquad\text{for any~$1+\delta \leq |\xi|_E \leq \delta^{-1}$}
\end{equation}
for a.e.~$x\in\Omega$. In other words, for any~$x\in\Omega$ the smallest eigenvalue of~$\nabla^2\F(x,\xi)$ is bounded below by~$\lambda(\delta)$ as long as~$1+\delta \leq |\xi|_E \leq \delta^{-1}$, but may vanish in the limit~$\delta \downarrow 0$. On the contrary, the largest eigenvalue of~$\nabla^2\F(x,\xi)$ is bounded above by~$\Lambda(\delta)$, but can grow arbitrarily fast when~$\delta$ decreases. \,\\


As an example of an integrand~$\mathcal{F}$ satisfying the structural conditions~\eqref{fregularity} and~\eqref{voraussetzung}, we consider the prototype given in~\eqref{prototypefunctional}, i.e.
\begin{equation*}
    \F_p(x,\xi) = \frac{a(x)}{p}(|\xi|-1)^p_+ ,
\end{equation*}
where~$a\in W^{1,\infty}(\Omega,\R_{> 0})$ denotes a Lipschitz continuous function with~$0<C_1 \leq a(x) \leq C_2$ that satisfies~$|\partial_{x_i}a(x)| \leq A$ for any~$i=1,\ldots,n$ and a.e.~$x\in\Omega$. Here,~$A,C_1,C_2$ denote positive constants. The degeneracy set of~$\F_p$ is given by the unit ball, i.e.~$E=B_1(0)$ and the associeated Minkowski functional~$|\cdot|_E$ is exactly the standard Euclidean norm on~$\R^n$, i.e.~$|\xi|_E=|\xi|$ for any~$\xi\in\R^n$. A direct calculation verifies that there holds
$$\nabla\F_p(x,\xi) = a(x) \frac{(|\xi|-1)^{p-1}_+}{|\xi|}\xi,$$
which further implies that for any~$i=1,\ldots,n$ we have
\begin{equation*}
    |\partial_{x_i} \nabla \F_p (x,\xi)| \leq A(|\xi|-1)_+^{p-1} \leq A (L-1)^{p-1}
\end{equation*}
for any~$x\in\Omega$ and any~$1\leq |\xi|\leq L$, where~$1\leq L < \infty$. Moreover, we have
\begin{align} \label{prototypehessian}
    \nabla^2 \F_p(x,\xi) = a(x)\bigg(\frac{(|\xi|- 1)^{p-1}_+ }{|\xi|} \foo{I}_n + \frac{(p-1)(|\xi|-1)^{p-2}_+}{|\xi|^2}(\xi \otimes \xi) - \frac{(|\xi|-1)^{p-1}_+}{|\xi|^3}(\xi\otimes\xi) \bigg)
\end{align}
for any~$x\in\Omega$ and~$\xi\in\R^n$. At this point, it is straightforward to check that conditions~\eqref{fregularity} are indeed satisfied. In fact, in the scenario where~$p \geq 2$, the mapping~$ \xi \mapsto \F_p(x, \xi)$ is twice continuously differentiable across all of~$\mathbb{R}^n$ for any~$x\in\Omega$. The sub-quadratic case with~$1 < p < 2$ presents a different situation as~$\F_p$ is only twice continuously differentiable with respect to~$\xi$ inside and outside the unit sphere, i.e. on~$\mathbb{R}^n \setminus \{ \xi \in \mathbb{R}^n: |\xi| = 1 \}$. In order to treat the full parameter range~$p>1$ in a unified way, we set
\begin{align*}
    \Gamma_1(|\xi|) &\coloneqq \min \bigg\{ \frac{(|\xi|- 1)^{p-1} }{|\xi|}, (p-1)(|\xi|-1)^{p-2}  \bigg\} \bigchi_{ \{|\xi|>1\} } \\
    \Gamma_2(|\xi|) &\coloneqq \max \bigg\{ \frac{(|\xi|- 1)^{p-1} }{|\xi|}, (p-1)(|\xi|-1)^{p-2}  \bigg\} \bigchi_{ \{|\xi|>1\} }.
\end{align*}
Further, we infer from~\cite[Lemma~2.7]{bogelein2023higher} and~\eqref{prototypehessian}, that there holds the ellipticity estimate
\begin{equation} \label{prototypeallgemeineell}
    C_1 \Gamma_1(|\xi|) |\eta|^2 \leq \langle \nabla^2 \F_p(x,\xi) \eta,\eta \rangle \leq C_2 \Gamma_2(|\xi|) |\eta|^2
\end{equation}
for any~for any~$x\in\Omega$, and~$\xi, \eta\in\R^n$. For instance, when~$p\geq 2$, the preceding estimate becomes
\begin{equation} \label{prototypeexplizitell}
C_1 \frac{(|\xi|- 1)^{p-1}_+ }{|\xi|}|\eta|^2 \leq \langle \nabla^2\F_p(x,\xi)\eta,\eta\rangle  \leq C_2 (p-1) (|\xi|-1)^{p-2}_+ |\eta|^2.
\end{equation}
Now, given~$0<\delta<1$, consider~$\xi\in\R^n$, such that~$1+\delta\leq|\xi|\leq \delta^{-1}$. Estimate~\eqref{prototypeexplizitell} establishes that conditions~\eqref{voraussetzung} are satisfied with constants~$\lambda(\delta) = C_1 \delta^p$ and~$\Lambda(\delta)= C_2 (p-1) \big(\frac{1-\delta}{\delta}\big)^{p-2}$ in the super-quadratic case~$p\geq 2$. In particular, $\lambda$ vanishes in the limit~$\delta \downarrow 0$, whereas~$\Lambda$ becomes unbounded. Only in the linear case~$p=2$, the eigenvalues of the Hessian~$\nabla^2\F_p(x,\xi)$ are bounded from above universally on the whole of~$\Omega\times\R^n$ by~$\Lambda=C_2$, while the smallest eigenvalue still vanishes when~$\xi$ is close to the unit ball. In contrast, in the the sub-quadratic setting~$1<p<2$, a similar reasoning verifies that the reverse circumstance to the super-quadratic case~$p>2$ applies, and the smallest eigenvalue of~$\nabla^2\F_p(x,\xi)$ vanishes for large~$|\xi|$, whereas the largest eigenvalue becomes unbounded near the unit sphere~$\{\xi\in\R^n:|\xi|=1\}$. 


\subsection{Definition of weak solution} \label{sec:weaksolution}
        
At this point, we present our notion of a weak solution to equation~\eqref{pde}.


\begin{mydef} \label{defweakform}
   Let~$\F$ satisfy the structure conditions~\eqref{fregularity} and~\eqref{voraussetzung}. A measurable function~$u\in W^{1,\infty}_{\loc}(\Omega)$ is a local weak solution of equation~\eqref{pde} in~$\Omega$, if the integral identity
\begin{equation} \label{weakform}
    \int_{\Omega}  \langle \nabla\F(x,Du), D\phi \rangle \,\dx= - \int_{\Omega} f\phi\,\dx
\end{equation}
is satisfied for any test function~$\phi\in C^{\infty}_0(\Omega)$. 
\end{mydef}


\begin{remark} \label{lipschitzannahme} \upshape
   As already pointed out in~\cite[Remark~1.2]{colombo2017regularity}, the local Lipschitz assumption made on weak solutions in Definition~\ref{defweakform} can be justified by imposing standard growth assumptions on~$\F$. To state an example, the vector field incorporated in the prototype equation~\eqref{prototype} satisfies standard~$p$-growth conditions as long as~$|\xi|> 1$. Consequently, it can be shown that weak solutions are of class~$W^{1,\infty}_{\loc}(\Omega)$, where we refer the reader to~\cite{brasco2011global,brascocongested} in the scalar case and to~\cite{clop2020very} in the vectorial case. Due to the very general form of~$\F$ and the structure conditions~\eqref{voraussetzung}, without additional information on the growth rate of~$\xi \mapsto \nabla\F(x,\xi)$, such as~$p$-growth, we cannot expect quantitative energy estimates for weak solutions to~\eqref{pde}. Already to ensure convergence of the integral involving the solution~$u$ in \eqref{weakform}, further knowledge about the qualitative behavior of~$\F$ for large gradient values would be necessary. Consequently, we focus on weak solutions that are \textit{a priori} locally Lipschitz continuous in~$\Omega$. 
\end{remark}


\subsection{Main result}

The following theorem presents the main result of this paper. 


\begin{mytheorem} \label{hauptresultat}
Let~$n\geq 2$,~$f\in L^{n+\sigma}(\Omega)$ for some~$\sigma>0$, and~$E\subset\R^n$ denote the bounded and convex set with~$0\in\inn{E}$ on which~$\F$ degenerates. Further, let~$u\in W^{1,\infty}_{\loc}(\Omega)$ be a local weak solution to~\eqref{pde} in~$\Omega$ under the structure conditions~\eqref{fregularity} and~\eqref{voraussetzung}. Then, there holds
\begin{equation*}
    \mathcal{K}(Du)\in C^0(\Omega)
\end{equation*}
for any~$\mathcal{K}\in C^0(\R^n)$ with~$\mathcal{K}\equiv 0$ on~$E$.
\end{mytheorem}


As an immediate consequence of Theorem~\ref{hauptresultat}, we obtain the following corollary for the composition of the vector field~$\nabla \F$ with the gradient~$Du$.


\begin{mycor} \label{corollaryzwei}
    Let the assumptions of Theorem~\ref{hauptresultat} hold true. If there additionally holds
    $$\nabla\F \in C^0(\Omega\times\R^n,\R^n),$$
    then we have
    \begin{equation*}
        \nabla \F(x,Du) \in C^0(\Omega,\R^n).
    \end{equation*}
\end{mycor}


\subsection{Strategy of the proof} \label{strategy}

In this section, we outline the essential steps leading to our main Theorem~\ref{hauptresultat}. To begin, we introduce a truncation $\Tilde{\F}$ of the integrand~$\F$ for values of~$|\xi|$ exceeding a certain threshold~$L < \infty$. We then modify the truncation by adding a suitable convex function to obtain a convex integrand~$\hat{\mathcal{F}}$. The threshold~$L$ aligns with the essential supremum of~$|Du|$ within a compactly contained ball~$B_R = B_R(y_0) \Subset \Omega$. Due to the local Lipschitz regularity of~$u$ in~$\Omega$, we find that~$L = \|Du\|_{L^\infty(B_R)} < \infty$. According to this construction, any local weak solution~$u$ to~\eqref{pde} is also a local weak solution to the redefined equation
$$\divv \nabla\hat{\F}(x,Du) = f \qquad \mbox{in~$B_R$}.$$
Moreover, we enhance the diffusion term of equation~\eqref{pde} by incorporating the additive quantity~$\epsilon\Delta u$, with a small parameter~$\epsilon\in(0,1]$. This modification yields a more regular version of~\eqref{pde}, with the integrand denoted by~$\Tilde{\F}_\epsilon$, that is elliptic across all of~$\R^n$ with respect to~$\xi$ and avoids degeneracy within the set~$E$. Of course, the ellipticity constant depends on the parameter~$\epsilon$ in general. Furthermore, the obtained operator~$\xi \mapsto \Tilde{\F}_\epsilon(x,\xi)$ satisfies quadratic growth on the whole of~$\R^n$ for any~$x\in\Omega$. Subsequently, we consider the unique weak solution~$u_\epsilon \in u + W^{1,2}(B_R)$ to the weak form of the approximating equation on a ball~$B_R \Subset \Omega$, with the Dirichlet boundary condition involving the local weak solution~$u$ of equation~\eqref{pde}. The nondegeneracy of the regularized equation allows us to establish higher regularity of the form
$$u_\epsilon \in W^{1,\infty}_{\loc}(B_R) \cap W^{2,2}_{\loc}(B_R)$$
for the approximating solutions. Moreover, due to the quadratic growth of~$\xi \mapsto \Tilde{\F}_\epsilon(x,\xi)$ for any~$x \in \Omega$, we are able to acquire a quantitative local~$L^\infty$-gradient bound for~$Du_\epsilon$, with the~$L^\infty$-gradient norm essentially bounded by the~$L^2$-gradient norm up to constants and an additive quantity. This expedient matter of fact further enables us to derive a local uniform bound for~$\|Du_\epsilon\|_{L^\infty}$ with respect to the approximating parameter~$\epsilon\in(0,1]$ in~$B_R$. The main difficulty hereinafter lies in proving Theorem~\ref{holdermainresult}, which establishes the local Hölder continuity in~$B_R$ of the function
\begin{equation*}
    \G_\delta(Du_\epsilon)\coloneqq \frac{(|Du_\epsilon|_E-(1+\delta))_+}{|Du_\epsilon|_E} Du_\epsilon
\end{equation*}
 for any~$\delta,\epsilon \in(0,1]$. Here, the Hölder exponent~$\alpha_\delta\in(0,1)$ and the positive Hölder constant~$C_\delta$ both depend on the data as well as on the parameter~$\delta\in(0,1]$, but are independent of the parameter~$\epsilon\in(0,1]$. Despite the general structure of the function~$\xi \mapsto\nabla\F(x,\xi)$ not being of Uhlenbeck-type in our setting, it still turns out advantageous to consider~$\G_\delta(Du_\epsilon)$. The mapping~$\G_\delta$ vanishes within a slightly larger set compared to~$\xi\mapsto \nabla\F(x,\xi)$, and is strictly positive outside the complement~$\{\xi\in\R^n: |\xi|_E>1+\delta\}$. As previously highlighted in~\cite{bogelein2023higher}, it turns out expedient to consider the function~$\G_\delta(Du_\epsilon)$ since in the set of points where~$\G_\delta$ is positive, the equation~\eqref{pde} exhibits an ellipticity property, with an ellipticity constant that depends on the parameter~$\delta\in(0,1]$. The proof of Theorem~\ref{holdermainresult} involves a careful treatment of both the non-degenerate and the degenerate regime. In the non-degenerate regime, the set of points where~$\partial_{e^*}u_\epsilon$ for at least one~$e^{*}\in\partial E^{*}$ is close to its supremum is large in measure, while in the degenerate regime, the set of points where~$\partial_{e^{*}} u_\epsilon$ is far from its supremum is large in measure for any~$e^{*}\in\partial E^{*}$. Here,~$E^*$ denotes the unit ball with respect to the induced dual mapping of the convex set~$E\subset \R^n$, where the reader is for the moment referred to Section~\ref{subsec:convex} for a more detailed explanation. In the non-degenerate regime, our goal is to establish a lower bound for
 $$|Du_\epsilon|_E = \sup\limits_{e^*\in\partial E^*} \partial_{e^*} u_\epsilon$$
 within a ball~$B_{\frac{\rho}{2}}(x_0) \Subset B_{2\rho}(x_0) \Subset B_R \Subset\Omega$. We achieve this result by leveraging the measure-theoretic information that characterizes the non-degenerate regime and is assumed to hold true on~$B_\rho(x_0)$. By differentiating the regularized equation, we find that for any~$e^*\in\partial E^*$, the function~$\partial_{e^*} u_\epsilon$ is a weak solution to a linear elliptic equation. Notably, the ellipticity constant for this equation is independent of~$\epsilon \in (0,1]$ and depends solely on the parameter~$\delta \in (0,1]$ as well as the given data, a consequence of the established lower bound for~$|Du_\epsilon|_E$. This fact enables us to utilize established excess-decay estimates for linear elliptic equations that are taken from \cite{gilbargtrudinger}. In the degenerate regime, a sub-solution to a linear elliptic equation is constructed that is given by
$$v_\epsilon\coloneqq (\partial_{e^*}u_\epsilon-(1+\delta))^2_+,$$
 facilitating the derivation of a De Giorgi class-type estimate. This results in a reduction of the supremum of~$\G_\delta(Du_\epsilon)$ on a smaller ball, as outlined in Proposition~\ref{degenerateproposition}. By combining the results of Propositions~\ref{nondegenerateproposition} and~\ref{degenerateproposition} through considering a sequence of shrinking nested balls, one navigates between the non-degenerate and degenerate regimes, adapting the approach based on the measure-theoretic information at each step. The transition between regimes is pivotal, as it determines the course of the analysis in subsequent steps. If the non-degenerate regime applies at any given step, it will be maintained throughout. However, transitioning to a smaller cylinder in the setting of the degenerate regime results in uncertainty regarding whether in the next step the non-degenerate or degenerate regime applies. Once we have established the Hölder continuity of~$\G_\delta(Du_\epsilon)$, an application of Arzelà-Ascoli's theorem and passing to the limit~$\epsilon\downarrow 0$ show that~$\G_\delta(Du)$ is also locally Hölder continuous in~$B_R$, with the Hölder exponent and constant depending on the parameter~$\delta$. By further taking the limit~$\delta\downarrow 0$, it follows that~$\G(Du)$ is locally uniformly continuous in~$B_R$. However, at this stage, the quantitative control over the Hölder exponent and constant is lost, leaving uncertainty about whether the limit function~$\G(Du)$ is Hölder continuous or what the optimal modulus of continuity may be. Eventually, the main regularity theorem~\ref{hauptresultat} can be deduced from these results.


\subsection{Plan of the paper} 
The paper is organized as follows: in Section~\ref{sec:preliminaries}, we will commence by presenting the notation and framework, which includes supplementary material required later on and our notion of a weak solution to equation~\eqref{pde}. Subsequently, in Section~\ref{sec:holder} we aim to give the proof of the main regularity Theorem~\ref{hauptresultat}. Sections~\ref{sec:nondegenerate} and~\ref{sec:degenerate} include the proof of Propositions~\ref{nondegenerateproposition} and~\ref{degenerateproposition}, essential components for proving an intermediate regularity result that is Theorem~\ref{holdermainresult}, concluding the article. \,\\


\textbf{Acknowledgements.}
The author would like to express gratitude to Professor Verena Bögelein for her guidance throughout the development of this article. \, \\
This research was funded in whole or in part by the Austrian Science Fund (FWF) [10.55776/P36295]. For open access purposes, the author has applied a CC BY public copyright license to any author accepted manuscript version arising from this submission. \,\\

\textbf{Conflict of Interest.} The author declares that there is no conflict of interest. \,\\

\textbf{Data availability.} This manuscript has no associated data.


\section{Preliminaries} \label{sec:preliminaries}
\subsection{Notation and setting}
Throughout this paper,~$\Omega\subset\R^n$ denotes a bounded domain. Its boundary will be denoted in the usual way by~$\partial\Omega \subset \R^{n-1}$. For the open ball in~$\R^n$ of radius~$\rho>0$ and center~$x_0\in\R^n$, we shall always write~$B_\rho(x_0)\subset\R^n$. For any set~$A\subset\R^n$ with~$|A|>0$ and any function~$f\in L^1(A)$, the mean value of~$f$ on~$A$ is given by
\begin{equation*}
    (f)_{A} \coloneqq \fint_{A} f(x)\,\dx = \frac{1}{|A|}\int_{A} f(x)\,\dx,
\end{equation*}
where~$|A|$ denotes the~$n$-dimensional Lebesgue measure of~$A$, i.e.~$|A|=\mathcal{L}^n(A)$, a notation that will be utilized throughout the entire article. In particular, if the considered set~$A$ is given by a ball~$B_\rho(x_0)$ with some radius~$\rho>0$ and center~$x_0\in\R^n$, we shall write
\begin{equation*}
    (f)_{x_0,\rho} \coloneqq \fint_{B_{\rho}(x_0)} f(x)\,\dx.
\end{equation*}
If it is convenient, we will omit the center~$x_0$ for brevity and simply denote the ball as~$B_\rho$ and the mean value of~$f$ on~$B_\rho$ as~$(f)_\rho$. The standard scalar product on~$\R^n$ will be denoted by~$\langle \cdot,\cdot\rangle$ for any~$\xi,\eta\in\R^n$, and the induced euclidean norm on~$\R^n$ by~$|\cdot|$. The dyadic product of two vectors~$\xi,\eta \in\R^n$ is indicated by~$\xi \otimes \eta$. The positive part of a real quantity~$a\in\R$ is denoted as~$a_+ = \max\{a,0\}$, while the negative part is denoted as~$a_- = \max\{-a,0\}$. \,\\

The expressions~$\nabla \F$ and~$\nabla^2\F$ represent the first and second order derivative of the function~$\F(x,\xi)$ with respect to~$\xi$ on~$\R^n\setminus E$, whereas~$\partial_x \nabla\F(x,\xi)$ denotes the derivative of~$\nabla\F$ with respect to the spatial variable~$x$. Instead of~$\nabla\F$, we shall sometimes write~$\h$ and both notations~$\nabla \F$ and~$\h$ will be used interchangeably. Directional derivatives of functions~$u\colon\Omega\to\R$ in direction~$v\in\R^n$ shall be indicated either by~$D_v u$ or by~$\partial_v u$ whenever the respective quantity exists either in the strong or the weak sense. \,\\

The following set of functions will appear multiple times throughout the article. We define 
\begin{equation} \label{gfunk}
    \G(\xi) \coloneqq \frac{(|\xi|_E - 1)_+ }{|\xi|_E}\xi \qquad\text{for~$\xi\in\R^n$},
\end{equation}
 which has already been considered in a similar fashion in the previous articles~\cite{bogelein2023higher,colombo2017regularity,santambrogio2010continuity}. Since the technique of approximation forms a major part of this paper, we will introduce a translated version of~$\G$. For a parameter~$\delta\geq 0$, we set
\begin{equation} \label{gdeltafunk}
    \G_\delta(\xi) \coloneqq \frac{(|\xi|_E -(1+\delta))_+ }{|\xi|_E} \xi \qquad\text{for~$\xi\in\R^n$}.
\end{equation}
Obviously this notion implies~$\G_0=\G$. We further remark that the function~$\G_\delta$ naturally arises in the context of the prototype equation~\eqref{prototype} and takes an important role in~\cite{bogelein2023higher}, where it simplifies to
$$\frac{(|\xi| -(1+\delta))_+ }{|\xi|} \xi.$$
Even though our setting overall differs from the one in~\cite{bogelein2023higher}, the function~$\G_\delta$ still turns out expedient in our investigations. The function~$\G_\delta$ will be employed in an approximation procedure to work around the degeneracy set~$E\subset\R^n$ of~$\F$. However, as~$\F$ is only elliptic outside of~$E$, further approximations are required in order to obtain an elliptic equation on the whole of~$\R^n$. For an approximating parameter~$\epsilon\in[0,1]$, we define
\begin{equation} \label{Fapprox}
     \F_\epsilon(x,\xi)\coloneqq \F(x,\xi)+\epsilon \textstyle{\frac{1}{2}} |\xi|^2\qquad\text{for any~$x\in\Omega$,~$\xi\in\R^n$}
\end{equation}
and note that~$\F_\epsilon$ is indeed elliptic on the whole of~$\R^n$ in the case where~$\epsilon\in(0,1]$, where the ellipticity constant additionally depends on the parameter~$\epsilon$. Note that the quantity~$\frac{1}{2}|\xi|^2$ denotes the convex function of the~$p$-energy in the case where~$p=2$. Thus, in the general divergence form~\eqref{pde} of the equation, this approximation procedure represents an addition of the well-known Laplace operator~$\Delta$ with some scaling parameter~$\epsilon\in[0,1]$. Consequently, we let
\begin{equation} \label{Gapprox}
    \h_\epsilon(x,\xi)\coloneqq \nabla \F_\epsilon(x,\xi) = \big(\h(x,\xi)+\epsilon\xi\big)\in\R^n\qquad\text{for any~$x\in\Omega$,~$\xi\in\R^n$}
\end{equation}
and
\begin{equation*} 
    \nabla^2\F_\epsilon(x,\xi)=\big(\nabla^2\F(x,\xi)+\epsilon I_n \big) \in \R^{n\times n}\qquad\text{for any~$x\in\Omega$,~$\xi\in\R^n\setminus\overline{E}$}.
\end{equation*}
Finally, we define the following bilinear forms on~$\R^n$ that naturally arise from a differentiation of the weak form~\eqref{weakform} of equation~\eqref{pde}, where for~$\eta,\zeta\in\R^n$ we set
 \begin{equation} \label{bilinear}
            \mathcal{B}(\cdot,\xi)(\eta,\zeta)\coloneqq \langle \nabla^2\F(\cdot,\xi)\eta, \zeta \rangle \qquad\text{for~$x\in\Omega,\xi\in\R^n\setminus\overline{E}$}.
        \end{equation}
      Moreover, the bilinear form naturally arising from the approximation~\eqref{Fapprox} is given by
 \begin{equation} \label{bilinearapprox}
            \mathcal{B}_\epsilon(\cdot,\xi)(\eta,\zeta)\coloneqq  \langle \nabla^2\F(\cdot,\xi)\eta,\zeta \rangle +\epsilon \langle \eta, \zeta\rangle  \qquad\text{for~$x\in\Omega,\xi\in\R^n\setminus\overline{E}$}.
        \end{equation}


\subsection{On convex sets and the Minkowski functional} \label{subsec:convex}

As introduced before, throughout this article we consider a bounded and convex set~$E\subset\R^n$ with~$0\in\inn{E}$ as the set of degeneracy of equation~\eqref{pde}. The convexity of~$E$ is required as it induces a mapping on~$\R^n$, which will be particularly important for the derivation of appropriate energy estimates in the course of Sections~\ref{sec:holder} -- \ref{sec:degenerate}. We consider th~\textit{Minkowski functional}~$|\cdot|_E$ on~$\R^n$, defined by
\begin{align*} 
   |\cdot|_E \colon\R^n \to\R_{\geq 0},\qquad |\xi|_E \coloneqq \inf\{ t>0:\xi\in tE  \},
\end{align*}
where~$tE\coloneqq\{tx:x\in E\}$. Moreover, the convexity allows the derivation of a triangle inequality and a generalized reverse triangle inequality for the Minkowski functional, as performed in Lemma~\ref{lem:minkowskitriangle}. In addition to the convexity, the property that~$0\in\inn{E}$ will turn out essential for the proof of Lemma~\ref{gdeltalem}, as it guarantees the Lipschitz regularity of the Minkowski functional~\eqref{Minkowski}. However, note that~$|\cdot|_E$ is not symmetric unless the set~$E$ is symmetric with respect to the origin~$0$. Concerning some properties of the~\textit{Minkowski functional}, we refer the reader to~\cite[Page 109]{minkowski} resp.~\cite[Page 119]{minkowski} for some properties of the latter.

Since~$E$ consists of points which satisfy~$|\xi|_E \leq 1$, i.e.~$E$ corresponds to the unit ball with respect to the Minkowski functional, we also introduce its dual mapping
\begin{align} \label{dualnorm}
    |\xi|_{E'}\coloneqq \sup\limits_{|e|_E \leq 1} \langle\xi,e \rangle = \sup\limits_{e \in E} \langle\xi,e \rangle
\end{align}
as well as the unit ball with respect to the dual mapping
\begin{align} \label{unitballdualnorm}
    E^*\coloneqq \overline{B}_1^{|\cdot|_{E'}} = \{\xi\in\R^n:|\xi|_{E'}\leq 1\} = \{\xi\in\R^n:\langle\xi,e\rangle \leq 1~\forall e\in E\}.
\end{align}
As outlined in~\cite[Section~3]{colombo2017regularity}, this yields the subsequent more appropriate representation formula of the Minkowski functional
\begin{align} \label{minkowskialternativ}
    |\xi|_E = \underset{\substack{\vspace{-0.05cm}\\ e*\in E^{*}}}{\sup} 
\langle\xi,e^*\rangle = \underset{\substack{\vspace{-0.05cm}\\ e*\in \partial E^{*}}}{\sup} 
\langle\xi,e^*\rangle
\end{align}
for any~$\xi\in\R^n$. Of particular importance for our reasoning is the following consequence of~\eqref{minkowskialternativ}: in several energy estimates we shall weakly differentiate a given equation in direction~$e^*$, where~$e^*\in \partial E^*$. This way, we will deduce information about the directional derivatives~$\partial_{e^*} u = \langle Du,e^*\rangle$. Eventually, as~$e^*\in E^*$ can be chosen arbitrarily, we pass to the supremum over all~$e^*\in E^*$ to obtain estimates for the Minkowski functional~$|Du|_E$ of the (weak) gradient of~$u$. We hereby identify the (weak) directional derivative in any direction~$v\in\R^n$ by~
\begin{align} \label{representation}
    v=\sum\limits_{i=1}^n v_i e_i,
\end{align} 
where~$(e_i)_{i=1,\ldots,n}$ denotes the standard base of~$\R^n$ and~$v_i$ represent the coefficients of~$v$ with respect to the standard base. Then, for a.e.~$v\in\R^n$ we have~$\partial_v u = \langle Du,v \rangle$.
Finally, straight from the definition it is clear that~$|\cdot|_E$ is homogeneous for any~$\lambda>0$ in the sense that~$|\lambda\xi|_E=\lambda|\xi|_E$ holds for any~$\xi\in\R^n$. Moreover, we have that~$|0|_E=0$ and, since~$0\in\inn{E}$, also that~$|\xi|_E<\infty$ for all~$\xi\in\R^n$. Due to~$0\in\inn{E}$ and~$E$ being bounded, we find radii~$0<r_E\leq R_E<\infty$ and balls~$B_{r_E}(0)\subset B_{R_E}(0)$, such that
\begin{align} \label{mengeeradii}
    B_{r_E}(0)\subset E\subset B_{R_E}(0),
\end{align}
where~$r_E>0$ is chosen as the largest of all radii~$r>0$ with~$B_r(0)\subset E$, while~$R_E>0$ denotes an arbitrarily chosen fixed radius with the stated property. It now follows easily from the definition of~\eqref{Minkowski} that there holds 
\begin{align} \label{betragminkowski}
    \frac{|\xi|}{R_E} \leq |\xi|_E \leq \frac{|\xi|}{r_E}
\end{align}
for any~$\xi\in\R^n$. This estimate is crucial, since it allows a comparison between the standard Euclidean norm~$|\cdot|$ and the Minkowski functional~$|\cdot|_E$ on~$\R^n$. Finally, for any~$\delta>0$ we define the outer parallel set of~$E$ with distance~$\delta>0$ by
\begin{align} \label{Edelta}
    E_\delta\coloneqq \{\xi\in\R^n:|\xi|_E\leq 1+\delta\}.
\end{align}


\subsection{Auxiliary material} \label{subsec:auxmat}
In this section, a series of algebraic inequalities and further additional material is summarized. \,\\


We recall a version of Kato's inequality.

\begin{mylem} \label{kato}
    Let~$k\in\N$ and~$B_R\subset\R^n$. For any~$u\in W^{2,1}_{\loc}(B_R,\R^k)$ there holds
    \begin{equation*}
        |D|Du||\leq |D^2 u|\quad\text{a.e. in $\Omega$}.
    \end{equation*}
\end{mylem}


The following lemma presents a triangle inequality and a generalized reverse triangle inequality for the Minkowski functional~$|\cdot|_E$.

\begin{mylem} \label{lem:minkowskitriangle}
    For any~$\xi,\eta\in\R^n$ there holds
    \begin{align} \label{est:minkowskitriangle}
        |\xi+\eta|_E\leq |\xi|_E+|\eta|_E.
    \end{align}
    In particular, this implies the following reverse triangle inequality
    \begin{align} \label{est:minkowskireversetriangle}
        ||\xi|_E-|\eta|_E| \leq \max\{|\xi-\eta|_E,|\eta-\xi|_E\}.
    \end{align}
\end{mylem}

\begin{remark} \label{remarkzutriangle} \upshape
   In general, the reverse triangle inequality established in~\eqref{est:minkowskireversetriangle} does not yield the standard reverse triangle inequality that applies to the Euclidean norm, since the Minkowski functional~$|\cdot|_E$ is not necessarily symmetric.
\end{remark}

\begin{proof}
    The triangle inequality~\eqref{est:minkowskitriangle} quickly follows from the definition of the Minkowski functional~\eqref{Minkowski} and the fact that~$E\subset\R^n$ is a bounded and convex set. The generalized reverse triangle inequality~\eqref{est:minkowskireversetriangle} is an immediate consequence of the triangle inequality~\eqref{est:minkowskitriangle}.
    \end{proof}


The fact that~$0\in\inn{E}$ allows us to deduce the following Lipschitz estimate for~$|\cdot|_E$.

\begin{mylem} \label{lem:minkowskilipschitz}
For any~$\xi,\eta\in\R^n$ there holds the Lipschitz estimate
\begin{align} \label{est:lipschitz}
    ||\xi|_E - |\eta|_E | \leq \mbox{$\frac{1}{r_E}$} |\xi-\eta|,
\end{align}
where~$r_E>0$ is specified in~\eqref{mengeeradii}.
\end{mylem}
\begin{proof}
According to the construction of the radius~$r_E$ in~\eqref{mengeeradii}, there holds~$B_{r_E}(0)\subset E$. Thus, for any~$\xi\in\R^n$ we have that
$$ \frac{r_E}{2} \frac{\xi}{|\xi|}\in B_{r_E}(0)\subset E.$$
Hence, the definition of the Minkowski functional yields~$|\xi|_E\leq \frac{2}{r_E}|\xi|$. By an application of Lemma~\ref{lem:minkowskitriangle} and by exploiting~\eqref{mengeeradii}, we infer the claimed Lipschitz estimate
\begin{align*}
    ||\xi|_E-|\eta|_E| \leq \max\{|\xi-\eta|_E,|\eta-\xi|_E\} \leq \mbox{$\frac{1}{r_E}$}|\xi-\eta|
\end{align*}
    for any~$\xi,\eta\in\R^n$. 
\end{proof}


\begin{mylem} \label{lem:algineq}
    For any~$\xi,\eta\in\R^n\setminus\{ 0\}$ there holds
    \begin{equation*}
        \bigg|\frac{\xi}{|\xi|_E}-\frac{\eta}{|\eta|_E}\bigg|_E \leq \frac{R_E}{r_E} \frac{2}{|\xi|_E} |\xi-\eta|_E.
    \end{equation*}
\end{mylem}

\begin{proof}
   By applying the homogeneity of~$|\cdot|_E$,~\eqref{betragminkowski}, Lemma~\ref{lem:minkowskitriangle}, and also Lemma~\ref{lem:minkowskilipschitz}, we calculate 
    \begin{align*}
        | |\eta|_E\xi - |\xi|_E \eta |_E &= | |\eta|_E(\xi-\eta) + \eta(|\eta|_E-|\xi|_E) |_E \\
        &\leq |\eta|_E|\xi-\eta|_E + |\eta|_E |\eta|_E-|\xi|_E \\
        &\leq |\eta|_E|\xi-\eta|_E + |\eta|_E\mbox{$\frac{1}{r_E}$}|\eta-\xi| \\
        &\leq |\eta|_E|\xi-\eta|_E + |\eta|_E\mbox{$\frac{R_E}{r_E}$}|\xi-\eta|_E \\
        &\leq 2 \mbox{$\frac{R_E}{r_E}$} |\eta|_E|\xi-\eta|_E,
    \end{align*}
    which yields the claimed estimate after dividing both sides of the preceding inequality by~$|\xi|_E|\eta|_E \neq 0$ and by exploiting the homogeneity of~$|\cdot|_E$ once more. 
\end{proof}


The subsequent lemma states a relation between the mapping~$\G_\delta$ introduced in~\eqref{gdeltafunk} and the Euclidean norm, which follows the idea from~\cite[Lemma~2.3]{bogelein2023higher}. 

\begin{mylem} \label{gdeltalem}
    Let~$\delta\geq 0$ and~$\xi,\eta\in\R^n$. Then, there holds 
    \begin{equation} \label{est:gdeltalipschitzeins}
        |\G_\delta(\xi)-\G_\delta(\eta)| \leq  3 \Big(\frac{R_E}{r_E}\Big)^2 |\xi-\eta|.
    \end{equation}
    Additionally, if~$\delta>0$ and~$|\xi|_E \geq 1+\delta$, we have
    \begin{equation} \label{est:gdeltalipschitzzwei}
        |\xi-\eta| \leq 3 \Big(\frac{R_E}{r_E}\Big)^2\Big(1+\frac{1}{\delta}\Big) |\G(\xi)-\G(\eta)|.
    \end{equation}
    
\end{mylem}

\begin{proof}
We first establish assertion~\eqref{est:gdeltalipschitzeins} and, as in~\cite[Lemma~2.3]{bogelein2023higher}, distinguish between different cases. Firstly, if~$|\xi|_E,|\eta|_E\leq 1+\delta$, the inequality is obvious since this implies~$\G_\delta(\xi)=\G_\delta(\eta)=0$. If~$|\xi|_E>1+\delta$ and~$|\eta|_E\leq 1+\delta$, we exploit~\eqref{betragminkowski} and Lemma~\ref{lem:algineq} to obtain
\begin{align*}
    |\G_\delta(\xi)-\G_\delta(\eta)| &= \frac{|\xi|}{|\xi|_E} (|\xi|_E-(1+\delta)) \\
    &\leq R_E(|\xi|_E-(1+\delta)) \\
    &\leq R_E(|\xi|_E-|\eta|_E) \\
    &\leq \mbox{$\frac{R_E}{r_E}$} |\xi-\eta|.
\end{align*} 
The reverse case where~$|\xi|\leq 1+\delta$ and~$|\eta|_E>1+\delta$ follows in the same way. Finally, if~$|\xi|_E,|\eta|_E >1+\delta$, we employ the homogeneity of the Minkowski functional,~\eqref{betragminkowski}, and Lemma~\ref{lem:algineq}, which yields
\begin{align*}
    |\G_\delta(\xi)-\G_\delta(\eta)| &\leq R_E \bigg| \frac{(|\xi|_E-(1+\delta))\xi}{|\xi|_E} - \frac{(|\eta|_E-(1+\delta))\eta}{|\eta|_E} \bigg|_E \\
    &\leq |\xi-\eta|_E + (1+\delta) \bigg| \frac{\xi}{|\xi|_E} - \frac{\eta}{|\eta|_E} \bigg|_E \\
    &\leq |\xi-\eta|_E + 2\mbox{$\frac{R_E}{r_E}$} \\
    &\leq 3 \mbox{$\frac{R_E}{r_E}$} |\xi-\eta|_E \\
    &\leq 3 \big(\mbox{$\frac{R_E}{r_E}$}\big)^2 |\xi-\eta|.
\end{align*}
Next, we aim to verify the second assertion~\eqref{est:gdeltalipschitzzwei}. If~$|\eta|_E \leq 1$, we have~$\G(\eta)=0$ and, by proceeding similarly to before, achieve
\begin{align*}
    \frac{|\xi-\eta|}{|\G(\xi)-\G(\eta)|} = \frac{|\xi-\eta||\xi|_E}{(|\xi|_E-1)|\xi|} \leq \frac{1}{r_E} \frac{|\xi-\eta|}{|\xi|_E-1} \leq \frac{R_E}{r_E} \frac{|\xi|_E + |\eta|_E}{|\xi|_E-1} \leq  \frac{R_E}{r_E}\Big(1+\frac{2}{\delta} \Big) \leq 2 \frac{R_E}{r_E}\Big(1+\frac{1}{\delta} \Big). 
\end{align*}
The remaining case where~$|\eta|_E>1$ is treated as follows: due to the construction of~$\G$, this function is a one to one mapping from the set~$\R^n\setminus \overline{E} = \R^n\setminus\{x\in\R^n:|x|_E\leq 1\}$ to~$\R^n\setminus\{0\}$. Thus, its inverse mapping exists on~$\R^n\setminus\{0\}$ and, which is quickly verified, is given by~$\G^{-1}(y)=\frac{(|y|_E + 1)y}{|y|_E}$. Setting~$\tildexi\coloneqq \G(\xi)$ and~$\Tilde{\eta}\coloneqq \G(\eta)$, we infer through Lemma~\ref{lem:algineq} and~\eqref{mengeeradii} the following
\begin{align*}
    \frac{|\G^{-1}(\tildexi) - \G^{-1}(\Tilde{\eta}) |}{|\tildexi-\Tilde{\eta}|} &= \frac{\big|\tildexi-\Tilde{\eta} + \frac{\tildexi}{|\tildexi|_E}-\frac{\Tilde{\eta}}{|\Tilde{\eta}|_E} \big|}{|\tildexi-\Tilde{\eta}|} \\
    &\leq 1 + \frac{R_E}{|\tildexi-\Tilde{\eta}|} \bigg| \frac{\tildexi}{|\tildexi|_E} - \frac{\Tilde{\eta}}{|\Tilde{\eta}|_E} \bigg|_E \\
    &\leq 1 + \frac{R_E}{|\tildexi-\Tilde{\eta}|} \frac{2}{|\tildexi|_E} \frac{R_E}{r_E}|\tildexi-\Tilde{\eta}|_E \\
    &\leq 1+\Big(\frac{R_E}{r_E}\Big)^2\frac{2}{\delta} \\
    &\leq 3\Big(\frac{R_E}{r_E}\Big)^2 \Big(1+\frac{1}{\delta} \Big).
\end{align*}
\end{proof}


The following lemma represents one of the main preliminary tools and states a monotonicity estimate for~$\h_\epsilon$ that is an immediate consequence of the ellipticity assumption~\eqref{voraussetzung} for any~$\xi\in\R^n$ with~$1+\delta \leq |\xi|_E$ and~$x\in\Omega$.


\begin{mylem} \label{monotonicityapprox}
   Let~$\epsilon\in[0,1]$ and~$\delta>0$. Then there exists~$C=C(\delta)$, such that for any~$x\in\Omega$ and any~$\Tilde{\xi},\xi\in\R^n$ with~$|\xi|_E \geq 1+\delta $ there holds
   \begin{equation} \label{est:monotonicityapprox}
       \langle \h_\epsilon(x,\Tilde{\xi})-\h_\epsilon(x,\xi),\Tilde{\xi}-\xi\rangle \geq \bigg(\epsilon +  C(\delta)\frac{r_E(2|\xi|_E -(2+\delta))}{2|\xi|_E(R_E+r_E)} \bigg)|\Tilde{\xi}-\xi|^2. 
   \end{equation}
Moreover, for any~$x\in\Omega$ and any~$\tildexi,\xi\in\R^n$ the quantity on the left-hand side of~\eqref{est:monotonicityapprox} is non-negative. 
\end{mylem}


\begin{proof}
We commence by treating the first assertion, where we distinguish between two cases. Firstly, we assume~$|\Tilde{\xi}|_E \leq |\xi|_E$. For~$s\in[0,1]$, we denote~$\xi_s\coloneqq(1-s)\xi+s\Tilde{\xi}=\xi-s(\xi-\tildexi)$ and compute
\begin{align*}
     \langle \h_\epsilon(x,\Tilde{\xi})-\h_\epsilon(x,\xi),\Tilde{\xi}-\xi\rangle &= \int_{0}^{1}   \Big \langle \frac{\d}{\ds} \h(x,\xi_s),\Tilde{\xi}-\xi \Big\rangle \,\ds +\epsilon|\Tilde{\xi}-\xi|^2.
\end{align*}
If~$|\xi_s|_E\geq 1+\frac{\delta}{2}$ for all~$s\in[0,1]$, then we have that~$\F$ is of class~$C^2$ on~$[\tildexi,\xi]\coloneqq\{ (1-s)\xi+s\tildexi:s\in[0,1] \}$. Moreover, in this case we may differentiate~$\h$ and use our ellipticity condition~\eqref{voraussetzung} to obtain
 \begin{align*}
        \langle \h_\epsilon(x,\Tilde{\xi})-\h_\epsilon(x,\xi),\Tilde{\xi}-\xi\rangle &= \displaystyle\int_{0}^{1} \langle \nabla^2\F(x,\xi_s)(\Tilde{\xi}-\xi),\Tilde{\xi}-\xi\rangle \,\ds +\epsilon|\Tilde{\xi}-\xi|^2  \\
        &\geq \lambda(\delta) |\Tilde{\xi}-\xi|^2  +\epsilon|\Tilde{\xi}-\xi|^2.
    \end{align*}
   Otherwise, the reverse triangle inequality~\eqref{est:minkowskitriangle} and assumption~$|\Tilde{\xi}|_E \leq |\xi|_E$ imply
   \begin{align*}
        |\xi_s|_E 
        &\geq (1-s)|\xi|_E - s|-\tildexi|_E \\
        &\geq (1-s)|\xi|_E - \mbox{$\frac{s}{r_E}$} |\tildexi| \\
        &\geq (1-s)|\xi|_E - s \mbox{$\frac{R_E}{r_E}$} |\xi|_E \\
        &\geq 1+\textstyle{\frac{\delta}{2}} 
    \end{align*} 
  at least for all~$s\in\Big[0, \frac{r_E(2|\xi|_E -(2+\delta))}{2|\xi|_E (R_E+r_E)} \Big]$. To simplify notation, we set~$\Tilde{s}\coloneqq \frac{r_E(2|\xi|_E -(2+\delta))}{2|\xi|_E (R_E+r_E)}$. If there exists  at least one~$\hat{s}\in(\Tilde{s},1]$ such that~$|\xi_{\hat{s}}|_E\leq 1$, we may without any loss of generality assume that~$|\xi_s|_E>1$ holds for all~$s\in[\Tilde{s},\hat{s})$ as well as~$|\xi_s|_E\leq 1$ for all~$s\in[\hat{s},1]$. This simplification indeed poses no restrictions, since in the case where~$|\xi_s|_E>1$ for some~$s\in(\hat{s},1]$, we follow the subsequent arguments verbatim and discard any possible contributions from additional applications of the ellipticity condition~\eqref{voraussetzung}. For any parameter~$\tau\in(0,\hat{s}-\Tilde{s})$, there exists~$C=C(\delta,\tau)>0$ with~$|\xi_{\Tilde{s}+\tau}|_E \geq 1+C(\delta,\tau)$, and we rewrite the integral term above as follows
 \begin{align*}
        \langle \h_\epsilon(x,\Tilde{\xi}) -\h_\epsilon(x,\xi),\Tilde{\xi}-\xi\rangle &= \displaystyle\int_{0}^{1} \Big\langle \frac{\d}{\ds} \h(x,\xi_s),\Tilde{\xi}-\xi \Big\rangle \,\ds +\epsilon|\Tilde{\xi}-\xi|^2 \\
        &= \displaystyle\int_{0}^{\Tilde{s}} \nabla^2\F(x,\xi_s)(\Tilde{\xi}-\xi),\Tilde{\xi}-\xi\rangle \,\ds + \displaystyle\int_{\Tilde{s}}^{\hat{s}-\tau} \nabla^2\F(x,\xi_s)(\Tilde{\xi}-\xi),\Tilde{\xi}-\xi\rangle \,\ds \\
        &\quad + \displaystyle\int_{\hat{s}-\tau}^{1} \Big\langle \frac{\d}{\ds} \h(x,\xi_s),\Tilde{\xi}-\xi \Big\rangle \,\ds
        +\epsilon|\Tilde{\xi}-\xi|^2 \\
        &\eqqcolon \foo{I}+\foo{II}+\foo{III} + \epsilon|\Tilde{\xi}-\xi|^2
    \end{align*}
with the obvious meaning of the quantities~$\foo{I}-\foo{III}$. The first term is bounded below by an application of our ellipticity estimate~\eqref{voraussetzung}, where we exploit the fact that~$|\xi_s|_E\geq 1+\mbox{$\frac{\delta}{2}$}$ for any~$s\in[0,\Tilde{s}]$, which yields the bound
    \begin{align*}
       \foo{I} &= \displaystyle\int_{0}^{\Tilde{s}} \langle \nabla^2\F(x,\xi_s)(\Tilde{\xi}-\xi),\Tilde{\xi}-\xi\rangle \,\ds  \geq \lambda(\delta)\frac{r_E(2|\xi|_E -(2+\delta))}{2|\xi|_E(R_E+r_E)} |\Tilde{\xi}-\xi|^2 .
    \end{align*}
 Similarly, again due to the ellipticity condition~\eqref{voraussetzung}, for the second term we have
    \begin{align*}
       \foo{II} &= \displaystyle\int_{\Tilde{s}}^{\hat{s}-\tau} \nabla^2\F(x,\xi_s)(\Tilde{\xi}-\xi),\Tilde{\xi}-\xi\rangle \,\ds  \geq C(\delta,\tau)(\hat{s}-\Tilde{s}-\tau)\Tilde{\xi}-\xi|^2 \geq 0. 
    \end{align*}
    For the last term~$\foo{III}$, we use the fact that~$\h$ is continuous across the whole of~$\R^n$ and that~$\h$ vanishes within the set~$\{x\in\R^n:|x|_E\leq 1\}$, and also the assumption~$|\xi_s|_E\leq 1$ for any~$s\in[\hat{s},1]$. This yields
    \begin{align*}
        \foo{III} = \displaystyle\int_{\hat{s}-\epsilon}^{1} \Big\langle \frac{\d}{\ds} \h(x,\xi_s),\Tilde{\xi}-\xi \Big\rangle \,\ds = \langle \h(\tildexi)-\h(\xi_{\hat{s}-\epsilon}),\tildexi-\xi \rangle = - \langle \h(\xi_{\hat{s}-\epsilon}),\tildexi-\xi \rangle \to 0
    \end{align*}
 in the limit~$\epsilon\downarrow 0$. Overall, this yields the claimed monotonicity estimate~\eqref{est:monotonicityapprox}. In the other case~$|\Tilde{\xi}|_E>|\xi|_E$, we infer similarly
 \begin{align*}
     |\xi_s|_E &\geq s |\tildexi|_E - (1-s)|-\xi|_E \\
     &\geq s |\xi|_E - (1-s) \mbox{$\frac{R_E}{r_E}$} |\xi|_E \\
     &= s \mbox{$\frac{R_E+r_E}{r_E}$} |\xi|_E - \mbox{$\frac{R_E}{r_E}$} |\xi|_E \\
     &\geq 1+\mbox{$\frac{\delta}{2}$}
 \end{align*}
for all~$s\in \Big[\frac{1}{2|\xi|_E (R_E+r_E)}(r_E(2+\delta)+2R_E|\xi|_E),1\Big]$. By following the approach taken for the former case~$|\tildexi|_E\leq |\xi|_E$, we obtain the very same lower bound, which finishes the proof. \,\\
The second assertion of the non-negativity of the left-hand side in~\eqref{est:monotonicityapprox} follows immediately: either by leveraging the convexity of $\mathcal{F}$ with respect to the second variable, as stated in~$\eqref{fregularity}_2$, or by distinguishing the cases where $|\tilde{\xi}|_E \geq 1 + \delta$ or $|\xi|_E \geq 1 + \delta$ for some $\delta > 0$, and then applying the first part of the lemma, respectively by noting that the left-hand side of~\eqref{est:monotonicityapprox} vanishes in the trivial case where $|\tilde{\xi}|_E \leq 1$ and $|\xi|_E \leq 1$.
    
\end{proof}


As a consequence of Lemma~\ref{gdeltalem} and Lemma~\ref{monotonicityapprox}, we obtain the next auxiliary result.

\begin{mylem} \label{lemgdeltakvgz}
    Let~$\epsilon\in[0,1]$ and~$\delta>0$. For any~$x\in\Omega$ and~$\Tilde{\xi}, \xi \in\R^n$ there holds
    \begin{align*}
        \epsilon|\Tilde{\xi}-\xi|^2 + |\G_\delta(\tildexi) - \G_\delta(\xi)|^2 \leq  C(\delta,R_E,r_E) \langle \h_\epsilon(x,\tildexi)- \h_\epsilon(x,\xi), \tildexi-\xi\rangle.
    \end{align*}
\end{mylem}
\begin{proof}
    In the case where~$|\xi|_E,|\Tilde{\xi}|_E\leq 1+\delta$, we have~$\G_\delta(\xi)=\G_\delta(\Tilde{\xi})=0$ and the claimed estimate readily follows by an application of Lemma~\ref{monotonicityapprox}. Next, we treat the case where either~$|\xi|_E > 1+\delta$ or~$|\tildexi|_E > 1+\delta$. An application of Lemma~\ref{gdeltalem} implies
    \begin{equation*}
        |\G_\delta(\tildexi) - \G_\delta(\xi)|^2\leq C(R_E,r_E) |\tildexi-\xi|^2.
    \end{equation*}
    Consequently, Lemma~\ref{monotonicityapprox} yields the claimed inequality
    \begin{align*}
        \epsilon|\tildexi - \xi|^2 + |\G_\delta(\tildexi) - \G_\delta(\xi)|^2 &\leq \epsilon |\tildexi - \xi|^2 + C(R_E,r_E)|\tildexi - \xi|^2  \\
        &\leq \epsilon |\tildexi - \xi|^2 + C(\delta,R_E,r_E) \frac{r_E(2|\xi|_E -(2+\delta))}{2|\xi|_E(R_E+r_E)} |\tildexi - \xi|^2  \\
        &\leq C(\delta,R_E,r_E) \langle \h_\epsilon(x,\tildexi) - \h_\epsilon(x,\xi),\Tilde{\xi}-\xi\rangle.
    \end{align*}
\end{proof}


\begin{mylem} \label{bilinearelliptic}
    Let~$\epsilon\in[0,1]$ and~$\delta>0$. For any~$x\in\Omega$ and any~$\xi,\eta\in\R^n$ with~$1+\delta \leq |\xi|_E \leq \delta^{-1}$, where~$E\subset\R^n$ with~$0\in\inn{E}$ denotes the bounded and convex set of degeneracy of~$\F$, there holds
    \begin{equation*}
        (\epsilon+\lambda(\delta))|\eta|^2 \leq \mathcal{B}_\epsilon(x,\xi)(\eta,\eta) \leq (\epsilon + \Lambda(\delta)) |\eta|^2.
    \end{equation*}
\end{mylem}
\begin{proof}
    The estimate is an immediate consequence of the ellipticity assumption~\eqref{voraussetzung} and the definition of the bilinear form~$\B_\epsilon$.  
\end{proof}


\begin{remark} \upshape
 The estimates from Lemma~\ref{monotonicityapprox} and Lemma~\ref{bilinearelliptic} both continue to hold true without the positive constants~$C(\delta)$,~$\lambda(\delta)$ and~$\Lambda(\delta)$ in the case where~$|\xi|_E \leq 1$.    
\end{remark}


We recall the following Poincar\'{e}-type inequality, where we refer the reader to~\cite[Lemma~2]{santambrogio2010continuity}.

\begin{mylem} \label{poincarelem}
    Let~$\Omega\subset\R^n$ be a bounded domain and~$u\in W^{1,2}_0(\Omega)$. There exists a positive constant~$C=C(n)$, such that
    $$\int_{\Omega}u^2\,\dx \leq C |\{x\in \Omega: |u(x)|>0\}|^{\frac{2}{n}}\int_{\Omega}|Du|^2\,\dx.$$
\end{mylem}


The subsequent iteration lemma allows a technique of reabsorbing certain quantities and stems from~\cite[Lemma~6.1]{giusti2003direct}.


\begin{mylem} \label{iterationlem}
       Let $\phi$ be a bounded, non-negative function on $0\leq R_0< R_1$ and assume that for $R_0\leq \rho < r \leq R_1$, there holds
      $$\phi(\rho) \leq \eta \phi(r) + \frac{A}{(r-\rho)^{\alpha}} + \frac{B}{(r-\rho)^{\beta}} + C$$
      for some constants $A,B,C,\alpha\geq \beta\geq 0$, and $\eta\in(0,1)$. Then, there exists a constant $\Tilde{C}=\Tilde{C}(\eta,\alpha)$, such that for all $R_0\leq \rho_0<r_0\leq R_1$ there holds
      $$\phi(\rho_0) \leq \Tilde{C}\bigg( \frac{A}{(r_0-\rho_0)^{\alpha}} + \frac{A}{(r_0-\rho_0)^{\alpha}} + C \bigg).$$
\end{mylem}


Lastly, another expedient and well-known lemma is stated, which concerns the geometric convergence of sequences. For a proof, the reader is referred to~\cite[Chapter~I, Lemma~4.1]{dibenedetto1993degenerate}.

\begin{mylem} \label{geometriclem}
Let~$(Y_i)_{i\in\N_0}\subset \R_{\geq 0}$ be a sequence of non-negative numbers, satisfying the recursive inequality
$$Y_{i+1} \leq C b^i Y^{1+\kappa}_i$$
for any~$i\in\N_0$, where~$C,\kappa >0$ and~$b>1$ denote positive constants. If there holds
$$Y_0 \leq C^{-\frac{1}{\kappa}}b^{-\frac{1}{\kappa^2}},$$
    then we have~$Y_i \to 0$ as~$i\to\infty$.
\end{mylem}


\section{Proof of Theorem~\ref{hauptresultat}} \label{sec:holder}
The aim of this section is the proof of the main regularity result that is Theorem~\ref{hauptresultat}. This will be achieved in a series of steps. We commence by regularizing our equation~\eqref{pde} in order to obtain an elliptic approximating equation. Moreover, due to the fact that any weak solution to~\eqref{pde} is \textit{a priori} assumed to be locally Lipschitz continuous, we modify the function~$\F(x,\xi)$ present in the diffusion term of our equation in order to obtain a revised function whose derivative satisfies quadratic growth with respect to the gradient variable~$\xi\in\R^n$.


\subsection{Regularizing the equation} \label{sec:regularizing}
We start off by performing a regularization of equation~\eqref{pde} through constructing an approximation of~$\F$ that satisfies quadratic growth for large~$|\xi|$. Moreover, a further modification leads to a redefined function~$\F_\epsilon$ that exhibits quadratic growth on the whole of~$\R^n$ with an ellipticity constant depending on the parameter~$\epsilon \in (0,1]$. We would like to note that a similar construction has already been employed in~\cite[Proof of Theorem~1.1]{colombo2017regularity}. Let~$u\in W^{1,\infty}_{\loc}(\Omega)$ denote a local weak solution to~\eqref{pde} and consider a ball~$B_R=B_R(y_0)\Subset\Omega$. In what follows, we will always omit the center~$y_0$ of the ball~$B_R=B_R(y_0)$ in any given quantity. For the further course of this section, we denote
\begin{equation} \label{schrankeu}
   K \coloneqq \|Du\|_{L^\infty(B_R)}<\infty,
\end{equation}
which is indeed finite due to the Lipschitz regularity of~$u$ in~$B_R$. Moreover, we set 
\begin{equation} \label{schrankeF}
    \Tilde{K} \coloneqq \|\F\|_{L^\infty(B_R \times B_{ K+2R_E 
    })} <\infty, \qquad L\coloneqq \Tilde{K} + 1 <\infty,
\end{equation}
where~$R_E>0$ denotes the radius from~\eqref{mengeeradii}. The additional constant~$1$ is selected arbitrarily and could be substituted with any real number that is greater than~$0$. Next, let~$\Psi\colon \R_{\geq 0}\to \R_{\geq 0}$ denote a non-negative and smooth cut-off function of class~$C^2(\R_{>0},\R_{> 0})$, such that 
\begin{equation*}
    \Psi(t)=t \quad \mbox{for~$t\in[0,\Tilde{K}]$}, \qquad\quad \Psi(t)=L \quad \mbox{for~$t\in[L,\infty)$},
\end{equation*}
that satisfies the estimate
\begin{equation} \label{psischranken}
    |\Psi'(t)| + |\Psi''(t)| \leq C_\Psi \qquad \mbox{for all~$t\in\R_{> 0}$}
\end{equation}
for some~$C_\Psi>0$. For~$x\in B_R$ and~$\xi\in\R^n$, we define a truncated version of~$\F$ by
\begin{equation} \label{tildef}
    \Tilde{\F}(x,\xi) \coloneqq \Psi(\F(x,\xi)),
\end{equation}
and note that due to this construction, there holds
\begin{equation} \label{tildefeeval}
    \F(x,Du(x))=\Psi(\F(x,Du(x)))=\Tilde{\F}(x,Du(x))
\end{equation}
for any~$x\in B_R$. Moreover, we have
\begin{align} \label{approxderivative}
        \begin{cases}
   \nabla \Tilde{\F}(x,\xi) =  \nabla\F(x,\xi) \Psi'(\F(x,\xi)),\quad \xi\in\R^n & \\
   \nabla^2 \Tilde{\F}(x,\xi) = (\nabla\F(x,\xi)\otimes\nabla\F(x,\xi)) \Psi''(\F(x,\xi)) +  \nabla^2\F(x,\xi) \Psi'(\F(x,\xi)),\quad \xi\in\R^n\setminus \overline{E} &
        \end{cases}
  \end{align}
for any~$x\in B_R$. We note that the truncated function~$(x,\xi)\mapsto\Tilde{\F}(x,\xi)$ remains constant on the set~$B_R\times (\R^n \setminus B_{N} )$ for some
\begin{align} \label{N}
    N \ge K+2R_E.
\end{align}  
This property follows from the continuity of $\mathcal F$ and the fact that the partial mapping~$\xi \mapsto \nabla^2\F(x,\xi)$ is positive definite on~$\R^n\setminus \overline{E}$ for any~$x \in B_R$, yielding strict convexity of~$\xi \mapsto \F(x,\xi)$ on~$\R^n\setminus \overline{E}$ for any~$x\in B_R$. Thus,~$\xi \mapsto \F(x,\xi)$ is in particular coercive on~$\R^n$ for any~$x\in B_R$. In particular, for the function~$\Tilde{\F}$ there hold the following estimates
\begin{align} \label{derivativebounds}
        \begin{cases}
   \|\nabla \Tilde{\F}\|_{L^\infty(B_R \times \R^n)} \leq C_\Psi\|\nabla \F\|_{L^\infty(B_R \times B_N)}  & \\
   \|\nabla^2 \Tilde{\F}\|_{L^\infty(B_R \times (\R^n \setminus E_\delta))} \leq C_\Psi\|\nabla \F\|^2_{L^\infty(B_R \times B_N )} + C_\Psi\|\nabla^2 \F\|_{L^\infty(B_R \times (B_N \setminus E_\delta))} &
        \end{cases}
  \end{align}
for any~$\delta>0$, where we recall that~$E_\delta$ denotes the outer parallel set of~$E$ according to~\eqref{Edelta}. Thus, the preceding estimates for the essential supremum of both~$\nabla\Tilde{\F}$ and~$\nabla^2\Tilde{\F}$ can be controlled by the constant~$C_\Psi$ and the~$L^\infty$-bounds for~$\nabla\F$ and~$\nabla^2\F$ respectively. We set 
\begin{align} \label{schrankehessian}
    C_{\F} \coloneqq \max\{ \|\nabla^2 \Tilde{\F}\|_{L^\infty(B_R \times (\R^n \setminus B_{K+R_E}))}, 1 \} < \infty
\end{align}
Next, let~$\Phi\in C^2(\R^n)$ denote a convex function that is subject to
\begin{align} \label{convexfunction}
       \begin{cases}
       \Phi(\xi)=0  & \mbox{for any~$|\xi| \leq K+R_E$}, \\
       |\nabla\Phi(\xi)| \leq (2C_{\F} + 1)|\xi|  & \mbox{for any~$\xi\in \R^n$}, \\
   \langle \nabla^2 \Phi (\xi) \eta,\eta \rangle \geq (C_{\F}+1) |\eta|^2 & \mbox{for any~$|\xi|\geq K+2 R_E$}, \\
        |\nabla^2\Phi(\xi)| \leq 2C_\F+1 & \mbox{for any~$\xi \in \R^n$}, 
        \end{cases}
  \end{align}
 for any~$\eta\in\R^n$. We now consider the mapping
\begin{align*}
    \hat{\F}(x,\xi) \coloneqq \Tilde{\F}(x,\xi) + \Phi(\xi) \qquad \mbox{for~$x\in B_R,\, \xi\in\R^n$}.
\end{align*}
By construction, we have that~$\R^n\ni\xi \mapsto \hat{\F}(x,\xi)$ is convex across~$\R^n$. Indeed, on~$B_R\times B_{K+2R_E}$ there holds~$\Tilde{\F}(x,\xi)=\F(x,\xi)$, such that, due to the fact that~$\Phi$ is convex across~$\R^n$,~$\hat{\F}$ is convex. On the complement
$$ B_R\times (\R^n\setminus B_{K+2R_E}) $$
the bound~\eqref{schrankehessian} and the properties~\eqref{convexfunction} yield the ellipticity estimate
\begin{align} \label{hatapproxellipticity}
    \langle \nabla^2 \hat{\F}(x,\xi) \eta,\eta \rangle &= \langle \nabla^2 \Tilde{\F}(x,\xi) \eta,\eta \rangle + \langle \nabla^2 \Phi(\xi) \eta,\eta \rangle \\
    &\geq -C_{\F}|\eta|^2 + (C_{\F}+1)|\eta|^2 = |\eta|^2 \nonumber
\end{align}
for any~$\eta\in\R^n$. Moreover, there holds the bound
\begin{align} \label{hatapproxbound}
    |\nabla^2\hat{\F}(x,\xi)| \leq 3C_{\F}+1 
\end{align}
for~$x\in B_R$ and~$\xi \in \R^n \setminus B_{K+R_E}$. We claim that~$\hat{\F}$ satisfies the set of structure conditions~\eqref{fregularity}. The first three conditions~$\eqref{fregularity}_1$ -- $\eqref{fregularity}_3$ certainly hold true due to the construction of~$\hat{\F}$. The fourth condition~$\eqref{fregularity}_4$ can be verified as follows: let~$(x,\xi)\in B_R\times\R^n$ be arbitrary. By exploiting~$\eqref{fregularity}_4$ for~$\F$, the construction of~$\hat{\F}$, the bounds~\eqref{betragminkowski} and~\eqref{psischranken}, and the fact that~$\Phi$ is independent of~$x$, there holds
\begin{align} \label{hatflipschitz}
    |\partial_{x_i}\nabla\hat\F(x,\xi)| &\leq |\partial_{x_i}\nabla\F(x,\xi) \Psi'(\F(x,\xi))| + |\nabla\F(x,\xi)\partial_{x_i}\F(x,\xi)\Psi''(x,\xi)| \\
    &\leq C(C_\Psi,N,r_E) + C_\Psi \underbrace{\| \partial_{x}\F\|_{L^\infty(B_R\times B_{N})}}_{\leq C(C_\Psi,N,r_E)} \underbrace{\| \nabla \F\|_{L^\infty(B_R\times B_{N})}}_{\eqqcolon \hat{C}_\F} \nonumber\\
    &\leq C(\hat{C}_\F,C_\Psi,N,r_E) \nonumber
\end{align}
for any~$i=1,\ldots,n$, which verifies that~$\hat{\F}$ satisfies condition~$\eqref{fregularity}_4$ with constant~$C(\hat{C}_\F,C_\Psi,N,r_E)$ for any $x\in B_R$ and any $\xi\in\R^n$. In turn, we exploited the fact that~$\Tilde{\F}(x,\xi)$ is constant for any~$(x,\xi)\in B_R\times (\R^n\setminus B_N)$, according to~\eqref{N}, and used the mean value theorem to bound the term involving~$\partial_x\F$. Finally, we aim to verify the ellipticity condition~\eqref{voraussetzung} for~$\hat{\F}$. This, however, follows from the preceding reasoning provided in~\eqref{hatapproxellipticity} and~\eqref{hatapproxbound}, as well as the fact that there holds
$$ \hat\F(x,\xi)=\Tilde\F(x,\xi)+\Phi(\xi)=\F(x,\xi)+\Phi(\xi) \qquad\mbox{in~$B_R\times B_{K+2R_E}$}, $$
concluding that for any~$\delta>0$ there exist positive constants~$\hat{\lambda}=\hat{\lambda}(C_\F,\delta)$ and~$\hat{\Lambda}=\hat{\Lambda}(C_\F,\delta)$ with~$0<\hat{\lambda}\leq\hat{\Lambda}$, such that there holds
\begin{align} \label{hatfvoraussetzung}
     \qquad \hat{\lambda}(C_\F,\delta) |\eta|^2 \leq \langle \nabla^2 \hat{\F}(x,\xi) \eta,\eta \rangle \leq \hat{\Lambda}(C_\F,\delta) |\eta|^2\qquad\text{for any~$1+\delta \leq |\xi|_E \leq \delta^{-1}$}
\end{align}
for any~$x\in B_R$. Since~$u$ is a weak solution to equation~\eqref{pde},~$u$ is also a weak solution to
\begin{equation*}
    \divv \hat{\h}(x,Du) = f  \qquad\mbox{in~$B_R$},
\end{equation*}
since the set of points~$B_R \cap \{|Du|\geq K\}$ is a subset of~$B_R$ of measure zero and there holds~$\hat\F=\F$ in~$B_R\times B_K$ according to~\eqref{schrankeu} and the definition of~$\hat{\F}$. Here, we abbreviated~$\hat{\h}(x,\xi) \coloneqq \nabla \hat{\F}(x,\xi)$ for any~$x\in B_R$ and~$\xi\in\R^n$. \,\\

The above redefined equation however still does not admit uniform ellipticity on the whole of~$\R^n$, necessitating an additional approximation in order to obtain an elliptic equation on~$\R^n$ that avoids degeneracy within the set~$E\subset\R^n$. For an approximating parameter~$\epsilon\in[0,1]$, we define
\begin{equation} \label{psifapprox}
    \hat{\F}_\epsilon(x,\xi) \coloneqq \hat{\F}(x,\xi) + \epsilon \textstyle{\frac{1}{2}} |\xi|^2 \qquad\mbox{for~$x\in B_R,\,\xi\in\R^n$}
\end{equation}
and note that~$\hat{\F}_\epsilon$ is indeed elliptic with respect to~$\xi$ on the whole of~$\R^n$ in the case where~$\epsilon\in(0,1]$, with the ellipticity constant additionally depending on the parameter~$\epsilon$. Moreover, by exploiting the definition of~$\hat{\F}_\epsilon$,~\eqref{approxderivative}, and also~\eqref{convexfunction}, we have that~$\hat{\F}_\epsilon$ satisfies the desired quadratic growth on~$\R^n$, i.e. there holds
\begin{align} \label{quadraticgrowth}
    |\nabla \hat{\F}_\epsilon(x,\xi)| \leq (\epsilon+C_\Psi\|\nabla \F \|_{L^\infty(B_R \times B_{N})}+2C_\F+1)(1+|\xi|) \leq C(C_\F,C_\Psi,\hat{C}_\F)(1+|\xi|)
\end{align}
for any~$x\in B_R$ and~$\xi\in\R^n$, where we additionally used that~$\epsilon\in(0,1]$. For simplicity of notation, we will omit any dependence on the constant~$C_\Psi$ from now on. \,\\
Moreover, for any~$|\xi|>K+2R_E$ we have, according to~\eqref{convexfunction}, the following quadratic growth 
\begin{align} \label{quadraticgrowthohnekonst}
    |\nabla \hat{\F}_\epsilon(x,\xi)| \leq (2C_\F+1) |\xi|. 
\end{align}
We continue by considering the unique weak solution
$$u_\epsilon \in u+W^{1,2}_0(B_R)$$
to the Dirichlet problem
\begin{align} \label{approx}
    \begin{cases}
       \divv \hat{\h}_\epsilon(x,Du_\epsilon) = f & \mbox{in~$B_R$} \\
        \,\qquad\quad\quad\quad\,\, u_\epsilon =u & \mbox{on~$\partial B_R$},
    \end{cases}
\end{align}
where we abbreviated~$\hat{\h}_\epsilon(x,\xi)\coloneqq \nabla\hat{\F}_\epsilon(x,\xi)$ for~$(x,\xi)\in B_R\times\R^n$, according to~\eqref{Gapprox}. Similarly, we will write~$\hat{\B}_\epsilon(x,\xi)(\eta,\zeta) \coloneqq \langle \nabla^2\hat{\F}_\epsilon(x,\xi)\eta,\zeta\rangle$ for~$\xi\in\R^n\setminus \overline{E}$ and~$\eta,\zeta\in\R^n$. Due to the quadratic growth~\eqref{quadraticgrowth} of the redefined equation, it is sufficient to consider the approximating solutions~$u_\epsilon$ to be \textit{a priori} of class~$W^{1,2}(B_R)$. The respective weak formulation of~\eqref{approx} results in
\begin{equation} \label{weakformapprox}
    \int_{B_R} \langle \hat{\h}_\epsilon(x,Du_\epsilon), D\phi\rangle  \,\dx= - \int_{B_R} f \phi\,\dx
\end{equation}
for any test function~$\phi\in C^{\infty}_0(B_R)$. We recall once more that the approximation~$\xi \mapsto \hat{\F}_\epsilon(x,\xi)$ is elliptic on~$\R^n$ for any~$x\in B_R$, with an ellipticity constant that depends on the parameter~$\epsilon\in(0,1]$ in general. This property enables us to obtain the higher regularity~$u_\epsilon \in W^{1,\infty}_{\loc}(B_R) \cap W^{2,2}_{\loc}(B_R)$ for the approximating solutions~$u_\epsilon$. Additionally, we obtain a quantitative local~$L^\infty$-estimate for the gradient and a local~$W^{2,2}$-estimate, both of which are stated in Proposition~\ref{approxregularityeins} and Proposition~\ref{approxregularityzwei}. Indeed, we immediately obtain from the fact that the boundary datum satisfies~$u\in W^{1,\infty}(B_R)$ and the maximum principle in~\cite[Theorem~8.1]{gilbargtrudinger}, that the solutions~$u_\epsilon$ are bounded within~$B_R$ by~$u$, i.e. for any~$\epsilon\in(0,1]$ there holds~$u_\epsilon\in L^\infty(B_R)$ with the quantitative bound
\begin{align} \label{uepsbeschränkt}
    \|u_\epsilon\|_{L^\infty(B_R)} \leq \|u\|_{L^\infty(B_R)}<\infty.
\end{align} 

Before stating the quantitative local~$L^\infty$-gradient and the~$W^{2,2}$-estimate for~$u_\epsilon$ on~$B_{R}$, the subsequent energy estimate turns out expedient, which states a uniform~$L^2$-bound with respect to the parameter~$\epsilon\in(0,1]$ for the approximating solutions~$u_\epsilon$.


\begin{mylem} \label{energybound}
    Let~$\epsilon\in(0,1]$. There holds the energy estimate
    \begin{align*}
        \int_{B_R} |Du_\epsilon|^2\,\dx &\leq C \int_{B_R} \big(1+|Du|^2\big)\,\dx
    \end{align*}
    with a constant $C=C(C_\F,\hat{C}_\F,\|f\|_{L^n(B_R)},n,R_E,r_E)$.
\end{mylem}
\begin{proof}

We start by testing the weak form~\eqref{weakformapprox} with~$\phi = u_\epsilon - u$. It can be verified through an approximation argument that this choice of~$\phi$ is indeed an admissible test function. By plugging~$\phi$ into~\eqref{weakformapprox}, there holds
\begin{align} \label{energytermeins}
    \int_{B_R} \langle \hat{\h}_\epsilon(x,Du_\epsilon), Du_\epsilon\rangle \,\dx
    &= \int_{B_R} \langle \hat{\h}_\epsilon(x,Du_\epsilon), Du\rangle \,\dx - \int_{B_R} f (u_\epsilon-u)\,\dx. 
\end{align}
In order to bound the term on the left-hand side further below, we note that due to Lemma~\ref{monotonicityapprox}, the quantity is non-negative. Thus, we pass to the subset of points~$B_R\cap\{|Du_\epsilon|_E \geq \frac{3}{2}\}$ and employ Lemma~\ref{monotonicityapprox} as well as~\eqref{betragminkowski}, where we recall~\eqref{hatfvoraussetzung}, to obtain that
\begin{align*}
    \int_{B_R}  \langle &\hat{\h}_\epsilon(x,Du_\epsilon), Du_\epsilon \rangle \,\dx \\
    &\geq  \int_{B_R} \bigg(\Tilde{C}(C_\F,R_E,r_E) \frac{2|Du_\epsilon|_E-\frac{5}{2}}{|Du_\epsilon|_E} |Du_\epsilon|^2 \bigchi_{ \{|Du_\epsilon|_E\geq \frac{3}{2}\} } +\epsilon|Du_\epsilon|^2 \bigg)\,\dx \\
    &\geq \int_{B_R} \bigg(\Tilde{C}(C_\F,R_E,r_E) \frac{2|Du_\epsilon|_E-\frac{5}{2}}{|Du_\epsilon|_E} (|Du_\epsilon|_E-\textstyle{\frac{3}{2}})^2_+ +\epsilon|Du_\epsilon|^2 \bigg)\,\dx \\
    & \geq \int_{B_R} (\Tilde{C}(C_\F,R_E,r_E) (|Du_\epsilon|_E-\textstyle{\frac{3}{2}})^2_+ +\epsilon|Du_\epsilon|^2 )\,\dx.
\end{align*}
Next, the first term on the right-hand side of the preceding inequality is bounded above with the Cauchy-Schwarz inequality,~\eqref{betragminkowski},
the gradient bound~\eqref{quadraticgrowth}, and also Young's inequality, by 
\begin{align} \label{energytermzwei}
    \int_{B_R}  \langle \hat{\h}_\epsilon(x,Du_\epsilon), Du \rangle \,\dx &\leq \int_{B_R}  \big(C (1+|Du_\epsilon|)|Du| +\epsilon|Du_\epsilon||Du| \big) \,\dx  \\
    &\leq \int_{B_R} \bigg( \mbox{$\frac{1}{4}$} \Tilde{C} (|Du_\epsilon|_E -\mbox{$\frac{3}{2}$})^2_+ +C|Du|^2 +C \bigg) \,\dx \nonumber \\
    &\quad + \mbox{$\frac{1}{2}$} \epsilon \int_{B_R} \big(|Du_\epsilon|^2\,\dx +|Du|^2 \big)\,\dx. \nonumber
\end{align}
with a constant~$C=C(C_\F,\hat{C}_\F,R_E,r_E)$. 
In order to treat the term involving the datum~$f$ in~\eqref{energytermeins}, we employ Hölder's inequality, Sobolev's inequality,~\eqref{betragminkowski}, and Young's inequality, which yields the bound
\begin{align} \label{energytermf}
    \bigg| \int_{B_R}  f(u_\epsilon-u)\,\dx \bigg| &\leq  \int_{B_R}  |f| |u_\epsilon-u| \,\dx  \\ 
    &\leq \|f\|_{L^n(B_R)}\bigg(\int_{B_R}|u_\epsilon-u|^{\frac{n}{n-1}}\,\dx\bigg)^{\frac{n-1}{n}} \nonumber \\
    &\leq C(n) \|f\|_{L^n(B_R)}\int_{B_R}|Du_\epsilon-Du|\,\dx \nonumber \\
    &\leq C(n) \|f\|_{L^n(B_R)}\int_{B_R} (|Du_\epsilon| + |Du| )\,\dx \nonumber \\
    &\leq C(n,R_E) \|f\|_{L^n(B_R)}\int_{B_R} (|Du_\epsilon|_E-\mbox{$\frac{3}{2}$})_+ + 1) \nonumber \\
    &\quad + C \|f\|_{L^n(B_R)}\int_{B_R}(1+|Du|^2)\,\dx \nonumber \\
    &\leq \mbox{$\frac{1}{4}$} \Tilde{C} \int_{B_R}  (|Du_\epsilon|_E-\mbox{$\frac{3}{2}$})^2_+ \,\dx + C(n,R_E) \|f\|_{L^n(B_R)}\int_{B_R}(1+ |Du|^2)\,\dx. \nonumber
\end{align}
Inserting the estimates~\eqref{energytermzwei} and~\eqref{energytermf} into~\eqref{energytermeins} and reabsorbing the quantities involving~$(|Du_\epsilon|_E-\mbox{$\frac{3}{2}$})^2_+$ into the left-hand side, we obtain 
\begin{align*} 
     \int_{B_R}  \big((|Du_\epsilon|_E-\mbox{$\frac{3}{2}$})^2_+ +\epsilon|Du_\epsilon|^2\big)\,\dx \leq  \epsilon C \int_{B_R} |Du|^2 \,\dx + C\|f\|_{L^n(B_R)}\int_{B_R}(1+ |Du|^2)\,\dx
\end{align*}
with~$C=C(C_\F,\hat{C}_\F,R_E,r_E,n)$. From this estimate, the claimed inequality is an immediate consequence due to~\eqref{betragminkowski}. 
\end{proof}


We continue by establishing the fact that~$u_\epsilon \in W^{2,2}_{\loc}(B_R)$, stated in the subsequent proposition.


\begin{myproposition} \label{approxregularityeins}
    Let~$\epsilon\in(0,1]$. Then, there holds
    \begin{equation*}
        u_\epsilon \in W^{2,2}_{\loc}(B_R).
    \end{equation*}
    Moreover, for any~$B_{2\rho}(x_0)\Subset B_R$ with~$\rho\in(0,1]$ there holds the quantitative~$W^{2,2}$-estimate
     \begin{equation} \label{approxw22}
        \int_{B_\frac{\rho}{4}(x_0)} |D^2u_\epsilon|^2\,\dx \leq \frac{C}{\epsilon^2\rho^2}\bigg(\int_{B_{2\rho}(x_0)}|Du_\epsilon|^2\,\dx + \rho^2 \int_{B_{2\rho}(x_0)} (1 + |f|^2)\,\dx \bigg)
    \end{equation}
    with~$C=C(C_\F,\hat{C}_\F,N,n,r_E)$.
\end{myproposition}


\begin{proof}
As usual, the existence of local second order weak derivatives of~$u_\epsilon$ in~$B_R$ will be achieved by employing the technique of difference quotients. We consider~$B_{2\rho}(x_0)\Subset B_R$ and~$h$ small enough, such that~$0<|h|<\text{dist}\big(B_{2\rho}(x_0),B_R\big)$. The finite difference of~$u_\epsilon$ for a.e.~$x\in B_{2\rho}(x_0)$ in direction~$x_i$ with increment~$h\neq 0$ is given by
$$\tau_{i,h}u_\epsilon(x)\coloneqq u_\epsilon(x+he_i)-u_\epsilon(x),$$
where~$e_i$ denotes the standard unit vector in direction~$x_i$ for an arbitrary~$i\in\{1,2,\ldots,n\}$. To simplify notation, we will write~$u=u_\epsilon$ and~$\tau_h u=\tau_{i,h}u_\epsilon(x)$, omitting the dependence on~$\epsilon\in(0,1]$ and on~$i\in\{1,2,\ldots,n\}$ respectively. We test the weak formulation~\eqref{weakformapprox} with~$\phi = \tau_{-h}[\zeta^2\tau_h u]$, where~$\zeta\in C^2_0(B_{\frac{\rho}{2}}(x_0),[0,1])$ is a smooth cut-off function that satisfies~$\zeta\equiv 1$ on~$B_{\frac{\rho}{4}}(x_0)$,~$|D\zeta|\leq \frac{8}{\rho}$, and~$|D^2\zeta| \leq \frac{64}{\rho^2}$. This choice of test function can be justified by an approximation argument. A change of variables, cf.~\cite[Chapter~5.8,~(10)]{evans2022partial}, yields a discrete integration by parts
\begin{align} \label{weakdifference}
    \int_{B_R} \zeta^2 \langle \tau_h \hat{\h}_\epsilon(x,Du), \tau_h Du \rangle \,\dx &+ 2\int_{B_R}  \langle  \tau_h \hat{\h}_\epsilon(x,Du), \tau_h u \zeta D\zeta  \rangle \,\dx   \\
    &= - \int_{B_R} f \tau_{-h}[\zeta^2\tau_h u]\,\dx, \nonumber
\end{align}
where we also utilized the commutativity of finite differences with weak derivatives. For the second term on the left-hand side in the preceding identity~\eqref{weakdifference}, we again use this integration by parts in the sense of changing variables and remove the finite difference from the vector field~$\hat{\h}_\epsilon$ by shifting it onto the quantity~$\zeta \tau_h u D\zeta$. This way, there holds
\begin{align} \label{weakdifferencetwice}
    \int_{B_R} \zeta^2 \langle \tau_h \hat{\h}_\epsilon(x,Du), \tau_h Du \rangle \,\dx & - 2\int_{B_R}  \langle \hat{\h}_\epsilon(x,Du), \tau_{-h}[\tau_h u \zeta D\zeta]  \rangle \,\dx   \\
    &= \int_{B_R} f \tau_{-h}[\zeta^2\tau_h u]\,\dx, \nonumber
\end{align}
We rewrite the finite difference~$\tau_h\hat{\h}_\epsilon(x,Du)$ in the first integral term in~\eqref{weakdifferencetwice} as follows
\begin{align*}
    \tau_h \hat{\h}_\epsilon(x,Du) &= \epsilon \tau_h Du(x) + \tau_h \hat{\h}(x,Du(x)) \\
    & = \epsilon \tau_h Du(x) + \hat{\h}(x+h e_i,Du(x+h e_i)) - \hat{\h}(x,Du(x))  \\
    & = \epsilon \tau_h Du(x) + \hat{\h}(x+h e_i,Du(x+h e_i)) - \hat{\h}(x,Du(x+h e_i)) \\
    &\quad + \hat{\h}(x,Du(x+h e_i)) - \hat{\h}(x,Du(x)) \\
    & \eqqcolon \foo{I} + \foo{II} + \foo{III}.
\end{align*}
 The quantities~$\foo{I}$ and~$\foo{III}$ remain unchanged. For the second term~$\foo{II}$, we exploit the Lipschitz continuity of~$\hat{\h}(x,\xi)$ with respect to the~$x$-variable. Thus, we exploit the bound~\eqref{hatflipschitz} and estimate the second quantity by
\begin{equation*}
   \foo{II} \leq C(\hat{C}_\F,N,r_E) |h|. 
\end{equation*}
In order to properly treat the second quantity in~\eqref{weakdifferencetwice} involving second order finite differences, we use the following product rule for finite differences
\begin{align*}
    \tau_{-h}[\tau_h u \zeta D\zeta](x) &= \tau_h u(x) \zeta(x) D\zeta(x) - \tau_h u(x-he_i) \zeta(x-he_i) D\zeta(x-he_i) \\
    &= \tau_{-h} \tau_h u(x) \zeta(x) D\zeta(x) + \tau_{-h}[\zeta D\zeta](x) \tau_hu(x-he_i) \\
    &= \tau_{-h} \tau_h u(x) \zeta(x) D\zeta(x) - \tau_{-h}[\zeta D\zeta](x) \tau_{-h}u(x)
\end{align*}
for a.e.~$x\in B_{2\rho}(x_0)$. Further, we apply the Cauchy-Schwarz inequality for this term as use the quadratic growth~\eqref{quadraticgrowth}. Taking modulus in the third integral term in~\eqref{weakdifferencetwice} involving the datum~$f$, we plug in our calculations from above. This way, we obtain
\begin{align*}
    & \underbrace{\int_{B_{R}} \zeta^2  \langle \hat{\h}(x,Du(x+he_i))  - \hat{\h}(x,Du(x)), Du(x+he_i)-Du(x) \rangle \,\dx}_{\ge 0} + 
    \epsilon\int_{B_{R}} \zeta^2 |\tau_h Du|^2 \,\dx \\
    &\quad\leq C |h| \int_{B_{\frac{\rho}{2}}(x_0)} \zeta^2 |\tau_h Du| \,\dx + \int_{B_{R}} |f| |\tau_{-h}[\zeta^2\tau_h u]|\,\dx \\
    &\quad\quad + C(\epsilon+1) \int_{B_{R}} (1+|Du|)|\tau_{-h}[\zeta D\zeta]| |\tau_{-h}u|\,\dx + C(\epsilon+1) \int_{B_{R}} (1+|Du|) \zeta |D\zeta| |\tau_{-h}\tau_h u| \,\dx \\ 
    &\quad \eqqcolon \Tilde{\foo{I}} + \Tilde{\foo{II}} + \Tilde{\foo{III}} + \Tilde{\foo{IV}}
\end{align*}
with a constant~$C=C(C_\F,\hat{C}_\F,N,r_E)$. We note that due to the second assertion of Lemma~\ref{monotonicityapprox}, the first term on the left-hand side above leads to a non-negative contribution and therefore can be discarded. 
We now treat each of the terms~$\Tilde{\foo{I}} - \Tilde{\foo{IV}}$ one after another and commence by estimating the first quantity further above by employing Young's inequality and by exploiting the fact that~$\zeta\leq 1$, which yields
\begin{align*}
    \Tilde{\foo{I}} &\leq \frac{\epsilon}{4} \int_{B_{R}} \zeta^2 |\tau_h Du|^2\,\dx + \frac{C}{\epsilon} |h|^2 |B_{\frac{\rho}{2}}(x_0)|.
\end{align*}
Next, we consider the term~$\Tilde{\foo{II}}$ containing the datum~$f$. Again, we apply the Cauchy-Schwarz inequality and Young's inequality, use the bounds for~$\zeta$ and~$D\zeta$, and also a standard estimate for finite differences that is taken from~\cite[Chapter~5.8,~Theorem~3]{evans2022partial}, to achieve
\begin{align*} \label{termf}
     \Tilde{\foo{II}} &\leq \frac{2}{\epsilon} |h|^2 \int_{B_{\rho}(x_0)} |f|^2\,\dx + \frac{\epsilon}{8|h|^2}\int_{B_{\frac{\rho}{2}}(x_0)} |\tau_{-h}[\zeta^2\tau_h u]|^2\,\dx  \\
    &\leq \frac{2}{\epsilon} |h|^2 \int_{B_{\rho}(x_0)} |f|^2\,\dx + \frac{\epsilon}{8}\int_{B_{\rho}(x_0)} \big|D[\zeta^2\tau_h u]\big|^2\,\dx \\
    &=\frac{2}{\epsilon}  |h|^2 \int_{B_{\rho}(x_0)} |f|^2\,\dx  + \frac{\epsilon}{8}\int_{B_\rho(x_0)} \big|\zeta^2 \tau_h Du + 2\zeta D\zeta \tau_h u\big|^2\,\dx  \\
    &\leq \frac{2}{\epsilon} |h|^2 \int_{B_{\rho}(x_0)} |f|^2\,\dx + \frac{\epsilon}{4}\int_{B_{R}} \zeta^2 |\tau_h Du|^2\,\dx  +  \frac{C \epsilon}{\rho^2}  |h|^2 \int_{B_{2\rho}(x_0)} |Du|^2
\end{align*}
with~$C=C(n)$. The third quantity~$\Tilde{\foo{III}}$ is bounded above by taking the essential supremum of~$Du$, applying Young's inequality, and then by proceeding in a similar fashion as for the second term
\begin{align*}
   \Tilde{\foo{III}} &\leq C(\epsilon+1)\rho^2 \int_{B_{\frac{\rho}{2}}(x_0)} |\tau_{-h}[\zeta D\zeta]|^2\,\dx + C \frac{(\epsilon+1)}{\rho^2} \int_{B_{\frac{\rho}{2}}(x_0)} |\tau_{-h}u|^2\,\dx \\
   &\leq C(\epsilon+1)\rho^2|h|^2 \int_{B_{\rho}(x_0)} |D[\zeta D\zeta]|^2\,\dx + C\frac{(\epsilon+1)}{\rho^2}|h|^2 \int_{B_{\rho}(x_0)} |Du|^2\,\dx \\
   &\leq C(\epsilon+1) \rho^2|h|^2 \int_{B_{\rho}(x_0)} (|D\zeta\otimes D\zeta|^2 + |D^2\zeta|^2)\,\dx + C\frac{(\epsilon+1)}{\rho^2}|h|^2 \int_{B_{\rho}(x_0)} |Du|^2\,\dx \\
   &\leq \frac{C(\epsilon+1)|h|^2}{\rho^2} |B_{\rho}(x_0)| + C\frac{(\epsilon+1)}{\rho^2}|h|^2 \int_{B_{\rho}(x_0)} |Du|^2\,\dx
\end{align*}
with~$C=C(C_\F,\hat{C}_\F,N,n,r_E)$. Finally, the fourth term~$\Tilde{\foo{IV}}$ is bounded above by Hölder's and Young's inequality, and also by exploiting the fact that finite differences and weak derivatives commutate, as follows
\begin{align*}
    \Tilde{\foo{IV}} &\leq C(\epsilon+1) \bigg( \int_{B_{\frac{\rho}{2}}(x_0)} (1+|Du|)^2 \zeta^2 |D\zeta|^2\,\dx \bigg)^{\frac{1}{2}} \bigg( \int_{B_{\frac{\rho}{2}}(x_0)} |\tau_{-h}\tau_h u|^2 \,\dx \bigg)^{\frac{1}{2}} \\
    &\leq C(\epsilon+1)|h| \bigg( \int_{B_{\frac{\rho}{2}}(x_0)} (1+|Du|)^2 \zeta^2 |D\zeta|^2\,\dx \bigg)^{\frac{1}{2}} \bigg( \int_{B_{\rho}(x_0)} |\tau_h Du|^2 \,\dx \bigg)^{\frac{1}{2}} \\
    &\leq \frac{C}{\rho^2} \frac{\epsilon+1}{\epsilon} |h|^2 \int_{B_{\frac{\rho}{2}}(x_0)} (1+|Du|)^2 \,\dx + \frac{\epsilon}{4} \int_{B_{R}} |\tau_h Du|^2 \,\dx.
\end{align*}
After reabsorbing all quantities from our estimates of the terms~$\Tilde{\foo{I}} - \Tilde{\foo{IV}}$ involving a finite difference~$\tau_h Du$ and using the fact that~$\epsilon,\rho\in(0,1]$, we end up with 
\begin{align*}
    \epsilon \int_{B_{R}} \zeta^2 |\tau_h Du|^2 \,\dx  
    \leq \frac{C|h|^2}{\epsilon \rho^2} \bigg( \int_{B_{2\rho}(x_0)} |Du|^2 \,\dx + \rho^2\int_{B_{2\rho}(x_0)} (1 + |f|^2)\,\dx \bigg)
\end{align*}
with a constant~$C=C(C_\F,\hat{C}_\F,N,n,r_E)$. For the second term on the left-hand side we exploit the choice of~$\zeta$ once more. Dividing both sides by~$\epsilon>0$ eventually leads to
\begin{equation*}
   \int_{B_{\frac{\rho}{4}}(x_0)} |\tau_h Du|^2 \,\dx \leq \frac{C}{\epsilon^2\rho^2}|h|^2\bigg( \int_{B_{2\rho}(x_0)} |Du|^2 \,\dx + \rho^2\int_{B_{2\rho}(x_0)} (1 + |f|^2)\,\dx\bigg)
\end{equation*}
with~$C=C(C_\F,\hat{C}_\F,N,n,r_E)$. As the term in brackets on the right-hand side of the preceding estimate is uniformly bounded with respect to~$|h|>0$ and~$i\in\{1,2,\ldots,n\}$ was chosen arbitrarily, we utilize another classical result concerning difference quotients and again refer to~\cite[Chapter~5.8,~Theorem~3]{evans2022partial}. This yields the local~$W^{2,2}$-estimate~\eqref{approxw22}.
\end{proof}


Next, we establish the regularity~$u_\epsilon\in W^{1,\infty}_{\loc}(B_R)$, which, in combination with Proposition~\ref{approxregularityeins}, yields that there holds
$$ u_\epsilon\in W^{1,\infty}_{\loc}(B_R) \cap W^{2,2}_{\loc}(B_R) $$
for any~$\epsilon\in(0,1]$. Moreover, we state a quantitative local~$L^\infty$-gradient bound for~$u_\epsilon$ in~$B_R$ that is taken from~\cite{brasco2011global}, which will turn out expedient when determining a local with respect to the parameter~$\epsilon\in(0,1]$ uniform~$L^\infty$-gradient bound for~$u_\epsilon$ in~$B_R$. 


\begin{myproposition} \label{approxregularityzwei}
    Let~$\epsilon\in(0,1]$. Then, there holds
    $$ u_\epsilon\in W^{1,\infty}_{\loc}(B_R) \cap W^{2,2}_{\loc}(B_R). $$
    Moreover, for any~$B_{2\rho}(x_0)\Subset B_R$ there holds the quantitative~$L^\infty$-gradient estimate
    \begin{equation} \label{approxlipschitz}
        \esssup\limits_{B_{\rho}(x_0)} |Du_\epsilon| \leq C \Bigg( \bigg( \fint_{B_{2\rho}(x_0)}|Du_\epsilon|^2 \,\dx \bigg)^{\frac{1}{2}} +1 \Bigg)
    \end{equation} 
   with~$C=C(C_\F,\hat{C}_\F,\|f\|_{L^{n+\sigma}(B_R)},N,n,R,R_E,r_E,\sigma)$. 
\end{myproposition}

\begin{proof}
We observe that~$\hat{\F}_\epsilon$ satisfies the ellipticity and boundedness conditions outlined in~\eqref{hatfvoraussetzung}, with constants~$\hat{\lambda} = \hat{\lambda}(C_F, \delta)$ and~$\hat{\Lambda} = \hat{\Lambda}(C_F, \delta)$. Moreover, the estimate~\eqref{hatflipschitz} and also the quadratic growth properties~\eqref{quadraticgrowth} (compare also~\eqref{quadraticgrowthohnekonst}) are available for use. Our aim is to exploit the~$L^\infty$-gradient bound in~\cite[Section 4]{brasco2011global} right after inequality~\cite[Section 4, (4.12)]{brasco2011global}.

First and foremost, it is important to note that the $L^\infty$-gradient bound in~\cite[Section 4]{brasco2011global}, that is established throughout~\cite[Sections 3 - 4]{brasco2011global}, was proven within a more general framework than the one initially introduced in \cite[Section 2]{brasco2011global}, where the model equation involving the function $\mathcal{H}^*(\xi) \coloneqq \frac{1}{q} (|z| - \delta)^q$ for $\delta > 0$, $q > 1$, and $\xi \in \mathbb{R}^n$ is presented. Indeed, in the approximation process described in~\cite[Section 2.1]{brasco2011global}, a broader class of structures is considered, without imposing any specific quantitative form -- such as that of the model function~$\mathcal{H}^*$ -- on the approximating functions~$\mathcal{H}^*_\epsilon$. Instead, the construction is based solely on the structural conditions outlined in \cite[Section 2, (G1) - (G3) \& (2.6)]{brasco2011global}. Note that for our purposes the case $q=2$ is relevant.

Notably, condition \cite[(G1)]{brasco2011global} is only necessary to ensure the existence of second order weak derivatives of the approximating solutions~$u_\epsilon$ that are locally square-integrable. 
The existence of second order weak derivatives is an essential property for differentiating the equation as carried out in \cite[Sections 3 – 4]{brasco2011global}. However, thanks to Proposition~\ref{approxregularityeins} this property is available in our situation. Therefore, assumption \cite[(G1)]{brasco2011global} becomes negligible in our context. Instead, assumptions \cite[Section 2, (G2) - (G3) \& (2.6)]{brasco2011global}, along with the integrability condition $f \in L^{n+\sigma}(\Omega)$ for some $\sigma > 0$, are sufficient to derive the desired $L^\infty$-gradient bound. 
In particular, because of the quadratic growth of~$\hat{\F}_\epsilon$, the conditions~\cite[Section 2, (G2) - (G3)]{brasco2011global} are satisfied for the choice~$q=2$ with constants~$\lambda = \lambda(C_\F, R_E)$ and~$M = K+R_E$, while the Lipschitz condition~\cite[Section 2, (2.6)]{brasco2011global} holds with a Lipschitz constant~$L = L(\hat{C}_\F, N)$ respectively, an immediate consequence of conditions~\eqref{hatflipschitz},~\eqref{hatfvoraussetzung}, and also~\eqref{quadraticgrowth}.

Concerning the notation used by the author, the positive constant~$M$, introduced in~\cite[(2.4)]{brasco2011global}, can be chosen in accordance with the radius~$R_E > 0$ specified in~\eqref{mengeeradii}. Specifically, by setting~$M =  K+R_E$ in~\cite[(2.4)]{brasco2011global} as above, we have that
$$ B_R\cap \{w_\epsilon>k_0\} = B_R\cap \{|Du_\epsilon|>K+R_E\}, $$
where the auxiliary function~$w_\epsilon$ is defined as~$w_\epsilon \coloneqq (1 + |Du_\epsilon|^2)$, and the constant~$k_0$ is given by~$k_0 \coloneqq (1+M^2)$ in~\cite[(2.4)]{brasco2011global}. This choice of~$M$ ensures that the equation considered in~\cite[(3.1)]{brasco2011global} remains elliptic whenever the test function~$\phi=D_iu_\epsilon (w^s_\epsilon-k^s_0) \zeta^2$,~$i=1,\ldots,n$, with some integrability exponent~$s > 0$, and a standard smooth cut-off function~$\zeta \in C_0^1(B_{2\rho}(x_0))$ that is compactly supported in $B_{2\rho}(x_0)\Subset B_R$), is positive. In particular, this choice of~$M$ allows a differentiation of the weak form of the equation in \eqref{weakformapprox}, which follows from the fact that the partial mapping~$\xi\mapsto \hat{\F}_\epsilon(x,\xi)$ is~$C^2$-regular within the considered subset of points, since~$x\in B_R\cap \{w_\epsilon>k_0\}$ implies that~$|Du_\epsilon(x)|_E>1$. By following the steps outlined in~\cite[Sections 3 - 4]{brasco2011global} verbatim, we infer the desired~$L^\infty$-gradient estimate~\eqref{approxregularityzwei}, with a constant $C$ depending only on the given data but remaining independent of~$\epsilon\in(0,1]$. 
\end{proof}


The next lemma yields the convergence of the approximation~$\G_\delta(Du_\epsilon) \to \G_\delta(Du)$ in~$L^2(B_R,\R^n)$, an important property that will be required later on when determining the regularity of the limit function~$\G_\delta(Du)$.


\begin{mylem} \label{approxkvgzinl2}
    Let~$\delta,\epsilon\in(0,1]$ and~$u_\epsilon$ be the unique weak solution to~\eqref{approx}. Then, there holds
    \begin{equation*}
        \G_\delta(Du_\epsilon) \to \G_\delta(Du)\qquad\text{in~$L^2(B_R,\R^n)$ as~$\epsilon\downarrow 0$.}
    \end{equation*}
\end{mylem}
\begin{proof}
Due to the definition of~$u_\epsilon$ and by omitting an additional approximation argument, we may test both weak forms~\eqref{weakform} and~\eqref{weakformapprox} with~$\phi=u_\epsilon-u$ and subtract them from each other to obtain the identity
\begin{align*} 
    \int_{B_R} \langle \hat{\h}_\epsilon(x,Du_\epsilon)-\hat{\h}_\epsilon(x,Du),Du_\epsilon-Du\rangle\,\dx &=  \epsilon \int_{B_R}  \langle Du, Du-Du_\epsilon\rangle \,\dx. 
\end{align*}
We now apply Lemma~\ref{lemgdeltakvgz} for~$\hat{\h}$ with~$\xi=Du$ and~$\tildexi= Du_\epsilon$, which is applicable due to~\eqref{hatfvoraussetzung}. Further, we employ Young's inequality, to achieve
\begin{align*}
   \epsilon &\int_{B_R}   |Du_\epsilon-Du|^2\,\dx + \int_{B_R}|\G_\delta(Du_\epsilon) - \G_\delta(Du)|^2\,\dx  \\
    &\leq C(C_\F,\delta,R_E,r_E) \int_{B_R} \langle \hat{\h}_\epsilon(x,Du_\epsilon)-\hat{\h}_\epsilon(x,Du), Du_\epsilon-Du \rangle \,\dx  \\
    &\leq  C(C_\F,\delta,R_E,r_E)\, \epsilon \int_{B_R} |Du||Du_\epsilon-Du|\,\dx \\  
    &\leq \epsilon \int_{B_R} |Du_\epsilon-Du|^2\,\dx +  C(C_\F,\delta,R_E,r_E) \, \epsilon \int_{B_R}|Du|^2\,\dx.  
\end{align*}
The first term on the right-hand side is reabsorbed into the left-hand side. Due to the fact that the constant~$C$ importantly is independent of~$\epsilon$, we may pass to the limit~$\epsilon\downarrow 0$ to obtain the claimed convergence.
\end{proof}


\subsection{Hölder continuity of~\texorpdfstring{$\G_\delta(Du_\epsilon)$}{}} \label{sec:holderapprox} 
Throughout this section, let~$B_R=B_R(y_0)\Subset\Omega$ denote the ball introduced in Section~\ref{sec:regularizing} and consider~$B_{2\rho}(x_0)\Subset B_R$. Due to Lemma~\ref{energybound} and the gradient bound~\eqref{approxlipschitz} from Proposition~\ref{approxregularityzwei}, we know that the family~$(Du_\epsilon)_{\epsilon\in(0,1]}$ is uniformly bounded with respect to the regularizing parameter~$\epsilon$ on~$B_{2\rho}(x_0)$, i.e. that there exists a positive constant~$\Tilde{M}<\infty$ depending on the data
\begin{align*}
   (C_\F,\hat{C}_\F,\|Du\|_{L^{2}(B_R)},\|f\|_{L^{n+\sigma}(B_R)},N,n,R,R_E,r_E,\sigma),
\end{align*}
but that is independent of the parameter~$\epsilon$, such that there holds
\begin{equation} \label{glmschranke} 
    \sup\limits_{\epsilon\in(0,1]} \|Du_\epsilon\|_{L^\infty(B_{2\rho}(x_0))} \leq \Tilde{M}. 
\end{equation}
Due to~\eqref{betragminkowski}, the preceding estimate implies the bound
\begin{equation} \label{minkowskischranke}
    \sup\limits_{\epsilon\in(0,1]} \esssup\limits_{B_{2\rho}(x_0)} |Du_\epsilon|_E \leq \mbox{$\frac{\Tilde{M}}{r_E}$} \eqqcolon M.
\end{equation}
In what follows, we assume that~$M\geq 3$. This can be rewritten in the following way
\begin{equation} \label{schrankeeins}
    \esssup\limits_{B_{2\rho}(x_0)} |Du_\epsilon|_E \leq 1+\delta+\mu
\end{equation}
for any~$\epsilon\in(0,1]$, for some parameter~$\mu>0$ satisfying
\begin{equation} \label{schrankezwei}
    1+\delta+\mu\leq M.
\end{equation}
Moreover, given any~$e^*\in \partial E^*$, we define the level set of~$\partial_{e^*} u_\epsilon$ on~$B_\rho(x_0)$ for some level parameter~$\nu\in(0,1)$ by
\begin{equation*}
    E^{\nu}_{e^*,\rho}(x_0) \coloneqq B_\rho(x_0)\cap \big\{ \partial_{e^*} u_\epsilon -(1+\delta) > (1-\nu)\mu \big\}.
\end{equation*}
For the remainder of the article, the following abbreviation
\begin{align} \label{hesseschrankedelta}
    C_\F(\delta) \coloneqq \max\{\|\nabla^2\hat{\F}\|_{L^\infty(B_R\times (\R^n\setminus E_\delta))}, 3C_\F+1\} < \infty,
\end{align}
where~$\delta\in(0,1]$ and we recall~\eqref{hatapproxbound}, will be employed.
\,\\

Our main result of this section establishes the local Hölder continuity of~$\G_\delta(Du_\epsilon)$ for any weak solution~$u_\epsilon$ to~\eqref{approx}, stated in the subsequent theorem.


\begin{mytheorem} \label{holdermainresult}
    Let $\delta,\epsilon\in(0,1]$ and $u_{\epsilon}$ be the unique weak solution to the Dirichlet problem~\eqref{approx}. Then, for any~$B_{2\rho}(x_0)\Subset B_R$ there holds
    $$\G_{\delta}(Du_{\epsilon}) \in C^{0,\alpha_\delta}(B_\rho(x_0),\R^n)$$
    with a Hölder exponent $\alpha_{\delta}\in(0,1)$ and a Hölder constant $C_{\delta}$ subject to the dependencies
    \begin{align*}
        \alpha_\delta &= \alpha_\delta(C_\F(\delta),\hat{C}_\F,\delta,\|f\|_{L^{n+\sigma}(B_R)},M,N,n,R,r_E,\sigma) \in(0,1), \\
        C_\delta &= C_\delta(C_\F(\delta),\hat{C}_\F,\delta,\|f\|_{L^{n+\sigma}(B_R)},M,N,n,R,R_E,r_E,\sigma) \geq 1,
    \end{align*}
    i.e. there holds
    \begin{equation} \label{gdeltaholderquant}
        |\G_{\delta}(Du_{\epsilon}(x)) - \G_{\delta}(Du_{\epsilon}(y))| \leq C_\delta |x-y|^{\alpha_\delta} \qquad \mbox{for a.e. $x,y\in B_\rho(x_0) \Subset B_R$}.
    \end{equation}
  
\end{mytheorem}


In order to prove Theorem~\ref{holdermainresult}, we will distinguish between the~\textit{ non-degenerate regime} and the~\textit{degenerate regime}. The non-degenerate regime is characterized by assuming that the measure-theoretic information~$|B_\rho(x_0)\setminus E^\nu_{e^*,\rho}(x_0)|\geq \nu |B_\rho(x_0)|$ is satisfied for at least one~$e^*\in\partial E^*$, whereas in the degenerate regime the reverse inequality holds true for any~$e^*\in\partial E^*$. Put briefly, in the non-degenerate regime the subset of points where~$\partial_{e^*}u_\epsilon$ is near its supremum is large in measure, while in the degenerate regime the subset of points where~$\partial_{e^*}u_\epsilon$ far from its supremum is large in measure. In what follows, we simplify notation by setting
\begin{align} \label{beta}
   \beta = \frac{\sigma}{n+\sigma} = 1-\frac{n}{n+\sigma} \in (0,1). 
\end{align}
 The $L^2$-excess of~$Du_\epsilon$ on~$B_\rho(x_0)$ will be denoted by
\begin{align} \label{excess}
    \Phi(x_0,\rho) \coloneqq \fint_{B_\rho (x_0)} \big| Du_\epsilon-(Du_\epsilon)_{x_0,\rho} \big| ^2 \,\dx,
\end{align}
which plays a crucial role in the treatment of the non-degenerate regime.
\,\\


The first proposition treats the non-degenerate regime.


\begin{myproposition} \label{nondegenerateproposition}
    Let~$\delta,\epsilon\in(0,1]$,~$\mu>0$, and assume that
    \begin{equation} \label{deltamu}
        \delta<\mu. 
    \end{equation}
    There exist an exponent~$\alpha \in (0,1)$, a constant~$C\geq 1$, a parameter~$\nu\in(0,\frac{1}{4}]$, and a radius~$\Tilde{\rho}\in(0,1]$, subject to the dependencies
    \begin{align*}
     \alpha &=\alpha(C_\F(\delta),\hat{C}_\F,\delta,\|f\|_{L^{n+\sigma}(B_R)},M,N,n,r_E,\sigma), \\
     C &= C(C_\F(\delta),\hat{C}_\F,\delta,\|f\|_{L^{n+\sigma}(B_R)},M,N,n,R_E,r_E,\sigma), \\
      \nu &= \nu(C_\F(\delta),\hat{C}_\F,\delta,\|f\|_{L^{n+\sigma}(B_R)},M,N,n,R,r_E,\sigma), \\
        \Tilde{\rho} &=\Tilde{\rho}(C_\F(\delta),\hat{C}_\F,\delta,\|f\|_{L^{n+\sigma}(B_R)},M,N,n,r_E,\sigma),
    \end{align*}    
    such that: if the measure condition
    \begin{equation} \label{nondegeneratemeascond}
    |B_\rho(x_0)\setminus E^{\nu}_{e^*,\rho}(x_0)|<\nu|B_\rho(x_0)|
    \end{equation}
     is satisfied for at least one~$e^*\in \partial E^*$ on~$B_{\rho}(x_0)\subset B_{2\rho}(x_0)\Subset B_{R}$ with $\rho\leq\Tilde{\rho}$, then the limit
    \begin{equation} \label{lebesguerepresentant}
        \Gamma_{x_0}\coloneqq \lim\limits_{r\downarrow 0} (\G_{2\delta}(Du_\epsilon))_{x_0,r}
    \end{equation}
     exists. Furthermore, we have the excess-decay estimate
    \begin{equation} \label{excessgdelta}
        \fint_{B_r(x_0)}|\G_{2\delta}(Du_\epsilon)-\Gamma_{x_0}|^2\,dx \leq C\Big(\frac{r}{\rho} \Big)^{2\alpha} \mu^2
    \end{equation}
    for any $r\in(0,\rho]$, as well as the bound
    \begin{align} \label{lebesguebound}
        |\Gamma_{x_0}| \leq R_E\mu.
    \end{align}
\end{myproposition}


In what follows, let~$\alpha\in(0,1)$ denote the parameter from Proposition~\ref{nondegenerateproposition} and~$0<\Tilde{\beta}<\alpha<\beta$, where~$\beta$ is given in~\eqref{beta}. The second proposition deals with the degenerate regime.
\begin{myproposition} \label{degenerateproposition}
    Let $\delta,\epsilon\in(0,1]$, $\mu>0$, and $\nu\in(0,\frac{1}{4}]$. Assuming that~\eqref{deltamu} holds true, 
    there exist a parameter $\kappa\in\big[2^{-\frac{\Tilde{\beta}}{2}},1\big)$ and a radius~$\hat{\rho}\in(0,1)$, both depending on the data 
    \begin{align*}
       (C_\F(\delta),\hat{C}_\F,\delta,M,\|f\|_{L^{n+\sigma}(B_R)},N,n,\nu,R,r_E,\sigma),
    \end{align*}
     such that: if the measure condition
    \begin{equation} \label{degeneratemeascond}
                |B_\rho(x_0)\setminus E^{\nu}_{e^*,\rho}(x_0)|\geq\nu|B_\rho(x_0)|
    \end{equation}
         is satisfied for any~$e^*\in \partial E^*$ on~$B_\rho(x_0)\subset B_{2\rho}(x_0)\Subset B_{R}$ with~$\rho\leq \hat{\rho}$, then there holds
    \begin{equation} \label{degenerateestzwei}
        \esssup\limits_{B_{\frac{\rho}{2}}(x_0)}\,|\G_{\delta}(Du_\epsilon)|_E \leq \kappa \mu.
    \end{equation}
\end{myproposition}


We postpone the proof of Proposition~\ref{nondegenerateproposition} and Proposition~\ref{degenerateproposition} for the moment and treat them later on in Sections~\ref{sec:nondegenerate} and~\ref{sec:degenerate}.


\begin{proof}[\textbf{\upshape Proof of Theorem~\ref{holdermainresult}}]
The proof of Theorem~\ref{holdermainresult} follows by combining the results of Propositions~\ref{nondegenerateproposition} and~\ref{degenerateproposition}. Roughly speaking, it builds upon a strategy in the spirit of DiBenedetto and Friedman's~\cite{friedman1984regularity,dibenedetto1985addendum,Friedman1985} and also Uhlenbeck's~\cite{uhlenbeck1977regularity} pioneering results on the Hölder continuity of the gradient of solutions to elliptic resp. parabolic~$p$-Laplacian systems. We also mention Manfredi's contribution in the elliptic setting~\cite{manfredi1986regularity}. The aim is to show that the Lebesgue representative~$\Gamma_{x_0}$, as defined in~\eqref{lebesguerepresentant}, exists and is locally Hölder continuous in~$B_R$. Let
$$ \nu=\nu(C_\F(\delta),\hat{C}_\F,\delta,\|f\|_{L^{n+\sigma}(B_R)},M,N,n,R,r_E,\sigma) \in(0,\mbox{$\frac{1}{4}$} ]$$
denote the parameter and
$$ \Tilde{\rho} =\Tilde{\rho}(C_\F(\delta),\hat{C}_\F,\delta,\|f\|_{L^{n+\sigma}(B_R)},M,N,n,r_E,\sigma)\in(0,1] $$
denote the radius from Proposition~\ref{nondegenerateproposition}, whereas
$$ \hat{\rho}=\hat{\rho}(C_\F(\delta),\hat{C}_\F,\delta,\|f\|_{L^{n+\sigma}(B_R)},M,N,n,R,r_E,\sigma) \in(0,1) $$
represents the radius from Proposition~\ref{degenerateproposition}, being now chosen for the above specific determination of~$\nu$. Further, let
$$ C=C(C_\F(\delta),\hat{C}_\F,\delta,\|f\|_{L^{n+\sigma}(B_R)},M,N,n,R_E,r_E,\sigma) \geq 1 $$
denote the constant from Proposition~\ref{nondegenerateproposition} and
$$ \kappa = \kappa(C_\F(\delta),\hat{C}_\F,\delta,\|f\|_{L^{n+\sigma}(B_R)},M,N,n,R,r_E,\sigma) \in \big[2^{-\frac{\Tilde{\beta}}{2}},1\big) $$
denote the parameter from Proposition~\ref{degenerateproposition}, again for the above choice of parameter~$\nu\in(0,\frac{1}{4}]$. Here, as before, by~$\alpha\in(0,1)$ we denote the parameter from Proposition~\ref{nondegenerateproposition} and
$$0<\Tilde{\beta}<\alpha<\beta.$$ 
For the following discussion, let
$$\rho^* \coloneqq \min\{\Tilde{\rho},\hat{\rho}\}\in(0,1)$$
and consider a ball~$B_{2\rho}(x_0) \Subset B_{R}$ with radius~$\rho\in(0,\rho^*]$. Further, we let~$\mu\coloneqq M-(1+\delta)>0$. We prove that~$\G_{\delta}(Du_\epsilon)$ is Hölder continuous in~$B_\rho(x_0)\Subset B_R$ with Hölder exponent
$$\alpha_\delta\coloneqq - \frac{\log \kappa}{\log 2} \in \big(0,\mbox{$\frac{\Tilde{\beta}}{2}$} \big]$$
and Hölder constant~$C_\delta\geq 1$. This way, there holds
$$ \alpha_\delta = \alpha_\delta(C_\F(\delta),\hat{C}_\F,\delta,\|f\|_{L^{n+\sigma}(B_R)},M,N,n,R,r_E,\sigma) \in (0,1) $$
and
$$ C_\delta = C_\delta(C_\F(\delta),\hat{C}_\F,\delta,\|f\|_{L^{n+\sigma}(B_R)},M,N,n,R,R_E,r_E,\sigma)\geq 1. $$
For technical reasons, we will actually prove the Hölder continuity of~$\G_{2\delta}(Du_\epsilon)$, from which the Hölder continuity of~$\G_\delta(Du)$ follows immediately due to the arbitrariness of~$\delta\in(0,1]$ by substituting~$\Tilde{\delta}\coloneqq \frac{\delta}{2}$. \,\\

We now show that
$$  \Gamma_{x_0}\coloneqq \lim\limits_{r\downarrow 0} (\G_{2\delta}(Du_\epsilon))_{x_0,r} $$
exists, and that there holds the quantitative~$L^2$ excess-decay estimate
\begin{align} \label{proofexcess}
     \fint_{B_r(x_0)}|\G_{2\delta}(Du_\epsilon)-\Gamma_{x_0}|^2\,dx \leq C \Big(\frac{r}{\rho} \Big)^{2\alpha_\delta} \mu^2
\end{align}
for any~$r\in(0,\rho]$, with a constant~$C\geq 1$ depending on the data
$$ (C_\F(\delta),\hat{C}_\F,\delta,\|f\|_{L^{n+\sigma}(B_R)},M,N,n,R,R_E,r_E,\sigma). $$
We consider a sequence of nested balls~$B_{\rho_i}(x_0)$ around~$x_0$, where
$$\rho_i\coloneqq \frac{\rho}{2^i}\qquad\mbox{for any~$i\in\N_0$}.$$
Moreover, we define levels
$$\mu_i\coloneqq \kappa^i\mu\in(0,\mu] \qquad\mbox{for any~$i\in\N_0$.}$$ According to this construction and the above choice of~$\alpha_\delta$, there holds
\begin{align} \label{levelest}
    \mu_i \leq 2^{-\alpha_\delta i}\mu = \Big(\frac{\rho_i}{\rho}\Big)^{\alpha_\delta} \mu \qquad \mbox{for any~$i\in\N_0$.}
\end{align}
We begin the iteration process by assuming that~\eqref{degeneratemeascond} holds true on~$B_{\rho_{0}}(x_0)=B_\rho(x_0)$ for any~$e^*\in\partial E^*$ to level~$\mu_0=\mu$, i.e. the degenerate regime applies. Then, if there holds~$\delta < \mu$, an application of Proposition~\ref{degenerateproposition} yields that 
   \begin{equation*} 
        \esssup\limits_{B_{\rho_1}(x_0)}\,|\G_{\delta}(Du_\epsilon)|_E \leq \kappa \mu = \mu_1.
    \end{equation*}
If~$\mu\leq \delta$, we infer, due to~$\G_{2\delta}(Du_\epsilon)=0$ and also~$\G_{2\delta} \leq \G_\delta$ a.e. in~$B_\rho(x_0)$, the estimate
   \begin{equation} \label{mudelta}
        \esssup\limits_{B_{\rho_i}(x_0)}\,|\G_{2\delta}(Du_\epsilon)|_E \leq \mu_i \qquad\mbox{for any~$i=0,1,\ldots$}
    \end{equation}
Due to simplicity, we therefore assume~$\delta<\mu_i$ in each step of the iteration, in case the degenerate regime applies. Now, if the measure-theoretic information~\eqref{degeneratemeascond} is satisfied on~$B_{\rho_1}(x_0)$ with level~$\mu_1$ for any~$e^*\in\partial E^*$, another application of Proposition~\ref{degenerateproposition} yields
   \begin{equation*} 
        \esssup\limits_{B_{\rho_2}(x_0)}\,|\G_{\delta}(Du_\epsilon)|_E \leq \kappa^2 \mu = \mu_2.
    \end{equation*}
Assume that the degenerate regime, i.e.~\eqref{degeneratemeascond}, applies on each ball of the iteration for~$i=0,1,\ldots,i^*-1,i^*$ for some~$i^*\in\N_0$. Then, we deduce that
   \begin{equation} \label{est:degenerateiterated}
        \esssup\limits_{B_{\rho_i}(x_0)}\,|\G_{\delta}(Du_\epsilon)|_E \leq  \mu_i \qquad \mbox{for any~$i=0,1,\ldots,i^*$}.
    \end{equation}
In case the measure-theoretic information~\eqref{degeneratemeascond} fails to hold for some~$i^*\in\N_0$, i.e. at a certain step of the iteration, we infer that there exists at least one~$e^*\in\partial E^*$, such that~\eqref{nondegeneratemeascond} is satisfied on~$B_{\rho_{i^*}}(x_0)$ to level~$\mu_{i^*}$. Then, if~$\delta<\mu_{i^*}$, Proposition~\ref{nondegenerateproposition} is applicable and the Lebesgue representative
$$  \Gamma_{x_0}\coloneqq \lim\limits_{r\downarrow 0} (\G_{2\delta}(Du_\epsilon))_{x_0,r} $$
exists, and there holds the quantitative~$L^2$ excess-decay estimate
\begin{align} \label{est:iterationexcess}
     \fint_{B_r(x_0)}|\G_{2\delta}(Du_\epsilon)-\Gamma_{x_0}|^2\,dx \leq C \Big(\frac{r}{\rho_{i^*}} \Big)^{2\alpha} \mu^2_{i^*}
\end{align}
for any~$r\in(0,\rho_{i*}]$, with a constant
$$ C = C(C_\F(\delta),\hat{C}_\F,\delta,\|f\|_{L^{n+\sigma}(B_R)},M,N,n,R_E,r_E,\sigma).  $$
Additionally, the Proposition yields
\begin{align} \label{est:nondegiterationgamma}
    |\Gamma_{x_0}|\leq R_E\mu_{i^*}. 
\end{align}
The estimates~\eqref{levelest} and~\eqref{est:iterationexcess}, due to~$\alpha_\delta\leq \frac{\Tilde{\beta}}{2}$, then imply
\begin{align*}
     \fint_{B_r(x_0)}|\G_{2\delta}(Du_\epsilon)-\Gamma_{x_0}|^2\,dx \leq C \Big(\frac{r}{\rho_{i^*}} \Big)^{2\alpha} \Big(\frac{\rho_{i^*}}{\rho} \Big)^{2\alpha_\delta} \mu^2 \leq C \Big(\frac{r}{\rho} \Big)^{2\alpha_\delta} \mu^2
\end{align*}
for any~$r\in(0,\rho_{i*}]$. If~$r\in (\rho_{i^*},\rho]$, we can find an index~$i=0,1,\ldots i^*$ with~$\rho_{i+1}\leq r\leq \rho_{i}$ and exploit~\eqref{betragminkowski},~\eqref{levelest},~\eqref{est:degenerateiterated}, and~\eqref{est:nondegiterationgamma}, to obtain
\begin{align*}
     \fint_{B_r(x_0)}|\G_{2\delta}(Du_\epsilon)-\Gamma_{x_0}|^2\,dx &\leq 2 \esssup\limits_{B_{\rho_i}(x_0)}\,|\G_{2\delta}(Du_\epsilon)| + 2 |\Gamma_{x_0}|^2 \\
     &\leq C \mu^2_i \leq C \Big(\frac{\rho_i}{\rho} \Big)^{2\alpha_\delta} \mu^2 \\
     & \leq C \Big(\frac{r}{\rho} \Big)^{2\alpha_\delta} \mu^2.
\end{align*}
 Thus, we have established the claimed excess-decay estimate~\eqref{proofexcess} if the non-degenerate regime applies at one step of the iteration process and the assumption~$\delta<\mu_{i^*}$ is fulfilled. Similarly to the degenerate regime, if~$\mu_{i^*}\leq \delta$, then we infer
   \begin{equation*}
        \esssup\limits_{B_{\rho_i}(x_0)}\,|\G_{2\delta}(Du_\epsilon)|_E \leq  \mu_i \qquad \mbox{for any~$i=0,1,\ldots$},
    \end{equation*}
which coincides with~\eqref{mudelta}. The last remaining scenario is the case where the degenerate regime, i.e.~\eqref{degeneratemeascond}, applies for any~$i\in\N_0$. However, in this case we once again obtain~\eqref{mudelta}, which yields that
\begin{align*}
   0 &\leq |(\G_{2\delta}(Du_\epsilon))_{x_0,\rho_i}| \leq \fint_{B_{\rho_i}(x_0)} |\G_{2\delta}(Du_\epsilon)|\,\dx \leq \mbox{$\frac{1}{r_E}$} \mu_i \to 0 \qquad\mbox{as~$i\to\infty$},
 \end{align*}
and further implies
$$\Gamma_{x_0}=\lim\limits_{r\downarrow 0} (\G_{2\delta}(Du_\epsilon))_{x_0,r} = 0.$$ 
Moreover, in this case we also establish the excess-decay estimate~\eqref{proofexcess}. Further, for any~$r\in(0,\rho]$ there exists~$i\in\N_0$ with~$\rho_{i+1}\leq r\leq \rho_i$. By employing~\eqref{betragminkowski},~\eqref{levelest}, and~\eqref{mudelta}, we thus infer
\begin{align*}
      \fint_{B_r(x_0)}|\G_{2\delta}(Du_\epsilon)-\Gamma_{x_0}|^2\,dx &\leq  \esssup\limits_{B_{\rho_i}(x_0)}\,|\G_{\delta}(Du_\epsilon)|^2_E \leq C \mu^2_i \leq \Big(\frac{\rho_i}{\rho} \Big)^{2\alpha_\delta} \mu^2 \leq C \Big(\frac{r}{\rho} \Big)^{2\alpha_\delta} \mu^2,
\end{align*}
where~$C$ exhibits the same dependence as before. The Hölder continuity of~$\G_{2\delta}(Du_\epsilon)$ now follows by considering its Lebesgue representative~$B_\rho(x_0)\ni x\mapsto\Gamma_{x}$. Let~$x,y\in B_\rho(x_0)$ with~$r\coloneqq|x-y|$. If~$r\leq \rho^*$, we let~$z\coloneqq \frac{x+y}{2}$ and infer from the reasoning provided above that
\begin{align*}
    |\Gamma_{x}-\Gamma_{y}|^2 & = \fint_{B_{\frac{r}{2}}(z)}  |\Gamma_{x}-\Gamma_{y}|^2\,\dx \\
    &\leq C(n) \bigg( \fint_{B_r(x)}|\G_{2\delta}(Du_\epsilon)-\Gamma_{x}|^2\,dx + \fint_{B_r(y)}|\G_{2\delta}(Du_\epsilon)-\Gamma_{y}|^2\,dx \bigg) \\
    &\leq C \bigg(\frac{|x-y|}{\rho^*} \bigg)^{2\alpha_\delta} \mu^2.
\end{align*}
This gives
\begin{align*}
    |\Gamma_{x}-\Gamma_{y}| \leq C \Big(\frac{|x-y|}{\rho^*} \Big)^{\alpha_\delta} \mu.
\end{align*}
In the other case where~$r>\rho^*$, we have 
\begin{align*}
    |\Gamma_{x}-\Gamma_{y}| \leq 2\mu \leq 2 \Big(\frac{|x-y|}{\rho^*} \Big)^{\alpha_\delta} \mu.
\end{align*}
This establishes the Hölder continuity of~$\G_{2\delta}(Du_\epsilon)$, and thus also of~$\G_{\delta}(Du_\epsilon)$. In particular, this reasoning yields the stated dependence of~$\alpha_\delta\in(0,1)$ and~$C_\delta\geq 1$ in Theorem~\ref{holdermainresult}. The proof is complete.
\end{proof}


We now state the proof of the main regularity result, that is Theorem~\ref{hauptresultat}.

\begin{proof}[\textbf{\upshape Proof of Theorem~\ref{hauptresultat}}]

Once Theorem~\ref{holdermainresult} has been established, the main regularity Theorem~\ref{hauptresultat} readily follows. As before, let us consider the arbitrarily chosen ball~$B_R=B_R(y_0)\Subset\Omega$ and let~$u_\epsilon\in u+W^{1,\infty}_0(B_R)$ denote the unique weak solution to the Dirichlet problem~\eqref{approx} for some parameter~$\epsilon\in(0,1]$. Moreover, consider~$\delta\in(0,1]$ and an arbitrarily chosen ball~$B_\rho(x_0)\Subset B_R$. Then, Theorem~\ref{holdermainresult} yields that~$\G_\delta(Du_\epsilon)$ is Hölder continuous in~$\overline{B}_\rho(x_0)$ with a Hölder exponent~$\alpha_\delta\in(0,1)$ and a Hölder constant~$C_\delta$, both depending on the data
\begin{align*}
       \alpha_\delta &= \alpha_\delta(C_\F(\delta),\hat{C}_\F,\delta,\|f\|_{L^{n+\sigma}(B_R)},M,N,n,R,r_E,\sigma) \in(0,1), \\
         C_\delta &= C_\delta(C_\F(\delta),\hat{C}_\F,\delta,\|f\|_{L^{n+\sigma}(B_R)},M,N,n,R,R_E,r_E,\sigma) \geq 1,
\end{align*}
Due to Lemma~\ref{lemgdeltakvgz}, we know that~$\G_\delta(Du_\epsilon)\to\G_\delta(Du)$ in~$L^2(B_R)$ as~$\epsilon\downarrow 0$, yielding the existence of a subsequence~$(\epsilon_i)_{i\in\N}\subset (0,1]$ with~$\G_\delta(Du_{\epsilon_i}) \to \G_\delta(Du)$ a.e. in~$B_R$ as~$\epsilon_i\downarrow 0$. Since the family~$(\G_\delta(Du_\epsilon))_{\epsilon\in(0,1]}$ is locally uniformly equicontinuous and also locally uniformly bounded in~$B_R$, an application of Arzelà-Ascoli's theorem and passing to the limit~$\epsilon\downarrow 0$ yields that the limit function~$\G_\delta(Du)$ itself is Hölder continuous in~$\overline{B}_\rho(x_0)$ with a Hölder constant~$C_\delta$ and Hölder exponent~$\alpha_\delta\in(0,1)$. Furthermore, there holds~$\G_\delta(Du)\to\G(Du)$ uniformly in~$\overline{B}_\rho(x_0)$, which follows from the estimate
\begin{align*}
    |\G_\delta(Du) - \G(Du)| &= \bigg| \frac{(|Du|_E - (1+\delta))_+}{|Du|_E}Du - \frac{(|Du|_E-1)_+}{|Du|_E}Du \bigg| \\
    &\leq \frac{|Du|}{|Du|_E} \big|(|Du|_E-(1+\delta))_+ - (|Du|_E-1)_+ \big| \\
    &\leq \mbox{$\frac{1}{r_E}$} \big|(|Du|_E-(1+\delta))_+ - (|Du|_E-1)_+ \big|  \\
    &\leq \mbox{$\frac{1}{r_E}$} \delta
\end{align*}
a.e. in~$\overline{B}_\rho(x_0)$. In turn, we used~\eqref{betragminkowski} and Lemma~\ref{lem:algineq}. Therefore, the limit function~$\G(Du)$ itself is uniformly continuous in~$\overline{B}_\rho(x_0)$. Let~$\epsilon >0$ be an arbitrary parameter, which we note is at this stage no longer associated with the approximation parameter mentioned previously. Then, there exists~$\kappa>0$ with
\begin{align} \label{modulusofcontinuity}
    |\G(Du(x))-\G(Du(y))| <\epsilon \qquad \mbox{for any~$x,y \in \overline{B}_\rho(x_0)$ with~$|x-y|<\kappa$}.
\end{align}
Further, let~$\mathcal{K}\in C^0(\R^n)$ denote an arbitrary continuous function that vanishes inside the bounded and convex set of degeneracy~$E\subset\R^n$ with~$0\in\inn{E}$. Due to the fact that~$u\in W^{1,\infty}_{\loc}(\Omega)$, there exists some positive and finite constant~$0<M<\infty$ with~$|Du| \leq M$ a.e. in~$\overline{B}_\rho(x_0)$.
Consequently,~$\mathcal{K}$ exhibits a modulus of continuity~$\omega_M\colon \R_{\geq 0} \to\R_{\geq 0}$ on~$\overline{B}_M$, i.e. there holds
$$|\mathcal{K}(\xi)-\mathcal{K}(\eta)|\leq \omega_M(|\xi-\eta|) \qquad\mbox{for any~$\xi,\eta \in \overline{B}_M$}.$$
We distinguish between two scenarios. Firstly, let us assume that there holds~$|Du(x)|_E\leq 1+\sqrt{\epsilon}$. If~$|Du(x)|_E\geq 1$, we exploit the fact that~$\mathcal{K}\equiv 0$ on~$E$ and also use the homogeneity of the Minkowski functional~$|\cdot|_E$ to obtain
\begin{align*}
    |\mathcal{K}(Du(x))| &= \bigg| \mathcal{K}(Du(x))-\mathcal{K}\Big(\frac{Du(x)}{|Du(x)|_E}\Big) \bigg| \\
    &\leq \omega_M\bigg(\Big|Du(x)-\frac{Du(x)}{|Du(x)|_E}\Big|\bigg) \\
    &= \omega_M\bigg(\frac{|Du(x)|}{|Du(x)|_E}(|Du(x)|_E-1)\bigg) \\
    &\leq \omega_M(R_E\sqrt{\epsilon}).
\end{align*}
If there holds~$|Du(x)|_E\leq 1$, then the above estimate clearly holds true as well. Due to~\eqref{modulusofcontinuity} being at our disposal, we further have
\begin{align*}
    (|Du(y)|_E-1)_+ &= |\G(Du(y))|_E \\
    &\leq |\G(Du(y)) - \G(Du(x))|_E + |\G(Du(x))|_E \\
    &< \mbox{$\frac{\epsilon}{r_E}$} + \sqrt{\epsilon} \\
    &\leq 2\max\{\mbox{$\frac{1}{r_E}$},1 \}\sqrt{\epsilon},
\end{align*}
where we used the triangle inequality~\eqref{lem:minkowskitriangle} and once again~\eqref{betragminkowski}. This implies
$$|Du(y)|_E \leq 1+2\max\{ \mbox{$\frac{1}{r_E}$},1 \}\sqrt{\epsilon},$$ such that by following the very same steps as above and exploiting~\eqref{betragminkowski}, we also obtain that~$|\mathcal{K}(Du(y))| \leq \omega_M\big(2 R_E\max\big\{\frac{1}{r_E},1 \big\}\sqrt{\epsilon} \big)$. Combining both estimates for~$|\mathcal{K}(Du(x))|$ and~$|\mathcal{K}(Du(y))|$, we thus find
$$|\mathcal{K}(Du(x))-\mathcal{K}(Du(y))| \leq \omega_M(R_E\sqrt{\epsilon}) + \omega_M\big(2 R_E\max\{ \mbox{$\frac{1}{r_E}$},1 \} \sqrt{\epsilon} \big) \leq 2 \omega_M\big(2 R_E\max\{ \mbox{$\frac{1}{r_E}$},1\} \sqrt{\epsilon}\big), $$
which yields the continuity of~$\mathcal{K}(Du)$ in the first case. If the second alternative~$|Du(x)|_E>1+\sqrt{\epsilon}$ applies, we may employ the second assertion of Lemma~\ref{gdeltalem} to obtain that there holds
\begin{align*}
    |Du(x)-Du(y)| &\leq \frac{C(R_E,r_E)}{\sqrt{\epsilon}}|\G(Du(x))-\G(Du(y))| \\
    &\leq C(R_E,r_E) \sqrt{\epsilon}.
\end{align*}
In turn, this implies
$$|\mathcal{K}(Du(x)) - \mathcal{K}(Du(y))| \leq \omega_M(C(R_E,r_E) \sqrt{\epsilon}),$$
which verifies that~$\mathcal{K}(Du)$ is continuous in~$\overline{B}_\rho(x_0)$. As all the balls~$B_\rho(x_0)\Subset B_R$ were chosen arbitrarily, we obtain that~$\mathcal{K}(Du)$ is also continuous within the entire domain~$\Omega$.

\end{proof}


\begin{remark} \label{remarknachhauptresultatproof}
    \upshape 
It is noteworthy that during the proof of Theorem~\ref{hauptresultat}, the control of the quantitative Hölder exponent is lost in the limit as~$\delta \downarrow 0$, leading to uniform continuity of the limit function~$\G(Du)$ locally in~$B_R$ instead of desirable Hölder continuity. To our knowledge it remains an open question even for the model equation~\eqref{prototype} what the best possible modulus of continuity of~$\G(Du)$ may be and whether~$\G(Du)$ is actually locally Hölder continuous in~$B_R$, posing an intriguing question for further exploration in future research, cf. the comments right after~\cite[Theorem~1.1]{bogelein2023higher}. However, in general,~$\mathcal{K}(Du)$ fails to be Hölder continuous for certain integrands~$\F$ in~$\Omega$ for any Hölder exponent~$\alpha \in (0,1)$. In fact, it has been demonstrated in~\cite[Remark~1.3]{colombo2017regularity} that for a certain integrand function~$\F$ there exist Lipschitz continuous functions~$\mathcal{K} \in C^0(\R^n, \R)$ that vanish inside the set of degeneracy~$E$, for which~$\mathcal{K}(Du)$ is not Hölder continuous.
\end{remark}


\section{The non-degenerate regime} \label{sec:nondegenerate}
The objective in this section is to establish the proof of Proposition~\ref{nondegenerateproposition} by following a series of steps. The underlying strategy lies in exploiting the measure-theoretic information~\eqref{nondegeneratemeascond} that characterizes the non-degenerate regime. The latter allows us to derive a lower bound for~$|Du_\epsilon|_E$ on the smaller ball~$B_{\frac{\rho}{2}}(x_0)$, provided the parameter~$\nu\in(0,\frac{1}{4}]$ incorporated in the measure-theoretic information~\eqref{nondegeneratemeascond} is chosen sufficiently small in dependence of the given data. Once the lower bound for~$|Du_\epsilon|_E$ has been established, we are able to state a quantitative excess-decay estimate for linear elliptic equations with bounded coefficients by exploiting a result found in~\cite{gilbargtrudinger}. Recall that in the course of this section, the set of assumptions~\eqref{schrankeeins},~\eqref{schrankezwei},~\eqref{deltamu}, and~\eqref{nondegeneratemeascond} is at our disposal.\,\\


We begin by exploiting the measure-theoretic information~\eqref{nondegeneratemeascond}, enabling us to derive a lower bound for~$|Du_\epsilon|_E $ on the smaller ball~$B_{\frac{\rho}{2}}(x_0)$.


\begin{myproposition} \label{lowerboundprop}
Let~$\delta,\epsilon\in(0,1]$ and~$u_\epsilon$ denote the weak solution to~\eqref{approx}. Then, there exists a parameter $\nu=\nu(C_\F(\delta),\hat{C}_\F,\delta,\|f\|_{L^{n+\sigma}(B_R)},M,N,n,R,r_E,\sigma) \in(0,\frac{1}{4}]$, such that: if there holds
\begin{equation} \label{lowerboundvs}
   | B_\rho(x_0)  \cap  \{ \partial_{e^*}u_\epsilon -(1+\delta) \leq (1-\nu) \mu \}|  \leq \nu |B_\rho(x_0)|
\end{equation}
for at least one~$e^*\in\partial E^*$ , then we have
\begin{equation} \label{lowerboundest}
    |D u_\epsilon|_E \geq 1+\delta+\frac{\mu}{4} \qquad\mbox{a.e. in~$B_{\frac{\rho}{2}}(x_0)$}.
\end{equation}
\end{myproposition}


\begin{proof}
To abbreviate notation, we shall simply write~$u=u_\epsilon$ in what follows. The proof consists of three major steps and follows an approach taken by Kuusi and Mingione~\cite[Proof of Proposition 3.7]{kuusi2013gradient} in a different setting. \,\\

\textbf{Step 1:} We commence by rescaling the solution~$u$ to the unit ball. For this matter, we introduce the rescaled mappings
\begin{alignat*}{2}
    \Tilde{u}(x) &\coloneqq \frac{u(x_0+\rho x)}{\rho},  && \qquad x\in B_{1}, \\
    \Tilde{\F}_\epsilon(x,\xi) &\coloneqq \hat{\F}_\epsilon(x_0+\rho x,\xi),  && \qquad x\in B_{1},\,\xi\in\R^n, \\
    \Tilde{\mathcal{B}}_\epsilon(x,\xi) &\coloneqq \hat{\mathcal{B}}_\epsilon(x_0+\rho x,\xi),  && \qquad x\in B_{1},\,\xi\in\R^n, \\
    \Tilde{f} (x) &\coloneqq \rho f(x_0+\rho x), && \qquad x\in B_{1}.
\end{alignat*}
For the following discussion, we shall also use the notation~$\Tilde{\h}_\epsilon(x,\xi) \coloneqq \nabla\Tilde{\F}_\epsilon(x,\xi)$. Consequently, since~$u$ solves~\eqref{approx} weakly on~$B_\rho(x_0)$, the rescaled function~$\Tilde{u}$ is a weak solution to
\begin{equation*} 
    \divv \Tilde{\h}_\epsilon(x,D\Tilde{u}) = \Tilde{f} \qquad \mbox{in~$B_{1}$},
\end{equation*}
i.e. there holds
\begin{align} \label{rescaledsolutionweakform}
    \int_{B_1} \langle \Tilde{\h}_\epsilon(x,D\Tilde{u}),D\phi \rangle\,\dx = - \int_{B_1} \Tilde{f} \phi\,\dx
\end{align}
for any test function~$\phi\in C^\infty_0(B_1)$. Moreover, the measure-theoretic assumption~\eqref{lowerboundvs} translates to
\begin{align} \label{lowerboundscaled}
   |B_1 \cap \{ \partial_{\Tilde{e}^*}\Tilde{u} -(1+\delta) \leq (1-\nu)\mu  \}| \leq \nu |B_1|
\end{align} 
on~$B_1$, where~$\Tilde{e}^*\in \partial E^*$ denotes such an element that satisfies~\eqref{lowerboundvs}. \,\\

\textbf{Step 2:} Next, we aim to establish a Caccioppoli-type estimate for~$\Tilde{u}$. Let~$\kappa\in(1+\delta,1+\delta+\mu)$ and~$e^*\in\partial E^*$ be chosen arbitrarily. We test the weak form~\eqref{rescaledsolutionweakform} with the test function~$\partial_{e^*} \phi$ for some~$\phi\in C^1_0(B_1)$ and subsequently integrate by parts to obtain
\begin{align} \label{rescaledintbyparts}
    \int_{B_1} \langle \partial_{e^*}[\Tilde{\h}_\epsilon(x,D\Tilde{u})], D\phi \rangle\,\dx =  \int_{B_1} \Tilde{f} \partial_{e^*}\phi\,\dx.
\end{align}
Here, we choose the test function
$$\phi_{e^*} = \zeta^2 ( \partial_{e^*} \Tilde{u} - a)_+ ( \partial_{e^*} \Tilde{u} - \kappa)_-,$$
where we abbreviated~$a\coloneqq 1+\delta$, and~$\zeta\in C^1_0(B_1,[0,1])$ denotes a standard non-negative smooth cut-off function. This choice of test function can be justified by an approximation argument. With this choice of test function, we have
\begin{align*}
    \partial_j \phi_{e^*} &= 2\zeta \partial_j \zeta (\partial_{e^*} \Tilde{u} - a)_+ (\partial_{e^*} \Tilde{u} - \kappa)_- + \zeta^2 \partial_{j}\partial_{e^*} \Tilde{u} (\partial_{e^*}\Tilde{u} -\kappa)_-  \bigchi_{\{ \partial_{e^*}\Tilde{u}>a\}} \\
             &\quad + \zeta^2 \partial_{j}\partial_{e^*}\Tilde{u} (\partial_{e^*}\Tilde{u}-a)_+ \bigchi_{\{\partial_{e^*}\Tilde{u}<\kappa\}}   
\end{align*} 
for a.e.~$x\in B_1$ and any~$j\in\{1,\ldots,n\}$. The left-hand side of~\eqref{rescaledintbyparts} yields the following two terms
\begin{align*}
    \int_{B_1}  & \langle \partial_{e^*}[\Tilde{\h}_\epsilon(x,D\Tilde{u})],D\phi_{e^*} \rangle\,\dx \\
    &= \int_{B_1}  \Tilde{\mathcal{B}}_\epsilon(x,D\Tilde{u})(\partial_{e^*} D\tilde{u},D\phi_{e^*})\,\dx + \int_{B_1}  \langle \partial_{e^*} \Tilde{\h}_\epsilon(x,D\Tilde{u}),D\phi_{e^*} \rangle\,\dx \\
    &\eqqcolon \foo{I} + \foo{II}.
\end{align*}
We commence by treating the first term~$\foo{I}$, which is given by
\begin{align} \label{lowerboundfirstterm}
    \foo{I} &= \int_{B_1} 2\zeta \Tilde{\mathcal{B}}_\epsilon(x,D\Tilde{u})(\partial_{e^*} D\tilde{u},D \zeta) (\partial_{e^*}\Tilde{u} - a)_+ (\partial_{e^*}\Tilde{u}-\kappa)_- \,\dx \\
    &\quad + \int_{B_1} \zeta^2 \Tilde{\mathcal{B}}_\epsilon(x,D\Tilde{u})(\partial_{e^*} D\tilde{u},\partial_{e^*} D\tilde{u}) \big( (\partial_{e^*}\Tilde{u}-\kappa)_-\bigchi_{\{\partial_{e^*}\Tilde{u}>a\}} + (\partial_{e^*}\Tilde{u} - a)_+\bigchi_{\{\partial_{e^*}\Tilde{u}<\kappa\}} \big) \,\dx. \nonumber 
\end{align}
At this point, we note that due to the set inclusion
$$ \{x\in B_1:\partial_{e^*}\Tilde{u}\geq 1+\delta\} \subset \{x\in B_1: |D\Tilde{u}|_E \geq 1+\delta \}, $$
we are in position to employ Lemma~\ref{bilinearelliptic} and recall~\eqref{hatfvoraussetzung}. We exploit the definition of~$C_\F$ according to~\eqref{schrankehessian} and also~$C_\F(\delta)$ in~\eqref{hesseschrankedelta}, to bound the preceding quantity~$\foo{I}$ further below by
\begin{align} \label{lowerfirsttermsummed}
   \foo{I} &\geq  \big(\epsilon+\Tilde{C}(C_\F(\delta),\delta)\big) \int_{B_1} \zeta^2 |\partial_{e^*}D\Tilde{u}|^2 \big( (\partial_{e^*}\Tilde{u}-\kappa)_-\bigchi_{\{\partial_{e^*}\Tilde{u}>a\}} + (\partial_{e^*}\Tilde{u} - a)_+\bigchi_{\{\partial_{e^*}\Tilde{u}<\kappa\}} \big) \,\dx \\
   &\quad -  C_\F(\delta) \int_{B_1} \zeta |\partial_{e^*} D\Tilde{u}| |D\zeta| (\partial_{e^*}\Tilde{u} - a)_+ (\partial_{e^*}\Tilde{u}-\kappa)_- \,\dx   \nonumber \\
   &\geq \big(\epsilon+\mbox{$\frac{3}{4}$}\Tilde{C}\big) \int_{B_1} \zeta^2 |\partial_{e^*}D\Tilde{u}|^2 \big( (\partial_{e^*}\Tilde{u}-\kappa)_-\bigchi_{\{\partial_{e^*}\Tilde{u}>a\}} + (\partial_{e^*}\Tilde{u} - a)_+\bigchi_{\{\partial_{e^*}\Tilde{u}<\kappa\}} \big) \,\dx \nonumber \\
   & \quad - C \int_{B_1}  |D\zeta|^2 (\partial_{e^*}\Tilde{u} - a)^2_+ \bigchi_{\{\partial_{e^*}\Tilde{u}<\kappa\}} \,\dx, \nonumber 
\end{align}
where~$C=C(C_\F(\delta),\delta)$ and~$\Tilde{C}=\Tilde{C}(C_\F(\delta),\delta)$ denotes some positive constants that depend on~$C_\F(\delta)$ and~$\delta\in(0,1]$. In the last step, we additionally employed Young's inequality and reabsorbed one of the two arising terms with constant~$\frac{1}{4}\Tilde{C}(\delta)$ into the first term. We now turn our attention to the second quantity~$\foo{II}$. There, we employ the Cauchy-Schwarz inequality and exploit the Lipschitz bound~\eqref{hatflipschitz} as well as the representation~\eqref{representation}, to obtain the estimate
\begin{align*}
    \foo{II} & = \int_{B_1} \langle \partial_{e^*} \Tilde{\h}_\epsilon(x,D\Tilde{u}),D\phi_{e^*} \rangle\,\dx \\
    &\leq C \int_{B_1} |D\phi_{e^*}| \,\dx \\
    &\leq C \int_{B_1} \zeta |D\zeta| (\partial_{e^*}\Tilde{u} - a)_+ (\partial_{e^*}\Tilde{u} - \kappa)_-\,\dx \\ 
    &\quad + C \int_{B_1} \zeta^2 |\partial_{e^*}D\Tilde{u}| \big( (\partial_{e^*}\Tilde{u}-\kappa)_-\bigchi_{\{\partial_{e^*}\Tilde{u}>a\}} + (\partial_{e^*}\Tilde{u} - a)_+\bigchi_{\{\partial_{e^*}\Tilde{u}<\kappa\}} \big)\,\dx 
\end{align*}
with constant~$C=C(\hat{C}_\F,N,n,r_E)$. In turn, we used that~$|v_i|\leq R_E$ for any~$i\in 1,\ldots,n$, where~$v_i$ denote the coefficients of~$e^*\in\partial E^*$ with respect to the standard base. Therefore, after employing Young's inequality, we achieve
\begin{align} \label{lowersecondtermsummed}
   \foo{II} & \leq C \int_{B_1}  \zeta |D\zeta| (\partial_{e^*}\Tilde{u} - a)_+ (\partial_{e^*}\Tilde{u}-\kappa)_-\,\dx  \\
    &\quad + \mbox{$\frac{1}{4}$}\Tilde{C} \int_{B_1} \zeta^2 |\partial_{e^*}D\Tilde{u}|^2 \big( (\partial_{e^*}\Tilde{u}-\kappa)_-\bigchi_{\{\partial_{e^*}\Tilde{u}>a\}} + (\partial_{e^*}\Tilde{u} - a)_+\bigchi_{\{\partial_{e^*}\Tilde{u}<\kappa\}} \big) \,\dx \nonumber \\
    &\quad + C \int_{B_1} \zeta^2 \big( (\partial_{e^*}\Tilde{u}-\kappa)_-\bigchi_{\{\partial_{e^*}\Tilde{u}>a\}} + (\partial_{e^*}\Tilde{u} - a)_+\bigchi_{\{\partial_{e^*}\Tilde{u}<\kappa\}} \big) \,\dx \nonumber
\end{align}
with~$C=C(C_\F(\delta),\hat{C}_\F,\delta,N,n,r_E)$. Let us now turn our attention to the right-hand side in~\eqref{rescaledintbyparts} involving the datum~$\Tilde{f}$. In a similar fashion to the second quantity~$\foo{II}$, we bound the right-hand side further above by
\begin{align} \label{lowerdatumtermsummed}
    \foo{III} &\coloneqq \int_{B_1} \Tilde{f} \partial_{e^*}\phi_{e^*} \,\dx \\
    &\leq \int_{B_1} 2 |\Tilde{f}| \zeta |\partial_{e^*}\zeta|  (\partial_{e^*}\Tilde{u} - a)_+ (\partial_{e^*}\Tilde{u}-\kappa)_-\,\dx \nonumber \\
    &\quad + \int_{B_1} |\Tilde{f}| \zeta^2 |\partial_{e^*}D\Tilde{u}| \big( (\partial_{e^*}\Tilde{u}-\kappa)_-\bigchi_{\{\partial_{e^*}\Tilde{u}>a\}} + (\partial_{e^*}\Tilde{u} - a)_+\bigchi_{\{\partial_{e^*}\Tilde{u}<\kappa\}} \big) \,\dx \nonumber \\
    &\leq C \int_{B_1} |\Tilde{f}| \zeta |D\zeta|  (\partial_{e^*}\Tilde{u} - a)_+ (\partial_{e^*}\Tilde{u}-\kappa)_-\,\dx \nonumber \\
   &\quad + \mbox{$\frac{1}{4}$}\Tilde{C} \int_{B_1} \zeta^2 |\partial_{e^*}D\Tilde{u}|^2 \big( (\partial_{e^*}\Tilde{u}-\kappa)_-\bigchi_{\{\partial_{e^*}\Tilde{u}>a\}} + (\partial_{e^*}\Tilde{u} - a)_+\bigchi_{\{\partial_{e^*}\Tilde{u}<\kappa\}} \big) \,\dx \nonumber \\
    &\quad + C \int_{B_1} \zeta^2 |\Tilde{f}|^2 \big( (\partial_{e^*}\Tilde{u}-\kappa)_-\bigchi_{\{\partial_{e^*}\Tilde{u}>a\}} + (\partial_{e^*}\Tilde{u} - a)_+\bigchi_{\{\partial_{e^*}\Tilde{u}<\kappa\}} \big) \,\dx \nonumber
\end{align}
with a constant~$C=C(C_\F(\delta),\delta)$. By combining our estimates~\eqref{lowerfirsttermsummed},~\eqref{lowersecondtermsummed}, and also~\eqref{lowerdatumtermsummed}, there holds 
\begin{align} \label{lowerboundcombined}
    &\big(\epsilon+\mbox{$\frac{3}{4}$}\Tilde{C}(C_\F(\delta),\delta)\big) \int_{B_1} \zeta^2 |\partial_{e^*} D\Tilde{u}|^2 \big( (\partial_{e^*}\Tilde{u}-\kappa)_-\bigchi_{\{\partial_{e^*}\Tilde{u}>a\}} + (\partial_{e^*}\Tilde{u} - a)_+\bigchi_{\{\partial_{e^*}\Tilde{u}<\kappa\}} \big) \,\dx  \\
    &\leq \mbox{$\frac{1}{2}$}\Tilde{C}(C_\F(\delta),\delta) \int_{B_1} \zeta^2 |\partial_{e^*} D\Tilde{u}|^2 \big( (\partial_{e^*}\Tilde{u}-\kappa)_-\bigchi_{\{\partial_{e^*}\Tilde{u}>a\}} + (\partial_{e^*}\Tilde{u} - a)_+\bigchi_{\{\partial_{e^*}\Tilde{u}<\kappa\}} \big) \,\dx  \nonumber \\
    &\quad + C \int_{B_1}  \zeta |D\zeta| (\partial_{e^*}\Tilde{u} - a)_+ (\partial_{e^*}\Tilde{u}-\kappa)_- \big(1+|\Tilde{f}|\big)\,\dx \nonumber \\
    &\quad + C \int_{B_1} \zeta^2 (\partial_{e^*}\Tilde{u} - a)_+\big( (\partial_{e^*}\Tilde{u}-\kappa)_- + \bigchi_{\{\partial_{e^*}\Tilde{u}<\kappa\}} \big) \big(1+|\Tilde{f}|^2 \big) \,\dx \nonumber \\
    &\quad + C \int_{B_1}  |D\zeta|^2 (\partial_{e^*}\Tilde{u} - a)^2_+ \bigchi_{\{\partial_{e^*}\Tilde{u}<\kappa\}} \,\dx \nonumber
\end{align}
with~$C=C(C_\F(\delta),\hat{C}_\F,\delta,N,n,r_E)$. Eventually, we reabsorb the quantity involving second order derivatives of~$\Tilde{u}$ from the right-hand side of~\eqref{lowerboundcombined} into the left-hand side. This leads to the following estimate
\begin{align*}
     &\big(\epsilon+\Tilde{C}(C_\F(\delta),\delta)\big) \int_{B_1} \zeta^2 |\partial_{e^*} D\Tilde{u}|^2 \big( (\partial_{e^*}\Tilde{u}-\kappa)_-\bigchi_{\{\partial_{e^*}\Tilde{u}>a\}} + (\partial_{e^*}\Tilde{u} - a)_+\bigchi_{\{\partial_{e^*}\Tilde{u}<\kappa\}} \big) \,\dx  \\
    &\leq C \int_{B_1}  \zeta |D\zeta| (\partial_{e^*}\Tilde{u} - a)_+ (\partial_{e^*}\Tilde{u}-\kappa)_- \big(1+|\Tilde{f}|\big)\,\dx \nonumber \\
    &\quad + C \int_{B_1} \zeta^2 (\partial_{e^*}\Tilde{u} - a)_+\big( (\partial_{e^*}\Tilde{u}-\kappa)_- + \bigchi_{\{\partial_{e^*}\Tilde{u}<\kappa\}} \big) \big(1+|\Tilde{f}|^2 \big) \,\dx \nonumber \\
    &\quad + C \int_{B_1}  |D\zeta|^2 (\partial_{e^*}\Tilde{u} - a)^2_+ \bigchi_{\{\partial_{e^*}\Tilde{u}<\kappa\}} \,\dx \nonumber
\end{align*}
with~$C=C(C_\F(\delta),\hat{C}_\F,\delta,N,n,r_E)$. Consequently, by further estimating the right-hand side of the preceding inequality with Hölder's inequality and by taking into account assumption~\eqref{schrankeeins} and also the fact that~$\Tilde{f}\in L^{n+\sigma}(B_1)$, we end up with
\begin{align*} 
  \int_{B_1} & \zeta^2 |\partial_{e^*} D\Tilde{u}|^2 (\partial_{e^*} \Tilde{u}-\kappa)_- \bigchi_{\{\partial_{e^*} \Tilde{u}>a\}}\,\dx + \int_{B_1} \zeta^2 |\partial_{e^*} D\Tilde{u}|^2  (\partial_{e^*}\Tilde{u}-a)_+ \bigchi_{\{\partial_{e^*} \Tilde{u}<\kappa\}}\,\dx \\ 
  &\leq C \big(\|D\zeta\|_{L^\infty(B_1)} + \|D\zeta\|^2_{L^\infty(B_1)}\big) |B_1 \cap \{1+\delta < \partial_{e^*} \Tilde{u}<\kappa\}|^{1-\frac{2}{n+\sigma}}
\end{align*}
with~$C=C(C_\F(\delta),\hat{C}_\F,\delta,M,N,n,R,r_E,\sigma)$, where we additionally exploited the fact that
$$\|\Tilde{f}\|_{L^{n+\sigma}(B_1)} \leq \rho^{\beta} \|f\|_{L^{n+\sigma}(B_{\rho}(x_0))} \leq R^{\beta} \|f\|_{L^{n+\sigma}(B_R)},$$
which follows from a change of variables. Note that the preceding estimate in particular implies 
\begin{align} \label{lowerboundee}
    &\|D [(\partial_{e^*} \Tilde{u} - \kappa)_-  (\partial_{e^*} \Tilde{u} - a)_+ \zeta]\|^2_{L^2(B_1)} \\
    &\quad \leq  C \big(\|\zeta\|^2_{L^\infty(B_1)} + \|D\zeta\|^2_{L^\infty(B_1)}\big) |B_1 \cap \{1+\delta < \partial_{e^*} \Tilde{u} <\kappa\}|^{1-\frac{2}{n+\sigma}} \nonumber
\end{align}
with~$C=C(C_\F(\delta),\hat{C}_\F,\delta,\|f\|_{L^{n+\sigma}(B_R)},M,N,n,R,r_E,\sigma)$. \,\\

\textbf{Step 3:} In this final step, we aim to apply the geometric convergence Lemma~\ref{geometriclem}. For this matter, we consider the sequences of levels
\begin{alignat*}{2}
\delta_m &\coloneqq 1+\delta+\frac{\mu}{8}\big(1-\frac{1}{2^{m}}\Big),\qquad && m\in\N_0, \\
    \kappa_m &\coloneqq 1+\delta+\frac{\mu}{4}\Big(1+\frac{1}{2^m}\Big),\qquad && m\in\N_0,
\end{alignat*}
so that for any~$m\in\N_0$ there holds
$$ 1+\delta \leq \delta_m \leq 1+\delta+\frac{\mu}{8} \leq 1+\delta+\frac{\mu}{4} \leq \kappa_m \leq 1+\delta+\frac{\mu}{2}.$$
Moreover, let us introduce the sequence of radii
$$r_m \coloneqq \frac{1}{2}+\frac{1}{2^{m+1}},\qquad  m\in\N_0,$$
and consider standard smooth cut-off functions~$\zeta_m \in C^1_0(B_{r_m},[0,1])$ satisfying
$$\zeta_m \equiv 1 \quad \mbox{on~$B_{r_{m+1}}$} \qquad \& \qquad \|D\zeta_m\|_{L^\infty(B_{r_m})} \leq 2^{m+3}$$
for any~$m\in\N_0$. It is noteworthy that the upper bound for~$D\zeta_m$ is independent of the size of the radius~$\rho$, a consequence of the rescaling of~$u_\epsilon$ conducted in Step $1$ of the proof. By abbreviating
$$Y_m \coloneqq |A_m|\coloneqq |B_{r_m} \cap \{\delta_m < \partial_{e^*} \Tilde{u} <\kappa_m\}|$$
and choosing the test function~$\zeta = \zeta_m$ as well as the level~$\kappa=\kappa_m$ in the estimate~\eqref{lowerboundee}, and by replacing the parameter~$a=1+\delta$ with~$\delta_m$, we obtain that 
$$\delta_{m+1} - \delta_m = \frac{\mu}{2^{m+4}}, \qquad k_m - k_{m+1} = \frac{\mu}{2^{m+3}}$$
for any~$m\in\N_0$. Exploiting this estimate and the choice of the cut-off functions~$\zeta_m$, and further employing the estimate~\eqref{lowerboundee} as well as the Poincar\'{e}-type inequality from Lemma~\ref{poincarelem}, we achieve
\begin{align*}
    (\delta_{m+1}-\delta_m)^2  (\kappa_m  - \kappa_{m+1})^2Y_{m+1} 
    &\leq \| (\partial_{e^*} \Tilde{u}-\delta_m)_+ (\partial_{e^*} \Tilde{u}  -\kappa_m)_-  \zeta_m\|^2_{L^2(A_{m+1})} \\
    &\leq  \| (\partial_{e^*} \Tilde{u}-\delta_m)_+ (\partial_{e^*} \Tilde{u}  -\kappa_m)_-  \zeta_m\|^2_{L^2(B_{r_m})}  \\
    &\leq  C(n)  |A_m|^{\frac{2}{n}} \|D[(\partial_{e^*} \Tilde{u}-\delta_m)_+ (\partial_{e^*} \Tilde{u}  -\kappa_m)_-   \zeta_m]\|^2_{L^2(B_{r_m})} \\
    &\leq C 2^{2(m+3)} |A_m|^{\frac{2}{n}} |A_m|^{1-\frac{2}{n+\sigma}} \\
    &= C 2^{2(m+3)} Y^{1+\frac{2\beta}{n}}_m 
\end{align*}
with~$C=C(C_\F(\delta),\hat{C}_\F,\delta,\|f\|_{L^{n+\sigma}(B_R)},M,N,n,R,R_E,\sigma)$. Due to the bound for~$\delta_{m+1} - \delta_m$ and~$\kappa_m - \kappa_{m+1}$, there holds
$$Y_{m+1} \leq C 2^{6(m+3)} Y^{1+\frac{2\beta}{n}}_m \leq C 2^{6m} Y^{1+\frac{2\beta}{n}}_m$$
with~$C=C(C_\F(\delta),\hat{C}_\F,\delta,\|f\|_{L^{n+\sigma}(B_R)},M,N,n,R,r_E,\sigma)$, where we additionally exploited the estimate~\eqref{deltamu} in order to estimate~$\frac{1}{\mu^4}\leq \frac{1}{\delta^4}$. At this point, we fix~$e^*\in\partial E^*$ through the specific choice of~$\Tilde{e}^*\in \partial E^*$, which denotes the parameter satisfying~\eqref{lowerboundscaled}. Finally, we apply the geometric convergence Lemma~\ref{geometriclem} with the choices~$b\coloneqq 2^6$ and~$\kappa \coloneqq \frac{2\beta}{n}$. The assumption of the lemma, which we recall is given by
$$Y_0 \leq C^{-\frac{1}{\kappa}}b^{-\frac{1}{\kappa2}},$$  
is satisfied provided the parameter~$\nu\in(0,\frac{1}{4}]$ is chosen small enough in dependence of the data
$$ (C_\F(\delta),\hat{C}_\F,\delta,\|f\|_{L^{n+\sigma}(B_R)},M,N,n,R,r_E,\sigma). $$
This matter of fact can be inferred from the bound
\begin{align*}
    Y_0 &= |B_1 \cap \{1+\delta < \partial_{\Tilde{e}^*}\Tilde{u}<1+\delta+\mbox{$\frac{\mu}{2}$} \}| \\
    &\leq |B_1 \cap \{ \partial_{\Tilde{e}^*}\Tilde{u} <1+\delta+\textstyle{\frac{\mu}{2}} \}| \\
    &\leq \nu|B_1|
\end{align*}
for~$\nu \in (0,\frac{1}{4}]$. An application of Lemma~\ref{geometriclem} yields the convergence~$Y_m \to 0$ as~$m\to\infty$, which implies
$$ \big|B_{\frac{1}{2}}\cap \{1+\delta+\mbox{$\frac{\mu}{8}$}\leq \partial_{\Tilde{e}^*}\Tilde{u} \leq 1+\delta+ \mbox{$\frac{\mu}{4}$} \} \big|=0.$$
We now argue that either there holds~$\partial_{\Tilde{e}^*}\Tilde{u} \leq 1 + \delta+\frac{\mu}{8}$ a.e. in~$B_{\frac{1}{2}}$, or~$\partial_{\Tilde{e}^*}\Tilde{u} \geq 1 + \delta + \frac{\mu}{4}$ is satisfied a.e. in~$B_{\frac{1}{2}}$, while the first alternative cannot arise due to the measure-theoretic information~\eqref{lowerboundscaled}. We exploit the fact that~$\partial_{\Tilde{e}^*}\Tilde{u}$ is weakly differentiable due to the regularity~$\Tilde{u}\in W^{1,\infty}(B_1)\cap W^{2,2}(B_1)$ and the fact that any Sobolev function is locally absolutely continuous along any coordinate axis. Let us therefore argue by contradiction and assume without any loss of generality that there holds~$\partial_{\Tilde{e}^*}\Tilde{u} \leq 1+\delta+\frac{\mu}{8}$ a.e. in~$\Tilde{B}$ as well as~$\partial_{\Tilde{e}^*}\Tilde{u} \geq 1+\delta+\frac{\mu}{4}$ a.e. in~$B_{\frac{1}{2}} \setminus \Tilde{B}$ on some subset~$\Tilde{B}\subset B_{\frac{1}{2}}$. But then, there exists at least one index~$i\in\{1,\ldots,n\}$ and a direction~$e_i \in\R^n$, a point~$x\in \partial \Tilde{B}$, and a small parameter~$\sigma>0$, such that the mapping
$$(-\sigma,\sigma) \ni s \mapsto \partial_{\Tilde{e}^*}\Tilde{u}(x+s e_i )$$
exhibits a discontinuity at~$s=0$. Referring to the result in~\cite[Theorem~4.20 (i) resp. Theorem~4.21 (i)]{evansgariepy}, we obtain a contradiction, due to the fact that~$\partial_{\Tilde{e}^*}\Tilde{u}$ is locally absolutely continuous along any coordinate axis, as explained above. This excludes the case where~$\partial_{\Tilde{e}^*}\Tilde{u}\leq 1+\delta+\frac{\mu}{8}$ a.e. on a subset of~$B_{\frac{1}{2}}$ and thus yields~$\partial_{\Tilde{e}^*}\Tilde{u} \geq 1+\delta+\frac{\mu}{4}$ a.e. in~$B_{\frac{1}{2}}$. Scaling back to the original solution~$u$ eventually yields the claimed estimate~\eqref{lowerboundest}, which follows quickly from the fact~\eqref{minkowskialternativ} that implies
$$1+\delta+\frac{\mu}{4} \leq \partial_{\Tilde{e}^*} u \leq |Du|_E \qquad\mbox{a.e. in~$B_{\frac{\rho}{2}}(x_0)$}.$$ 
This finishes the proof. 
\end{proof}


The lower bound~\eqref{lowerboundest} from the preceding Proposition~\ref{lowerboundprop} implies that for any~$i=1,\ldots,n$ the partial derivative~$D_i u_\epsilon$ is a weak solution to a linear elliptic equation with bounded coefficients, whose ellipticity constant only depends on the parameter~$\delta\in(0,1]$ and is independent of~$\epsilon$. This expedient property allows us to obtain a quantitative excess-decay estimate for the excess-functional~\eqref{excess}, constituting the major component for the proof of Proposition~\ref{nondegenerateproposition}. Hereby, we rely on the following quantitative local~$L^\infty$-estimate and quantitative oscillation estimate.


\begin{myproposition} \label{lem:gilbargtrudinger}
    Let~$\delta,\epsilon\in(0,1]$ and~$u_\epsilon$ denote the weak solution to~\eqref{approx}. Then, there exists a constant~$C\geq 1$ and a parameter~$\alpha\in(0,1)$, both depending on
    \begin{align*}
        C&=C(C_\F(\delta),\hat{C}_\F,\delta,\|f\|_{L^{n+\sigma}(B_R)},M,N,n,r_E,\sigma), \\
        \alpha&= \alpha(C_\F(\delta),\hat{C}_\F,\delta,\|f\|_{L^{n+\sigma}(B_R)},M,N,n,r_E,\sigma),
    \end{align*}
    such that for any~$i\in\{1,\ldots,n\}$ there holds the quantitative local~$L^\infty$-estimate
    \begin{align} \label{est:gilbargtrudingersup}
        \esssup\limits_{B_{\frac{\rho}{2}}(x_0)} |D_i u_\epsilon-(D_i u_\epsilon)_{x_0,\rho}| \leq C \bigg(\fint_{B_{\rho}(x_0)}|D_i u_\epsilon-(D_i u_\epsilon)_{x_0,\rho}|^2\,\dx + (\rho^\beta+\rho^{2\beta}) \bigg).
    \end{align}
    Moreover, for any~$\theta\in(0,\frac{1}{2}]$ there holds the quantitative oscillation estimate
    \begin{align} \label{est:gilbargtrudingerosc}
      \essosc\limits_{B_{\theta\rho}(x_0)} \big[D_i u_\epsilon-(D_i u_\epsilon)_{x_0,\rho} \big] \leq C \bigg( \theta^{\alpha} \esssup\limits_{B_{\frac{\rho}{2}}(x_0)} |D_i u_\epsilon-(D_i u_\epsilon)_{x_0,\rho}| +\theta^{\alpha}\rho^{\alpha}(\rho^\beta+\rho^{2\beta})  \bigg).
    \end{align}
\end{myproposition}
  
\begin{proof}
    We differentiate the weak form of~$u_\epsilon$ by testing the latter with~$D_i\phi$, where~$\phi\in C^\infty_0(B_{\frac{\rho}{2}}(x_0))$ denotes an arbitrary smooth test function, and perform an integration by parts in the diffusion term. Indeed, we are allowed to differentiate the diffusion term due to Proposition~\ref{lowerboundprop}, which verifies that~$B_{\frac{\rho}{2}}(x_0))\ni x \mapsto \nabla\hat{\F}_\epsilon(x,Du_\epsilon(x))$ is differentiable as there holds~$|Du_\epsilon|_E\geq 1+\delta+\frac{\mu}{4}$ on this set of points. Thus, we obtain that for any~$i=1,\ldots,n$, the function~$v\coloneqq D_i u_\epsilon$ is a weak solution to
\begin{equation} \label{sollinell}
      \int_{B_{\frac{\rho}{2}}(x_0)} \hat{\B}_\epsilon(x,Du_\epsilon)(D v,D\phi)\,\dx =  \int_{B_{\frac{\rho}{2}}(x_0)} \bigg( f D_i\phi - \sum\limits_{j=1}^n g^i_j D_j \phi \bigg) \,\dx,
\end{equation}
 where the bounded coefficients~$g^i \colon B_{\frac{\rho}{2}}(x_0) \times \R^n \to \R^n$ are given by
$$|g^i(x)|\coloneqq |\partial_{x_i} \nabla \hat{\F}_\epsilon(x,Du_\epsilon)| \leq C(\hat{C}_F,N,r_E),
$$
with the bound being an immediate consequence of the Lipschitz estimate~\eqref{hatflipschitz}. Due to Proposition~\ref{lowerboundprop}, the bilinear form~$\hat{\mathcal{B}}_\epsilon(x,Du_\epsilon)$ is elliptic and bounded on the whole of~$\R^n$ for any~$x\in B_{\frac{\rho}{2}}(x_0)$, with an ellipticity constant that is independent of the parameter~$\epsilon\in(0,1]$, i.e. there exists a positive constant~$C=C(C_\F(\delta),\delta)\geq 1$ with
\begin{align} \label{ellipticityohneeps}
  C^{-1}|\xi|^2 \leq \hat{\mathcal{B}}_\epsilon(x,Du_\epsilon)(\xi,\xi) \leq C|\xi|^2  
\end{align}
for any~$x\in B_{\frac{\rho}{2}}(x_0)$ and any~$\xi\in\R^n$. The preceding bound readily follows from Lemma~\ref{bilinearelliptic} and also Proposition~\ref{lowerboundprop}, implying that there holds~$|Du_\epsilon|_E \geq 1+\frac{5}{4}\delta$ a.e. in~$B_{\frac{\rho}{2}}(x_0)$ due to~\eqref{deltamu}. Moreover, we can choose~$C$ to depend on~$C_\F(\delta)$ rather than~$C_\F$, according to~\eqref{hatfvoraussetzung}, due to the set inclusion
$$ \{ \xi\in\R^n:|\xi|_E\geq \mbox{$\frac{5}{4}$}\delta\} \subset \R^n\setminus E_\delta. $$
Put differently,~$v=D_i u_\epsilon$ is a weak solution to a linear equation with elliptic and bounded coefficients. We note that~$D_i u_\epsilon - (D_i u_\epsilon)_{x_0,\rho}$ is still a weak solution to~\eqref{sollinell}. Both claimed estimates~\eqref{est:gilbargtrudingersup} -- \eqref{est:gilbargtrudingerosc} are now direct consequences of the classical results found in~\cite[Chapter~8.6,~Theorem~8.17]{gilbargtrudinger} and also~\cite[Chapter~8.9,~Theorem~8.22]{gilbargtrudinger}.
\end{proof}


Proposition~\ref{lem:gilbargtrudinger} allows us to the deduce a quantitative excess-decay estimate for the excess-functional~\eqref{excess}, stated in the subsequent lemma.


\begin{mylem} \label{excessdecay}
Let~$\delta,\epsilon\in(0,1]$ and~$u_\epsilon$ denote the weak solution to~\eqref{approx}. Then, there exists a constant~$\Tilde{C}=\Tilde{C}\geq 1$ and an exponent~$\Tilde{\alpha}=\Tilde{\alpha}\in(0,1)$, both depending on the data
\begin{align*}
    \Tilde{C}&=\Tilde{C}(C_\F(\delta),\hat{C}_\F,\delta,\|f\|_{L^{n+\sigma}(B_R)},M,N,n,r_E,\sigma), \\
    \Tilde{\alpha}&=\Tilde{\alpha}(C_\F(\delta),\hat{C}_\F,\delta,\|f\|_{L^{n+\sigma}(B_R)},M,N,n,r_E,\sigma),
\end{align*}
such that
\begin{equation} \label{excessdecayest}
   \Phi(x_0,\theta\rho)\leq \Tilde{C} \theta^{2\Tilde{\alpha}} \big(\Phi(x_0,\rho)+\rho^{\Tilde{\alpha}}\mu^2\big)
\end{equation}
holds true for any~$\theta\in(0,1]$, provided there holds~$\rho\in(0,1]$.
\end{mylem}


\begin{proof}

We note that in the case where~$\theta\in(\frac{1}{2},1]$, the claimed excess-decay estimate~\eqref{excessdecayest} readily follows by an application of the~$L^2$-minimality of the mean value
\begin{equation*}
    \Phi(x_0,\theta\rho) \leq C(n)\Phi(x_0,\rho) \leq 2^{2\Tilde{\alpha}}C(n)\theta^{2\Tilde{\alpha}} \Phi(x_0,\rho) \leq C(n)\Phi(x_0,\rho)
\end{equation*}
for any~$\Tilde{\alpha}\in(0,1)$. Therefore, let us treat the case where~$\theta\in(0,\frac{1}{2}]$. Consequently, we are is position to apply both estimates~\eqref{est:gilbargtrudingerosc} as well as~\eqref{est:gilbargtrudingersup} from the previous Proposition~\ref{lem:gilbargtrudinger}, and use the assumption~$\rho\in(0,1]$ and~\eqref{schrankezwei}, to obtain
\begin{align*}
    \fint_{B_{\theta\rho}(x_0)}|D_i & u_\epsilon - (D_i u_\epsilon)_{x_0,\theta\rho}|^2\,\dx \\
    &\leq \bigg(\essosc\limits_{B_{\theta\rho}(x_0)} D_i u_\epsilon \bigg)^2 \\
    &= \bigg(\essosc\limits_{B_{\theta\rho}(x_0)} [D_i u_\epsilon - (D_i u_\epsilon)_{x_0,\rho}] \bigg)^2 \\
    &\leq C \Bigg(\theta^{2\alpha} \esssup\limits_{B_{\frac{\rho}{2}}(x_0)} |D_i u_\epsilon - (D_i u_\epsilon)_{x_0,\rho}|^2
 +\theta^{2\alpha}\rho^{2\alpha}\rho^{2\beta}\Bigg) \\
    &\leq C \Bigg(\theta^{2\alpha} \fint_{B_{\rho}(x_0)}|D_i u_\epsilon - (D_i u_\epsilon)_{x_0,\rho}|^4\,\dx +  \theta^{2\alpha} \big( \rho^{2\alpha} +  \rho^\beta +  \rho^{2\beta} \big) \Bigg) \\
    &\leq \Tilde{C} \Bigg(\theta^{2\Tilde{\alpha}} \fint_{B_{\rho}(x_0)}|D_i u_\epsilon - (D_i u_\epsilon)_{x_0,\rho}|^2\,\dx +  \theta^{2\Tilde{\alpha}}\rho^{\Tilde{\alpha}}\mu^2\Bigg)
\end{align*}
with a constant $\Tilde{C}=\Tilde{C}(C_\F(\delta),\hat{C}_\F,\delta,\|f\|_{L^{n+\sigma}(B_R)},M,N,n,r_E,\sigma)\geq 1$ and an exponent $\Tilde{\alpha}\coloneqq \min\{\alpha,\beta\}\in(0,1)$, where the parameter $\alpha\in(0,1)$, which depends on the same data, denotes the exponent arising from the application of~\eqref{est:gilbargtrudingerosc} from Lemma~\ref{lem:gilbargtrudinger} and~$\beta\in(0,1)$ is given in~\eqref{beta}. Since $i\in\{1,\ldots,n\}$ was chosen arbitrarily, we sum up with respect to $i=1,\ldots,n$ and eventually achieve the claimed estimate~\eqref{excessdecayest}.    
\end{proof}


At this point, all ingredients are in place to treat the proof of Proposition~\ref{nondegenerateproposition}.


\begin{proof}[\textbf{\upshape Proof of Proposition~\ref{nondegenerateproposition}}]

Let~$\Tilde{\alpha}\in(0,1)$ and~$\Tilde{C}$ denote the exponent and constant from the preceding Lemma~\ref{excessdecay}. We choose the parameter 
$$\theta\coloneqq \min \bigg\{\frac{4 r^2_E M^2+\delta^2}{\Tilde{C}\delta^2},(2\Tilde{C})^{-\frac{1}{2(\Tilde{\alpha}-\alpha)}},2^{-\frac{1}{\alpha}} \bigg\},$$
where~$\alpha\in(0,\Tilde{\alpha})$ denotes an arbitrary exponent. Further, we define the radius
$$\Tilde{\rho}\coloneqq \min\big\{\theta^{2\frac{\Tilde{\alpha}-\alpha}{\Tilde{\alpha}}},1 \big\},$$
such that~$\theta$ and~$\Tilde{\rho}$ both depend on the data 
$$(C_\F(\delta),\hat{C}_\F,\delta,\|f\|_{L^{n+\sigma}(B_R)},M,N,n,r_E,\sigma).$$
In what follows, let us consider a radius~$\rho\in(0,\Tilde{\rho}]$ and a ball~$B_{2\rho}(x_0)\Subset B_R$. We note that the general bound~\eqref{schrankeeins} yields an estimate for the excess-functional
\begin{align} \label{decayabsch}
    \Phi(x_0,\rho) \leq 4 \Tilde{M}^2 = 4 r^2_E M^2.
\end{align}
We will now show inductively that for any~$i\in\N$ there holds
\begin{equation} \label{excessiterated}
    \Phi(x_0,\theta^i \rho) \leq \theta^{2\alpha i}\mu^2,
\end{equation}
where we refer to~$\eqref{excessiterated}_i$ in the~$i$-th step of the iteration. We commence by treating the case~$i=1$. Our choices of~$\theta$ and~$\Tilde{\rho}$, together with an application of Lemma~\ref{excessdecay} and exploiting the bound~\eqref{decayabsch}, yield the estimate
\begin{align*}
    \Phi(x_0,\theta\rho) \leq \Tilde{C}\theta^{2\Tilde{\alpha}} \big(\Phi(x_0,\rho)+\rho^{\Tilde{\alpha}}\mu^2\big) \leq \Tilde{C}\theta^{2\Tilde{\alpha}} \bigg(\frac{4r^2_EM^2}{\delta^2}+1\bigg) \mu^2 \leq \theta^{2\alpha}\mu^2,
\end{align*}
which establishes that~$\eqref{excessiterated}_1$ holds true. Next, let us consider the case where~$i>1$ and assume that~$\eqref{excessiterated}_{i-1}$ is satisfied. Additionally, we apply Lemma~\ref{excessdecay} and utilize the choices of~$\theta$ and~$\Tilde{\rho}$, to obtain
\begin{align*}
    \Phi(x_0,\theta^{i}\rho) &\leq \Tilde{C}\theta^{2\Tilde{\alpha}} \big(\Phi(x_0,\theta^{i-1}\rho)+\theta^{2\Tilde{\alpha}(i-1)}\rho^{\Tilde{\alpha}}\mu^2\big) \\
    &\leq \Tilde{C}\theta^{2\Tilde{\alpha}} \big(\theta^{2\alpha(i-1)}\mu^2+\theta^{2\alpha (i-1)} \mu^2\big) \\
    & = 2\Tilde{C} \theta^{2(\Tilde{\alpha}-\alpha)} \theta^{2\alpha i} \mu^2 \\
    & \leq \theta^{2\alpha i} \mu^2.
\end{align*}
This verifies that also~$\eqref{excessiterated}_i$ is satisfied. We continue by treating both assertions~\eqref{lebesguerepresentant} and~\eqref{excessgdelta}. From Lemma~\ref{gdeltalem} and~$\eqref{excessiterated}_{i}$, we achieve for any~$i\geq 2$ the following
\begin{align} \label{intmedstep}
    \Tilde{\Phi}(x_0,\theta^{i-1} \rho) &\coloneqq \fint_{B_{\theta^{i-1} \rho}(x_0)} |\G_{2\delta}(Du)-(\G_{2\delta}(Du))_{x_0,\theta^{i-1}}\rho|^2\,\dx \\
    &\leq  \fint_{B_{\theta^{i-1} \rho}(x_0)} |\G_{2\delta}(Du)-\G_{2\delta}((Du)_{x_0,\theta^{i-1}\rho})|^2\,\dx \nonumber \\
    &\leq C \Phi(x_0,\theta^{i-1}\rho) \leq C \theta^{2\alpha(i-1)}\mu^2 \nonumber
\end{align}
with~$C=C(R_E,r_E)$. This estimate further yields
\begin{align*}
    |(\G_{2\delta}(Du))_{x_0,\theta^{i-1}\rho}   - (\G_{2\delta}(Du))_{x_0,\theta^i\rho}|^2 
    &\leq \fint_{B_{\theta^i \rho}(x_0)} |\G_{2\delta}(Du)-(\G_{2\delta}(Du))_{x_0,\theta^{i-1}\rho}|^2\,\dx  \\ 
    &\leq \frac{1}{\theta^n} \Tilde{\Phi}(x_0,\theta^{i-1} \rho) \leq \frac{C}{\theta^n} \theta^{2\alpha(i-1)}\mu^2. 
\end{align*}
Let us now consider natural numbers~$l<m$ and take roots in the preceding estimate. Due to the fact that~$\theta^\alpha \leq \frac{1}{2}$, we thus obtain
\begin{align*}
    |(\G_{2\delta}(Du))_{x_0,\theta^{l}\rho} - (\G_{2\delta}(Du))_{x_0,\theta^m\rho}|
    &\leq \sum\limits_{i=l+1}^{m} |(\G_{2\delta}(Du))_{x_0,\theta^{i-1}\rho}  -(\G_{2\delta}(Du))_{x_0,\theta^i\rho}| \\
    &\leq C\theta^{-\frac{n}{2}} \sum\limits_{i=l+1}^{m} \theta^{\alpha (i-1)} \mu \leq C\theta^{-\frac{n}{2}} \frac{\theta^{\alpha l}}{1-\theta^\alpha} \mu \leq C \theta^{-\frac{n}{2}} \theta^{\alpha l}\mu
\end{align*}
with a constant~$C=C(R_E,r_E)$. Thus, we have established that~$((\G_{2\delta}(Du))_{x_0,\theta^{i}\rho} )_{i\in\N}$ is a Cauchy sequence in~$\R^n$. Let us denote its limit by
$$\Gamma_{x_0} \coloneqq \lim\limits_{i\to\infty} (\G_{2\delta}(Du))_{x_0,\theta^{i}\rho}.$$
Next, we pass to the limit~$m\to\infty$ in the preceding estimate to derive
\begin{align*}
    |(\G_{2\delta}(Du))_{x_0,\theta^{l}\rho}-\Gamma_{x_0}| \leq C \theta^{-\frac{n}{2}} \theta^{\alpha l}\mu \qquad \mbox{for any~$l\in\N$.}
\end{align*}
Together with estimate~\eqref{intmedstep}, this yields
\begin{align*}
    \fint_{B_{\theta^l \rho}(x_0)} |\G_{2\delta}(Du)-\Gamma_{x_0}|^2\,\dx &\leq 2\Tilde{\Phi}(x_0,\theta^{l}\rho) + 2 |(\G_{2\delta}(Du))_{x_0,\theta^{l}\rho}-\Gamma_{x_0}|^2 \\
    &\leq C \theta^{2\alpha l}\mu^2 + C \theta^{-n} \theta^{\alpha 2l}\mu^2 \\
    &\leq C \theta^{-n} \theta^{2\alpha l}\mu^2.
\end{align*}
As a last step, this estimate is converted into the excess-decay~\eqref{excessgdelta}. Let~$r\in(0,\rho]$ and choose~$l\in\N$, such that~$\theta^{l+1}\rho <r \leq \theta^{l}\rho$ holds true. Due to our choice of~$\theta$ and the preceding inequality, we obtain 
\begin{align*}
    \fint_{B_r(x_0)} |\G_{2\delta}(Du) - \Gamma_{x_0}|^2\,\dx &\leq \frac{1}{\theta^n} \fint_{B_{\theta^{l}\rho}(x_0)} |\G_{2\delta}(Du) - \Gamma_{x_0}|^2\,\dx \\
    &\leq C\theta^{-2n} \theta^{2\alpha l}\mu^2 \\
    &\leq \Tilde{C}\Big(\frac{r}{\rho} \Big)^{2\alpha}\mu^2
\end{align*}
with~$\Tilde{C}=\Tilde{C}(C_\F(\delta),\hat{C}_\F,\delta,\|f\|_{L^{n+\sigma}(B_R)},M,N,n,R_E,r_E,\sigma)$. Finally, applying Jensen's inequality together with the above estimate, we achieve
\begin{align*}
    |(\G_{2\delta}(Du))_{x_0,r} - \Gamma_{x_0}|^2 \leq \fint_{B_r(x_0)} |\G_{2\delta}(Du) - \Gamma_{x_0}|^2\,\dx\dt \leq C\Big(\frac{r}{\rho} \Big)^{2\alpha}\mu^2,
\end{align*}
which establishes that
$$\Gamma_{x_0} = \lim\limits_{r\downarrow 0} (\G_{2\delta}(Du))_{x_0,r}.$$
The bound~\eqref{lebesguebound} is a consequence of~\eqref{betragminkowski} and estimate~\eqref{schrankeeins}, which yield~$|(\G_{2\delta}(Du))_{x_0,r}| \leq R_E\mu$ for any~$r\in(0,\rho]$. This finishes the proof of the proposition and also the non-degenerate regime.
\end{proof}


\section{The degenerate regime} \label{sec:degenerate}
In this section, we aim to demonstrate the proof of Proposition~\ref{degenerateproposition}, which outlines our primary result in the degenerate regime. To facilitate this proof, we will examine the function defined by 
$$v_\epsilon \coloneqq (\partial_{e^*} u_\epsilon -(1+\delta))^2_+,$$
where~$\delta,\epsilon\in(0,1]$ and~$e^*\in\partial E^*$. As we will establish throughout this section,~$ v_\epsilon$ turns out to be a weak sub-solution to a linear elliptic equation, as indicated in the format of equation \eqref{subsolungl}. By applying the De Giorgi-type estimate from Lemma~\ref{degiorgilem}, which will serve as a key tool in our argument, we are able to derive a quantitative reduction of the supremum of~$v_\epsilon$ within a smaller ball.\,\\

We next aim to establish the fact that for any~$e^*\in\partial E^*$, the mapping
$$v_\epsilon = (\partial_{e^*} u_\epsilon -(1+\delta))^2_+,$$
is a weak sub-solution to a linear elliptic equation. For this matter, we will differentiate the weak form~\eqref{weakformapprox} with respect to~$x_{e^*}$ by testing the weak form with~$\partial_{e^*}\phi$, where~$\phi\in C^\infty_0(B_R)$ denotes an arbitrary smooth test function, as already performed similarly throughout the proof of Proposition~\ref{lowerboundprop}. Recall that this procedure is admissible since second order weak derivatives of~$u_\epsilon$ exist and are locally square integrable according to Proposition~\ref{approxregularityeins}, while~$u_\epsilon$ is additionally of class~$W^{1,\infty}_{\loc}(B_R)$. We start with the subsequent lemma.


\begin{mylem} \label{mainintegralin}
    Let~$\delta,\epsilon\in(0,1]$,~$e^*\in\partial E^*$, and~$u_\epsilon$ denote the weak solution to~\eqref{approx}. Then, there holds the energy estimate
    \begin{align} \label{eeabsorbed}
      &(\epsilon +\Tilde{C}(\delta)) \int_{B_R} \zeta |\partial_{e^*} Du_\epsilon|^2 \bigchi_{\{\partial_{e^*} u_\epsilon>1+\delta\}}  \,\dx + \int_{B_R} \B_\epsilon(x,Du_\epsilon)(\partial_{e^*}Du_\epsilon,D\zeta) (\partial_{e^*}u_\epsilon-(1+\delta))_+ \,\dx \\
    &\leq  C \int_{B_R} |D\zeta| (\partial_{e^*}u_\epsilon-(1+\delta))_+ \big(1+|f| \big) \,\dx + C \int_{B_R} \zeta \bigchi_{\{\partial_{e^*}u_\epsilon>1+\delta\}} \big( 1+|f|^2 \big) \,\dx \nonumber
\end{align}
for any non-negative~$\zeta\in C^1_0(B_R,[0,1])$, with constant~$C=C(C_\F(\delta),\hat{C}_\F,\delta,N,n,R_E)$.
\end{mylem}


\begin{proof}
For convenience, we will denote~$u$ in place of~$u_\epsilon$ and omit the dependence on the parameter~$\epsilon \in (0,1]$. Let~$e^*\in \partial E^*$ be taken arbitrarily. By testing the weak form~\eqref{weakformapprox} with the function~$\partial_{e^*}\phi$, where~$\phi\in C^\infty_0(B_R)$ denotes an arbitrary smooth cut-off function, and subsequently integrating by parts, we obtain
\begin{equation} \label{diffleftside}
    \int_{B_R}  \langle \partial_{e^*}[ \hat{\h}_\epsilon(x,Du) ], D\phi \rangle \,\dx = \int_{B_R} f \partial_{e^*}\phi\,\dx
\end{equation}
for any~$\phi\in C^{\infty}_0(B_R)$. Now, we choose the test function
$$\phi_{e^*} = \zeta ( \partial_{e^*}u-(1+\delta))_+,$$
 where~$\zeta\in C^1_{0}(B_R,[0,1])$ denotes a non-negative, smooth cut-off function. Indeed, this choice of test function can be justified via an approximation argument. In order to abbreviate notation, we simply write~$a\coloneqq 1+\delta$, for some arbitrary but fixed parameter~$\delta\in(0,1]$. 
A direct calculation verifies
\begin{equation*}
    \partial_j\phi_{e^*} = \partial_j\zeta  (\partial_{e^*}u-a)_+ + \zeta \partial_{j}\partial_{e^*} u  \bigchi_{\{\partial_{e^*}u>a\}} \qquad\mbox{a.e. in~$B_R$}
\end{equation*}
for any~$j\in\{1,\ldots,n\}$. The left-hand side in~\eqref{diffleftside} yields the following two integral terms
\begin{align*}
     \int_{B_R} & \langle \partial_{e^*}[ \hat{\h}_\epsilon(x,Du) ], D\phi_{e^*} \rangle \,\dx \\
     &=  \int_{B_R}  \hat{\B}_\epsilon(x,Du)(\partial_{e^*}Du,D\phi_{e^*}) \,\dx +  \int_{B_R}  \langle \partial_{e^*} \hat{\h}_\epsilon(x,Du), D\phi_{e^*} \rangle \,\dx \\
     &\eqqcolon \foo{I} + \foo{II}.
\end{align*}
We begin with the treatment of the first term~$\foo{I}$, which gives
\begin{align*}
    \foo{I} &= \int_{B_R} \hat{\B}_\epsilon(x,Du)(\partial_{e^*}Du,D\zeta)  (\partial_{e^*}u-a)_+ \,\dx \\
    &\quad + \int_{B_R} \zeta \hat{\B}_\epsilon(x,Du)(\partial_{e^*}Du,\partial_{e^*}Du)   \bigchi_{\{\partial_{e^*}u>a\}} \,\dx.
\end{align*}
For the second term, we wish to exploit Lemma~\ref{bilinearelliptic} in order to further bound the quantity~$\foo{I}$ below. Indeed, the lemma is applicable, due to the set inclusion
$$ \{ x\in B_R: \partial_{e^*}u \geq 1+\delta \} \subset \{ x\in B_R: |Du|_E \geq 1+\delta \}, $$
which is an immediate consequence of the display~\eqref{minkowskialternativ}, and also~\eqref{hatfvoraussetzung}. Thus, we further estimate the term~$\foo{I}$ from below by
\begin{align} \label{diffeinssummed}
    \foo{I} &\geq \int_{B_R} \hat{\B}_\epsilon(x,Du)(\partial_{e^*}Du,D\zeta)  (\partial_{e^*}u-a)_+ \,\dx \\
    &\quad + (\epsilon + \Tilde{C}(C_\F(\delta),\delta)) \int_{B_R} \zeta |\partial_{e^*}Du|^2 \bigchi_{\{\partial_{e^*}u>a\}} \,\dx \nonumber,
\end{align}
where~$\Tilde{C}(C_\F(\delta),\delta)$ denotes some positive constant only depending on~$\delta\in(0,1]$ that arises from~\eqref{hatfvoraussetzung}. In turn, we also exploited the definition of~$C_\F$ in~\eqref{schrankehessian} and~$C_\F(\delta)$ in~\eqref{hesseschrankedelta}. As a next step, we treat the second integral term~$\foo{II}$. Here, we employ the Cauchy-Schwarz inequality and exploit the Lipschitz assumption~\eqref{hatflipschitz} in combination with~\eqref{representation}, just as in the proof of Proposition~\ref{lowerboundprop}. This way, we obtain
\begin{align} \label{est:lstattm}
    \foo{II} &\leq C \int_{B_R} |D\phi_{e^*}| \,\dx \\
    &\leq C \int_{B_R} |D\zeta| (\partial_{e^*}u-a)_+ \,\dx + C \int_{B_R} \zeta |\partial_{e^*} Du| \bigchi_{\{\partial_{e^*}u >a\}} \,\dx \nonumber
\end{align}
with constant~$C=C(\hat{C}_\F,N,n,r_E)$. Thus, employing Young's inequality, we achieve
\begin{align} \label{diffzweisummed}
    \foo{II} &\leq C \int_{B_R} |D\zeta| (\partial_{e^*}u-a)_+ \,\dx + \mbox{$\frac{1}{4}$} \Tilde{C}(C_\F(\delta),\delta) \int_{B_R} \zeta |\partial_{e^*}Du|^2 \bigchi_{\{\partial_{e^*}u >a\}} \,\dx \\
    &\quad + C \int_{B_R} \zeta \bigchi_{\{\partial_{e^*}u >a\}} \,\dx \nonumber
\end{align}
with a constant~$C=C(C_\F(\delta),\hat{C}_\F,\delta,N,n,r_E)$. Now, the right-hand side of~\eqref{diffleftside} involving the datum~$f$ is dealt with. Similarly to the treatment of the second integral term~$\foo{II}$, we first estimate
\begin{align*}
    \foo{III} &\coloneqq \int_{B_R} f \partial_{e^*}\phi_{e^*}\,\dx \\
                &\leq \int_{B_R} |f| |D\zeta| (\partial_{e^*}u-a)_+  \,\dx + \int_{B_R} \zeta |f| |\partial_{e^*}D u| \bigchi_{\{\partial_{e^*}u >a\}}  \,\dx.
\end{align*}
 By further employing Young's inequality once more, there holds
\begin{align} \label{diffdatumsummed}
    \foo{III} &\leq \int_{B_R} |f| |D\zeta| (\partial_{e^*}u-a)_+  \,\dx  + \mbox{$\frac{1}{4}$} \Tilde{C}(C_\F(\delta),\delta) \int_{B_R} \zeta  |\partial_{e^*}Du|^2 \bigchi_{\{\partial_{e^*}u >a\}}  \,\dx \\
          &\quad + C \int_{B_R} \zeta |f|^2 \bigchi_{\{\partial_{e^*}u >a\}} \,\dx \nonumber
\end{align}
with~$C=C(C_\F(\delta),\delta)$. Combining all our estimates~\eqref{diffeinssummed},~\eqref{diffzweisummed}, and also~\eqref{diffdatumsummed}, we end up with
\begin{align} \label{diffcombined}
    &(\epsilon +\Tilde{C}(C_\F(\delta),\delta)) \int_{B_R} \zeta |\partial_{e^*} Du|^2 \bigchi_{\{\partial_{e^*}u >a\}}  \,\dx  +  \int_{B_R} \B_\epsilon(x,Du)(\partial_{e^*}Du,D\zeta) (\partial_{e^*}u-a)_+ \,\dx \nonumber\\
    &\leq  \mbox{$\frac{1}{2}$} \Tilde{C}(C_\F(\delta),\delta) \int_{B_R} \zeta |\partial_{e^*} Du|^2 \bigchi_{\{\partial_{e^*}u >a\}} \,\dx + C \int_{B_R} |D\zeta| (\partial_{e^*}u-a)_+ \,\dx + C \int_{B_R} \zeta \bigchi_{\{\partial_{e^*}u>a\}} \dx  \nonumber \\
    &\quad   + C \int_{B_R} |f| |D\zeta| (\partial_{e^*}u-a)_+  \,\dx + C\int_{B_R} \zeta |f|^2 \bigchi_{\{\partial_{e^*}u>a\}}  \,\dx \nonumber 
\end{align}
with a constant~$C=C(C_\F(\delta),\hat{C}_\F,\delta,N,n,R_E)$. Consequently, after reabsorbing all quantities on the right-hand side of the preceding estimate into the left-hand side, there holds
\begin{align*}
    &(\epsilon +\Tilde{C}(\delta)) \int_{B_R} \zeta |\partial_{e^*} Du|^2 \bigchi_{\{\partial_{e^*} u>a\}}  \,\dx +   \int_{B_R} \B_\epsilon(x,Du)(\partial_{e^*}Du,D\zeta) (\partial_{e^*}u-a)_+ \,\dx \\
    &\leq  C \int_{B_R} |D\zeta| (\partial_{e^*}u-a)_+ \big(1+|f| \big) \,\dx + C \int_{B_R} \zeta \bigchi_{\{\partial_{e^*}u>a\}} \big( 1+|f|^2 \big) \,\dx \nonumber
\end{align*}
with~$C=C(C_\F(\delta),\hat{C}_\F,\delta,N,n,r_E)$. This concludes the proof of the lemma. 
\end{proof}


It follows directly from Lemma~\ref{mainintegralin} that the function
$$v_\epsilon = (\partial_{e^*}u_\epsilon - (1+\delta))^2_+,$$
where~$\delta,\epsilon \in (0, 1]$ and~$e^*\in\partial E^*$ remain arbitrary, is a sub-solution to an elliptic linear equation. This matter of fact is crucial for exploiting the measure-theoretic information~\eqref{degeneratemeascond}, characterizing the degenerate regime.


\begin{mylem} \label{subsolutionlem}
    Let~$\delta,\epsilon\in(0,1]$,~$e^*\in\partial E^*$, and~$u_\epsilon$ denote the weak solution to~\eqref{approx}. Assume that both conditions~\eqref{schrankeeins} and~\eqref{schrankezwei} hold true on~$B_{2\rho}(x_0)\Subset B_R$. Then, the function
    \begin{equation} \label{subsol}
        v_\epsilon = (\partial_{e^*}u - (1+\delta))^2_+
    \end{equation}
    is a weak sub-solution to an elliptic linear equation, meaning that
    \begin{align} \label{subsolungl}
        \int_{B_\rho(x_0)} &\hat{\B}_\epsilon(x,Du_\epsilon)(D v_\epsilon,D\zeta)\,\dx \\
        &\leq C \bigg( \int_{B_\rho(x_0)}
        |D\zeta| (1+|f|) \bigchi_{\{v_\epsilon>0\}} \,\dx  + \int_{B_\rho(x_0)}\zeta (1+|f|^2) \bigchi_{\{v_\epsilon>0\}} \,\dx \bigg) \nonumber
    \end{align}
    holds true for any non-negative test function~$\zeta\in C^1_0(B_\rho(x_0))$, with~$C=C(C_\F(\delta),\hat{C}_\F,\delta,N,n,r_E)$.
\end{mylem}


\begin{proof}
We apply Lemma~\ref{mainintegralin} and note that the domain of integration~$B_R$ in all integral quantities on the right-hand side may be reduced to the smaller ball~$B_{\rho}(x_0) \cap \{\partial_{e^*}u_\epsilon>1+\delta\}$. Since the first term on the left-hand side of~\eqref{eeabsorbed} yields a non-negative contribution, we discard it. Due to assumption~\eqref{schrankeeins}, we have the bounds~$(\partial_{e^*}u_\epsilon-(1+\delta))_+ \leq \mu \leq M$, allowing us to estimate the quantities involving the datum~$f$ on the right-hand side of inequality~\eqref{eeabsorbed} further above.
Finally, we note that there holds
$$Dv_\epsilon = 2(\partial_{e^*} u_\epsilon-(1+\delta))_+ \partial_{e^*} Du_\epsilon \qquad \mbox{a.e. in~$B_\rho(x_0)$},$$ which, together with the bilinearity of~$B_\epsilon$, implies
\begin{equation*}
    2\hat{\B}_\epsilon(x,Du_\epsilon)(\partial_{e^*}Du_\epsilon,D\zeta) (\partial_{e^*}u_\epsilon-(1+\delta))_+ = \hat{\B}_\epsilon(x,Du_\epsilon)(Dv_\epsilon,D\zeta).
\end{equation*}
In turn, this yields the claimed estimate~\eqref{subsolungl}.
\end{proof}


To follow up, we immediately obtain an energy estimate of De Giorgi class-type involving the solution~$v_\epsilon$ to the linear elliptic equation in the sense of~\eqref{subsolungl}.


\begin{mylem}[A De Giorgi class-type estimate for a sub-solution] \label{degiorgilem}
    Let~$\delta,\epsilon\in(0,1]$,~$e^*\in\partial E^*$, and~$u_\epsilon$ denote the weak solution to~\eqref{approx}. Further, let~$v_\epsilon$ denote the weak sub-solution defined in~\eqref{subsol}. Then, for any~$k>0$ and any~$\tau\in(0,1)$ there holds 
    \begin{align*}
    \int_{B_{\tau\rho}(x_0)}|D(v_\epsilon-k)_+|^2\,\dx &\leq \frac{C}{(1-\tau)^2\rho^2}\int_{B_\rho(x_0)}(v_\epsilon-k)^2_+\,\dx \\
    &\quad + C(\rho^{2(1-\beta)} + \|f\|^2_{L^{n+\sigma}(B_R)}) |A_k|^{1-\frac{2}{n}+\frac{2\beta}{n}}
    \end{align*}
    with~$C=C(C_\F(\delta),\hat{C}_\F,\delta,M,N,n,r_E,\sigma)$ and~$\beta=\frac{\sigma}{n+\sigma}$. Here, the level set~$A_k$ is given by
    $$A_k\coloneqq \{x\in B_\rho(x_0):v_\epsilon(x)>k\}.$$
\end{mylem}

\begin{proof}
Since the bilinear form~$\hat{\B}_\epsilon(x,Du_\epsilon)(Dv_\epsilon,D\zeta)$ vanishes on the set of points 
$$\{x\in B_\rho(x_0): \partial_{e^*} u_\epsilon \leq 1+\delta\},$$
we may redefine~$\B_\epsilon$ to the identity operator on this very subset of~$B_R$ . Respectively, the redefined bilinear form~$\Tilde{\B}_\epsilon$ is given by
\begin{equation*}
    \Tilde{\B}_\epsilon(x,Du_\epsilon(x)) \coloneqq \begin{cases}
        \Tilde{\foo{I}}_n, & \text{on~$\{x\in B_\rho(x_0): \partial_{e^*}u_\epsilon(x) \leq 1+\delta\}$} \\
        \hat{\B}_\epsilon(x,Du_\epsilon(x)), & \text{on~$\{x\in B_\rho(x_0): \partial_{e^*}u_\epsilon(x) > 1+\delta\}$},
    \end{cases}
\end{equation*}
where~$\Tilde{\foo{I}}_n(\xi,\zeta)\coloneqq \langle \foo{I}_n\xi, \zeta \rangle  =  \langle \xi, \zeta \rangle $ for any~$\xi,\zeta\in\R^n$. Moreover, we exploit the bounds~\eqref{schrankeeins} and~\eqref{schrankezwei}. This way,~$v_\epsilon$ is a weak sub-solution to an elliptic linear equation in the sense that
\begin{equation} \label{subsolunglneu}
     \int_{B_\rho(x_0)} \Tilde{\B}_\epsilon(x,Du_\epsilon)(D v_\epsilon,D\zeta)\,\dx \leq C \int_{B_\rho(x_0)}\big(|D\zeta|(1+|f|) + \zeta|f|^2 + 1\big) \bigchi_{\{v_\epsilon >0\}} \,\dx
\end{equation}
holds true for any non-negative~$\zeta\in C^1_0(B_\rho(x_0))$ with~$C=C(C_\F(\delta),\hat{C}_\F,\delta,M,N,n,r_E)$. Due to our construction, the operator~$\Tilde{\B}_\epsilon(x,Du_\epsilon)$ is elliptic in~$B_\rho(x_0)$, i.e. there exists a constant~$C=C(C_\F(\delta),\delta)\geq 1$, such that
\begin{equation} \label{unifell}
   C^{-1}|\xi|^2\leq \Tilde{B}_\epsilon(x,Du_\epsilon(x))(\xi,\xi) \leq C|\xi|^2
\end{equation}
holds true for any~$x\in B_\rho(x_0)$ and any~$\xi\in\R^n$. The preceding bound holds true with constant~$C=1$ in the set of points~$\{x\in B_\rho(x_0):  \partial_{e^*}u_\epsilon(x) \leq 1+\delta\}$, which readily follows from the very definition of~$\Tilde{B}_\epsilon$. On the complementary subset~$\{x\in B_\rho(x_0): \partial_{e^*}u_\epsilon(x) > 1+\delta\}$, we apply Lemma~\ref{bilinearelliptic} and recall that~$C_\F$ can be controlled from below by~$1$ and from above by~$C_\F(\delta)$ on the considered subset of points. We note once more that the set inclusion
$$ \{ x\in B_R: \partial_{e^*}u \geq 1+\delta \} \subset \{ x\in B_R: |Du|_E \geq 1+\delta \} $$
is exploited. For the treatment of the upper bound, we exploited the fact that~$\epsilon\in(0,1]$ and~\eqref{derivativebounds}. The claimed energy estimate now follows in a standard way by testing the identity~\eqref{subsolunglneu} with~$\zeta = \eta^2 (v_\epsilon-k)_+$, where the smooth cut-off function~$\eta\in C^1_0(B_\rho(x_0))$ is chosen is a way, such that~$\eta\equiv 1$ on~$B_{\tau\rho}(x_0)$ and~$|D\eta|\leq \frac{2}{(1-\tau)\rho}$, and by applying Young's inequality and Hölder's inequality with exponents~$(\frac{n+\sigma}{2},\frac{n+\sigma}{n+\sigma-2})$. We omit the details and instead refer the reader to~\cite[Chapter~10.1]{dibenedetto2023parabolic} and also to~\cite[Chapter~12.1]{dibenedetto2023parabolic},~\cite[Proposition~7.1]{degeneratesystems} in the parabolic setting. 
\end{proof}


Let~$\delta,\epsilon\in(0,1]$,~$e^*\in\partial E^*$, and~$u_\epsilon$ be the unique weak solution to~\eqref{approx}. Assume that~\eqref{degeneratemeascond} holds true for some parameter~$\nu>0$ on a ball~$B_\rho(x_0)\subset B_{2\rho}(x_0)\Subset B_R$. We note that condition~\eqref{schrankeeins} implies the upper bound
\begin{equation} \label{vepsbound}
    \esssup\limits_{B_{2\rho}(x_0)}v_\epsilon \leq \mu^2.
\end{equation}
As Lemma~\ref{degiorgilem} is already at our disposal, we immediately obtain the subsequent two lemmas of De Giorgi-type, where we refrain from stating the proof due to similarity and rather refer the reader to~\cite[Lemma~6.1 -- Lemma~6.2]{bogelein2023higher} for a detailed elaboration.


\begin{mylem} \label{degiorgilemeins}
    For any $\theta\in(0,1)$, there exists a parameter $\Tilde{\nu}\in(0,1)$, which depends on the given data $$(C_\F(\delta),\hat{C}_\F,\|f\|_{L^{n+\sigma}(B_R)},\delta,M,n,R,r_E,\sigma),$$
    such that: if the measure-theoretic information
\begin{equation*}
    |\{x\in B_{\rho}(x_0):v_\epsilon(x)>(1-\theta)\mu^2\}|<\Tilde{\nu}|B_\rho(x_0)|
\end{equation*}
    applies for any~$e^*\in\partial E^*$, then we either have
\begin{equation*}
    \mu^2<\frac{\rho^\beta}{\theta}
\end{equation*}
    or there holds
\begin{equation*}
    v_\epsilon \leq \big(1- \textstyle{\frac{1}{2}} \theta\big)\mu^2\quad\text{a.e. in~$B_{\frac{\rho}{2}}(x_0)$.}
\end{equation*}
\end{mylem}


At this juncture, it is crucial to assume that condition~\eqref{degeneratemeascond} holds true for all~$ e^* \in \partial E^* $. This assumption allows the derivation of Lemma~\ref{degiorgilemzwei}, which serves as our counterpart to \cite[Lemma~6.2]{bogelein2023higher}, demonstrated in the appendix of the mentioned article.


\begin{mylem} \label{degiorgilemzwei}
    Assume that there exists some~$\nu\in(0,1)$ for which the measure-theoretic information~$\eqref{degeneratemeascond}$ holds true for any~$e^*\in \partial E^*$. Then, for any~$i_*\in\N$ we either have
    \begin{equation*}
    \mu^2<2^{i_*}\frac{\rho^\beta}{\nu}
\end{equation*}
    or
\begin{equation*}
    |\{x\in B_{\rho}(x_0):v_\epsilon(x)>(1-2^{-i_*}\nu)\mu^2\}|<\frac{C}{\nu\sqrt{i_*}}|B_\rho(x_0)|
\end{equation*}
for any~$e^*\in \partial E^*$, with~$C=C(C_\F(\delta),\hat{C}_\F,\delta,\|f\|_{L^{n+\sigma}(B_R)},M,N,n,R,r_E,\sigma) \geq 1$.
\end{mylem}


With the aid of both Lemma~\ref{degiorgilemeins} and Lemma~\ref{degiorgilemzwei} we are finally in position to provide the proof of Proposition~\ref{degenerateproposition}. 


\begin{proof}[\textbf{\upshape Proof of Proposition~\ref{degenerateproposition}}]
Let~$\Tilde{\nu}\in(0,1)$ denote the parameter from Lemma~\ref{degiorgilemeins} and $C\geq 1$ the constant from Lemma~\ref{degiorgilemzwei}. We choose an integer~$i_*\in\N$ large enough, such that
\begin{equation*}
    i_* \geq \Big(\frac{C}{\Tilde{\nu}\nu}\Big)^2
\end{equation*}
is satisfied. Consequently, due to the dependence of the parameter~$\Tilde{\nu}\in(0,1)$ and the constant~$C\geq 1$,~$i_*$ depends on the data
$$(C_\F(\delta),\hat{C}_\F,\delta,M,\|f\|_{L^{n+\sigma}(B_R)},N,n,R,r_E,\sigma).$$
An application of Lemma~\ref{degiorgilemzwei} yields that either there holds
   \begin{equation*}
     \mu^2<\frac{2^{i_*}\rho^{\beta}}{\nu}
\end{equation*}
    or
\begin{equation} \label{laststep}
    |\{x\in B_{\rho}(x_0):v_\epsilon(x)>(1-2^{-i_*}\nu)\mu^2\}|<\frac{C}{\nu\sqrt{i_*}}|B_\rho(x_0)|
\end{equation}
for any~$e^*\in\partial E^*$. The former case can be excluded by choosing a radius~$\hat{\rho}\in(0,1)$ small enough, such that the relation 
 $$\hat{\rho}^\beta \leq \frac{\delta^2 \nu}{2^{i_*}}$$
 is satisfied, and perform all the above steps for an arbitrarily chosen radius~$\rho\leq \hat{\rho}\leq 1$. The radius~$\hat{\rho}$ thus depends on the data as stated in Proposition~\ref{degenerateproposition}. In particular, we rely on assumption~\eqref{deltamu} in order to do so. If the other case~\eqref{laststep} applies, we are in position to employ Lemma~\ref{degiorgilemeins} with the choice~$\theta=2^{-i_*}\nu$, which yields that, either we have
   \begin{equation*}
    \mu^2<\frac{2^{i_*}\rho^{\beta}}{\nu}
\end{equation*}
or
\begin{equation*}
    v_\epsilon \leq (1-2^{-(i_*+1)}\nu)\mu^2\quad\text{a.e. in~$B_{\frac{\rho}{2}}(x_0)$}
\end{equation*}
holds true for any~$e^*\in\partial E^*$. The former case is already excluded through the choice of~$\hat{\rho}\in(0,1)$ small enough in dependence of the data, as explained above. In the second case, we note that due to
$$v_\epsilon = ( \partial_{e^*} u_\epsilon-(1+\delta))^2_+,$$
we take roots in the preceding inequality to obtain that
$$ ( \partial_{e^*} u_\epsilon-(1+\delta))_+ \leq \kappa \mu $$
for any~$\kappa\geq (1-2^{-(i_*+1)}\nu)^{\frac{1}{2}}$. Thus, the parameter~$\kappa$ may be chosen appropriately, such that~$\kappa\in\big[2^{-\frac{\Tilde{\beta}}{2}},1\big)$ is satisfied, where~$\alpha\in(0,1)$ denotes the parameter from Proposition~\ref{nondegenerateproposition},~$\beta$ is given in~\eqref{beta}, and~$0<\Tilde{\beta}<\alpha<\beta$. Finally, since the preceding estimate holds for all~$e^*\in \partial E^*$, we eventually pass to the supremum over all~$e^*\in\partial E^*$ and exploit~\eqref{minkowskialternativ}. 
This yields the desired quantitative shrinking of supremum
$$ (|Du|_E-(1+\delta))_+ \leq \kappa\mu $$
in~$B_{\frac{\rho}{2}}(x_0)$ for any~$\kappa\in\big[2^{-\frac{\Tilde{\beta}}{2}},1\big)$ and finishes the proof of the proposition and the degenerate regime.
\end{proof}


\nocite{*}
\bibliographystyle{plain}
\bibliography{Literatur.bib}


\end{document}